\documentclass[msc,oneside]{ubcthesis}

\usepackage{afterpage}
\usepackage{float}
\usepackage{graphicx}
\usepackage{subfigure}
\usepackage{here}
\usepackage[numbers,sort&compress]{natbib}
\usepackage{longtable}
\usepackage{psfrag}
\usepackage[unicode=true,
  linktocpage,
  linkbordercolor={0.5 0.5 1},
  citebordercolor={0.5 1 0.5},
  linkcolor=blue]{hyperref}
\usepackage{amsfonts,amsmath,amssymb, amsthm}
\usepackage{xcolor}
\usepackage[T1]{fontenc}

\institution{The University Of British Columbia}
\faculty{The Faculty of Graduate and Postdoctoral Studies}
\institutionaddress{Vancouver}

\previousdegree{A.S., Fresno City College, 2020}
\previousdegree{B.Sc., Massachusetts Institute of Technology, 2024}

\title{Progress in Projection Theory}
\subtitle{and other dimensional developments}
\author{Paige Dote (Paige Bright)}
\copyrightyear{\number\year}
\submitdate{\monthname\ \number\year} 

\program{Mathematics}

\renewcommand\thepart         {\Roman{part}}
\renewcommand\thechapter      {\arabic{chapter}}
\renewcommand\thesection      {\thechapter.\arabic{section}}
\renewcommand\thesubsection   {\thesection.\arabic{subsection}}
\renewcommand\thesubsubsection{\thesubsection.\arabic{subsubsection}}
\renewcommand\theparagraph    {\thesubsubsection.\arabic{paragraph}}

\newtheorem{theorem}{Theorem}[chapter]
\newtheorem*{theorem*}{Theorem}
\newtheorem{question}[theorem]{Question}
\newtheorem{proposition}[theorem]{Proposition}
\newtheorem{corollary}[theorem]{Corollary}
\newtheorem{conjecture}[theorem]{Conjecture}
\newtheorem{heuristic}[]{Heuristic}
\newtheorem{lemma}[theorem]{Lemma}

\theoremstyle{definition}

\newtheorem{definition}[theorem]{Definition}
\newtheorem*{definition*}{Definition}
\newtheorem{example}[theorem]{Example}
\newtheorem{remark}[theorem]{Remark}
\newtheorem{notation}[theorem]{Notation}

\newcommand{\R}{\mathbb{R}}
\newcommand{\de}{\delta}
\renewcommand{\hat}{\widehat}
\newcommand\numberthis{\addtocounter{equation}{1}\tag{\theequation}}
\DeclareMathOperator{\supp}{supp}

\begin{document}

\frontmatter

\maketitle

\newpage

\noindent The following individuals certify that they have read, and recommend to the Faculty of Graduate and Postdoctoral Studies for acceptance, the thesis entitled:
\\

\noindent \textbf{Progress in projection theory and other dimensional developments}

\bigskip\noindent
submitted by \textbf{Paige Dote (Paige Bright)}
in partial fulfillment of the requirements for the degree of \textbf{Master of Science} in	\textbf{Mathematics}.

\bigskip\noindent
\textbf{Examining Committee:}

\bigskip\noindent
Izabella {\L}aba, Professor, Mathematics, UBC \\
\textit{Co-Supervisor}

\bigskip\noindent
Pablo Shmerkin, Professor, Mathematics, UBC\\
\textit{Co-Supervisor}

\bigskip\noindent
Joshua Zahl, Associate Professor, Mathematics, UBC \\
\textit{Supervisory Committee Member}

\begin{abstract} 
We provide exposition into the field of projection theory, which lies at the intersection of incidence geometry and geometric measure theory. We first give the necessary preliminaries in Chapter \ref{ch:background}, focusing on incidences between points and lines and the definition of Hausdorff dimension. With this background in tow, in Chapter \ref{ch:Topics} we dive into thorough surveys on three topics in projection theory: orthogonal projections, Furstenberg sets, and radial projections. We particularly highlight the interconnectedness of these topics in both the discrete and continuum settings. Through these surveys, we also give the necessary background for discussing applications of projection theory in Chapter \ref{ch:problems}. 

The first application is on Beck-type problems, first studied in the discrete setting by J\'ozsef Beck in 1983 and in the continuum setting by Orponen, Shmerkin, and Wang in 2022. 
Given $X\subset \R^n$, these problems seek to understand how large the set of lines that contain at least 2 points of $X$, $\mathcal L(X)$, can be. To this end, we present a continuum Erd\H{o}s--Beck theorem due to myself and Marshall in 2024, which motivates and makes use of a dual Furstenberg set estimate due to myself, Fu, and Ren from the same year.

The second application is on Falconer-type (distance) problems which have been a prominent topic in both the discrete and continuum settings. Given $X\subset \R^n$, these problems seek to understand how large $X$ must be until the set of distinct ``distances'' between points of $X$ is large (for a reasonable notion of ``distance''). To this end, we present a Falconer-type distance problem for dot products due to myself, Marshall, and Senger in 2024, making use of both standard and modern results for orthogonal and radial projections.
\end{abstract}

\chapter{Lay Summary} 

This thesis focuses on problems in geometry and more specifically projection theory. Problems in this area often grapple with how points lie on lines and planes, as we will see throughout the thesis.

In Chapter \ref{ch:background}, we give foundational definitions and background in the discrete and continuum settings we will assume familiarity with throughout the thesis. In Chapter \ref{ch:Topics}, we apply these tools towards giving a survey on a few topics in projection theory, focusing on key results and techniques. In Chapter \ref{ch:problems}, we present two applications of the topics presented in the previous chapter to problems in projection theory. The first is on Beck-type problems, and the second is on a Falconer-type problem. These problems give insight into how different topics in projection theory work together geometrically, and can even give us new insight into the original topics themselves. Elucidating this insight is the primary goal of this thesis.

\chapter{Preface} 

Chapter \ref{ch:background} provides a standard introduction to discrete geometry and geometric measure theory. Chapter \ref{ch:Topics} is largely expository, providing surveys on topics in projection theory. As part of these surveys, I present a few modified arguments from \textit{Exceptional Set Estimates for Radial Projections in $\mathbb{R}^n$} \cite{BrightGan} and \textit{Exceptional Set Estimates in Finite Fields} \cite{BrightGan2} over finite fields, both published in the journal \textit{Annales Fennici Mathematici}. This research was undertaken with Shengwen Gan in equal collaboration. The former, \cite{BrightGan}, was a problem suggested by Larry Guth. Gan and I met daily to collaboratively obtain our results, and co-wrote \cite{BrightGan} in 2022. The latter, \cite{BrightGan2}, was inspired by probing the methods in \cite{BrightGan} further the following semester. Together, Gan and I identified the problem, designed our approach, and co-wrote \cite{BrightGan2}.

Chapter \ref{ch:problems} presents two applications of projection theory. The first application, covered in Section \ref{SEC-BECKTYPE}, is on Beck-type theorems. This research was undertaken with Caleb Marshall in equal collaboration. At the University of Pennsylvania Study Guide Writing Workshop 2023, Marshall and I posed the question, developed our methodology for attacking it together, and co-wrote the paper over the following year. The work with Marshall \cite{BrightMarshall} also makes use of a Furstenberg set estimate presented in Section \ref{sec:DualFurstenberg}, which gives an introduction to the preprint \textit{Radial Projections in $\R^n$ Revisited} \cite{BrightFuRen}. The research that went into \cite{BrightFuRen} was undertaken jointly with Yuqiu Fu and Kevin Ren in equal collaboration. Fu, Ren, and I collaboratively identified the question as being approachable by generalizing a previous result of Fu and Ren \cite{FuRen} using methodology I studied at the 2023 workshop, and with this approach we resolved the problem and co-wrote the paper. Both \cite{BrightFuRen} and \cite{BrightMarshall} are intended for publication, though they have not yet been submitted. 

The second application, covered in Section \ref{SEC:FALCONERTYPE}, is on a Falconer-type problem for dot products. This section is adapted from the preprint \textit{Pinned Dot Product Set Estimates} \cite{BrightMarshallSenger}, under peer review after being submitted for publication. This research was undertaken jointly with Caleb Marshall and Steven Senger in equal collaboration. We identified the problem and approach at the Joint Mathematics Meeting 2024, and co-wrote \cite{BrightMarshallSenger} in 2024.

Generative AI was not used for any aspect of the work within this thesis.

\tableofcontents   


\chapter{Acknowledgments}

Over the past few years, whenever people would ask me what drew me to mathematics, and especially harmonic analysis, I've always say that it's the people. Having had the opportunity to attend MIT and the University of British Columbia, I now find myself in the fortunate and fulfilling position of having many people to acknowledge.

I'd like to first thank my high school math teachers who pushed me to pursue my passions early on, including Sandra Atkins, Lisa Portela, Travis McDonald, Heather Walker, Maria Nogin, and Adnan Sabuwala.

Next, I must thank Larry Guth for sparking my interest in harmonic analysis and projection theory in the first place. After my first year of my undergraduate studies at MIT, I was starting to get more interested in analysis and reached out to Larry for a reading project. He agreed, and ever since that summer I've been hooked. From this first reading project to the research opportunities since, and everything in between, I am deeply grateful for Larry's support throughout my time as an undergraduate. I look forward to getting to be a part of his research group as a PhD student this fall. To this end, I must also extend my appreciation to the members of Larry's Fourier analysis reading group, including but not limited to Alex Ortiz, Marjorie Drake, Yuqiu Fu, Alex Cohen, Dmitrii Zakharov, Jiahui Yu, Manik Dhar, Dominique Maldague, Sarah Tammen, Felipe Hern\'andez, and Shengwen Gan. This group constantly exemplified the type of graduate student I want to be, and I am excited to pay it forward. I especially thank Alex Ortiz and Marjorie Drake, with whom I've had many great discussions about mathematics and life (and the intersection thereof).

It's also worth mentioning that for the reading project, Larry suggested Alex Iosevich's \textit{A View from the Top} \cite{IosevichAView}, a book which introduces the world of harmonic analysis from the perspective of discrete geometry and combinatorics. The pedagogical thought that went into the writing of this book (and much of Iosevich's work) has proven deeply insightful, and inspired how I went about writing this thesis. As such, I must extend deep appreciation to Alex Iosevich and past students of his who I have had the pleasure to meet and work with over the years, especially Krystal Taylor and Steven Senger. 

At the University of British Columbia, I must thank my supervisors Izabella {\L}aba, Pablo Shmerkin, and Joshua Zahl. They gave me an immense amount of time and room to grow as a graduate student, mathematician, and person for which I am deeply grateful. I have a great deal more I hope to one day learn from them. I am similarly thankful for the opportunity to become a part of the larger harmonic analysis and fractal geometry community at UBC, including Caleb Marshall, Malabika Pramanik, Alexia Yavicoli, Angel Cruz, Yuveshen Mooroogen, Simone Maletto, Andrew Alexander, Mukul Choudhuri, Chenjian Wang, Junjie Zhu, Matthew Bull-Weizel, Yuhan Chu, and Sushrut Tadwalkar. I especially thank Yuve and Angel, people who have exemplified something I've believed for quite some time: caring about and putting time into education as a graduate student can make you both a greater mathematician and friend.

I also take this time to express my sincere gratitude and appreciation towards the organizers of the University of Pennsylvania Study Guide Workshop 2023: Joshua Zahl, Hong Wang, Yumeng Ou, and Philip Gressman. This workshop was deeply changing for me as a researcher and student. It was at this workshop that I got to collaborate with and get to know Ryan Bushling, Caleb Marshall, and Alex Ortiz---some of the first graduate students to make me feel more like a peer than a mentee. I especially thank Ryan for his thorough and kind comments on drafts of my thesis. Furthermore, at this workshop I was particularly inspired by Caleb and Joshua to apply to and eventually attend UBC in the first place. Lastly, through the workshop, I gained a much deeper appreciation for the paper of Orponen, Shmerkin, and Wang \cite{OSW}, which led to numerous projects, including \cite{BrightFuRen}, \cite{BrightMarshall}, and more that are still underway. I am proud to say that their paper is one of the first in my graduate studies that I have come to know like the back of my hand, and one that I will likely return to for many years to come.

Last but not least, I thank Caleb Marshall and Malabika Pramanik, to whom this thesis is dedicated. I've gotten the pleasure to know Caleb as a close colleague and an even closer friend over these past few years, and I look forward to all that awaits them in the future. How lovely it is to know someone in mathematics I thoroughly believe will be a lifelong collaborator and friend. Furthermore, though she was not an official advisor of mine, it has been great fun getting to know and be mentored by Malabika. Through her encouragement, my pipe dream of making Chapter \ref{ch:Topics} of this thesis a thorough and detailed survey of projection theory became a reality. She constantly encouraged me to write on topics I care deeply about---both mathematically and personally---which means more than I can ever say. I hope to put as much care into my work, and into the universe, as she does.

\chapter[Dedication]{}

\begin{center}
    \textit{To be and to become. \\ \, \\ To the years of work that got us here, \\ and to the years that lie ahead. \\\,\\ $\sim$ \\ \, \\ ``It is invaluable to have a friend who shares your \\ interests and helps you stay motivated.''\,-- Maryam Mirzakhani \\ \,\\ To the invaluable Caleb Z. Marshall-Cheng, \\ and to continuing to grow together. \\\,\\ $\sim$ \\ \, \\ ``People are fractal images of the universe.... \\ If you choose to care, then the universe cares.''\,-- Brennan Lee Mulligan \\ \, \\ To the thoughtful Dr. Malabika Pramanik, \\ and to continuing to put care back out into the universe.}
\end{center}

\mainmatter

\chapter{Introduction}\label{ch:intro}
The work of this thesis revolves around projection theory---a topic with roots in the discrete settings of incidence geometry and extremal combinatorics, and one that has recently utilized and motivated work on tools from the continuum settings of geometric measure theory and harmonic analysis.

The primary goal of this thesis is to present an introduction to various topics in projection theory, which has seen vast developments over the past few decades. To this end, we provide a reasonably thorough survey of select topics in this area---in particular, orthogonal projections, Furstenberg sets, and radial projections. Beyond being of personal interest, I hope that these surveys, and thus this thesis, may serve as a helpful resource to others. The secondary goal of this thesis is to provide specific and concrete applications of projection theory to other problems in fractal geometry---namely Beck-type and Falconer-type problems. These applications are well-motivated in their own right, and it is my hope that by presenting them the reader may gain some insight into how one can engage with and begin exploring this interesting intersection of incidence and fractal geometry.

\section{Beck-type Problems}

Beck-type problems will be the focus of Section \ref{SEC-BECKTYPE}. In the discrete setting, Beck-type problems were first studied by J\'ozsef Beck in 1983 \cite{Beck83}, who explored the following question:

\begin{question} \label{q:intro:beckquestion}
Given a set $X\subset \R^n$ ($n\geq 2$), how large must the set of lines that contain at least two points of $X$ be? In other words, how can we bound the size of
\[
\mathcal L(X) := \{\text{lines } \ell \text{ in } \R^n : |X\cap \ell| \geq 2\}
\]
below in terms of the size of $X$?
\end{question}

\begin{remark}
    For the introductory purposes, we focus our attention to $\R^2$, but it is worth noting that all of the results in this section adapt to $\R^n$.
\end{remark}

Beck studied this in the discrete geometry setting with $X \subset \R^2$ finite and cardinality being the notion of ``size.'' On the one hand, if $X$ is contained in a line, $\ell$, then $|\mathcal L(X)| = 1$ (namely $\mathcal L(X) = \{\ell\}$). On the other hand, we should heuristically expect that if our points in $X$ are randomly placed, roughly any two points in $X$ should determine a line in $\mathcal L(X)$, i.e. we should generically expect that $|\mathcal L(X)| \gtrsim |X|^2$. 

\begin{notation} \label{SimNotation}
    Here and throughout, $A\gtrsim B$ means there exists a positive constant $C$ such that $A\geq CB$, and $A\gtrsim_{\chi} B$ means there exists a positive constant $C:= C(\chi)$ depending on $\chi$ such that $A\geq CB$.
\end{notation}

Beck's theorem from incidence geometry states that these are the only two cases---either a large fraction of $X$ lies on a given line $\ell$, or $|\mathcal L(X)| \gtrsim |X|^2$.

\begin{theorem}[Discrete Beck's theorem \cite{Beck83}] \label{thm:introBeck}
For $X\subset \R^2$ finite, either
\begin{itemize}
    \item[1)] there exists a line $\ell$ with $|X\cap \ell| \gtrsim |X|$, or 
    \item[2)] $|\mathcal L(X)| \gtrsim |X|^2$.
\end{itemize}
\end{theorem}

\begin{remark}
    In Section \ref{sec:discretebecklines}, we generalize the above statement to lines that contain pairs of distinct points from two sets $X,Y$. See Theorem \ref{thm:ch4Beckbivariable}.
\end{remark}

In the continuum setting, an analogous statement to Beck's theorem was obtained in 2022 by Orponen, Shmerkin, and Wang \cite{OSW} utilizing Hausdorff dimension. For the purposes of this introduction, the unfamiliar reader can simply think of Hausdorff dimension as a notion of size, denoted $\dim$, that is defined on metric spaces (see Section \ref{sec:HausdorffDim}). Orponen, Shmerkin, and Wang's ``continuum version of Beck's theorem'' states the following:

\begin{theorem}[Continuum Beck's theorem \cite{OSW}] \label{thm:introOSWBeck}
    For $X\subset \R^2$ Borel, either 
    \begin{itemize}
        \item[1)] there exists a line $\ell$ such that $\dim (X\setminus \ell) < \dim X$, or
        \item[2)] $\dim \mathcal L(X) \geq \min \{2\dim X, 2\}$.
    \end{itemize}
\end{theorem}

\begin{remark} 
We will refer to problems regarding lowerbounds on the size of $\mathcal L(X)$ as Beck-type problems. Furthermore, to define the Hausdorff dimension of the set of lines $\mathcal L(X)$, we need define the metric space structure on the set of affine lines in $\R^2$. See Section \ref{sec:HausdorffDim}.
\end{remark}

Firstly, note that the tools used to prove Theorem \ref{thm:introOSWBeck} (the statement of which is intuitive, if nothing else, from the discrete point of view) built upon and led to breakthroughs in projection theory. These breakthroughs involved the theory of orthogonal projections, Furstenberg sets, and radial projections. We introduce these topics in Section \ref{sec:introTaste}.

In regards to the lowerbound itself, note that the ``2'' term is natural in Theorem 1.3 Item 2 in that the set of affine lines in $\R^2$ is 2-dimensional. This can informally be seen by putting (non-vertical) affine lines into slope-intercept form, i.e. $\ell = \{(x,y): y = mx + b\}$. In particular, we can parameterize non-vertical lines by the 2-dimensional set of parameters $(m,b) \in \R^2$. Hence, in $\R^2$ it is always the case that $\dim \mathcal L(X) \leq 2$ (see Section \ref{sec:HausdorffDim}).  Furthermore, loosely speaking, the ``$2\dim X$'' term in the minimum is the analogue of $|X|^2$ in Theorem \ref{thm:introBeck}. This type of logarithmic discrete-to-continuum phenomenon motivates the first application of projection theory we will discuss within this thesis known as tlle \textit{Erd\H{o}s--Beck theorem} (resp. continuum Erd\H{o}s--Beck). In the discrete setting, Beck proved the following bound conjectured by Erd\H{o}s:

\begin{theorem}[Discrete Erd\H{o}s--Beck theorem \cite{Beck83}] \label{thm:introerdosBeck}
    Let $X\subset \R^2$ be a finite set and $0 \leq t \leq |X|-2$
    be such that 
    \[ \max_{\ell \in \mathcal L(X)} |X\cap  \ell| =|X|-t.
    \]
    Then, $|\mathcal L(X)|\gtrsim |X|t.$
\end{theorem}

That is, if we are in the case where there exists a line with a large fraction of $X$, but we still know that we have at least $t$-many points of $X$ off of any particular line, then $|\mathcal L(X)| \gtrsim |X|t$. Upon learning of the continuum version of Beck's theorem of Orponen--Shmerkin--Wang, Marshall and I set out to prove a continuum Erd\H{o}s--Beck theorem. Motivated by the approach of \cite{OSW}, Marshall and I utilized the theory of radial projections and (dual) Furstenberg sets to obtain the following in 2024:

\begin{theorem}[Continuum Erd\H{o}s--Beck theorem 
\cite{BrightMarshall}] \label{IntroEB}
    Given a Borel set $X\subset \R^2$, suppose there exists a line $\ell'$ such that $\dim (X\setminus \ell')<\dim X$. Further, let $0< t < \dim X$ be such that $\dim (X\setminus \ell)\geq t$ for all lines $\ell$. Then,
    \[
    \dim \mathcal L(X) \geq \dim X + t.
    \]
\end{theorem}

\noindent The proof of Theorem \ref{IntroEB} is the content of Section \ref{sec:BMarshall}.

\section{Falconer-type Problems}

Falconer-type (distance) problems will be the focus of Section \ref{SEC:FALCONERTYPE}. We begin our discussion again with the discrete setting, which may best be motivated by the following question asked by Paul Erd\H{o}s in 1946 \cite{Erdos1946}:

\begin{question}\label{falconertypequestion}
    Given a set $A\subset \R^n$ $(n\geq 2)$, how large must the set of distances determined by $A$ be? In other words, how can we bound the size of
    \[
    \Delta(A) := \{|x-y| : x,y\in A\}
    \]
    below in terms of the size of $A$?
\end{question}

The \textit{Erd\H{o}s distance problem} conjectures that given $A\subset \R^n$ finite, then for all $\epsilon>0$ there exists a constant $C_\epsilon>0$ such that 
\[
|\Delta(A)| \geq C_\epsilon |A|^{\frac{2}{d}-\epsilon}.
\]
The sharpness of this conjecture can be observed by taking $A$ to be a lattice $A = \mathbb{Z}^n \cap [0,p]^n$. Furthermore, while the problem remains open in $\R^n$ for all $n\geq 3$, the problem was resolved in the plane by Guth and Katz \cite{GuthKatz}. 

In the continuum setting, an analogous question was introduced by Falconer in 1985 \cite{Falconer82} known as the \textit{Falconer distance problem}. The problem asks for a threshold value $s$ such that any compact set $A\subset \R^n$ with $\dim A > s$ satisfies that $\Delta(A)$ has positive Lebesgue measure. 

\begin{conjecture}[Falconer] \label{FalconerConj}
    Given a compact set $A\subset \R^n$ with $n\geq 2$, 
    \[
    \dim A > \frac{n}{2} \quad \text{ implies that } \quad \mathcal L^1(\Delta(A)) >0.
    \]
\end{conjecture}

\begin{remark}
Here, $\mathcal L^n$ is the Lebesgue measure on $\R^n$. See Section \ref{sec:HausdorffDim}. We also use $\mathcal L$ (without a numerical superscript) to denote families of lines.
\end{remark}

The study of this conjecture has utilized and motivated work on a number of tools in geometric measure theory and harmonic analysis, and the conjecture remains open for all $n$. Falconer \cite{Falconer85} proved that if $\dim A >\frac{n+1}{2}$ the conclusion holds, and if $\dim A \leq \frac{n}{2}$ the conclusion fails. The best lowerbounds towards the Falconer distance conjecture currently give:

\begin{theorem}[\cite{GuthIosevichOuWang,DuOuRenZhang}]
Let $n\geq 2$ and let
\[
f(n) = \begin{cases}
    \frac{5}{4}, & n=2 \quad \quad \text{(Guth--Iosevich--Ou--Wang \cite{GuthIosevichOuWang})} \\
    \frac{n}{2} + \frac{1}{4} - \frac{1}{8n + 4}, & n\geq 3 \quad \quad \text{(Du--Ou--Ren--Zhang \cite{DuOuRenZhang})}.
\end{cases}
\]
Then, for all compact sets $A\subset \R^n$, 
\[
\dim A > f(n) \implies \mathcal L^1(\Delta(A))>0.
\]
\end{theorem}

Moreover, both groups of authors obtain \textit{pinned results}, showing that if $A\subset \R^n$ is compact and $\dim A>f(n)$, there exists an $a\in A$ so that the \textit{pinned distance set} pinned at $a,$
\[
\Delta^a(A) := \{|a-y| : y\in A\},
\]
has positive Lebesgue measure (and thus so too does $\Delta(A)$). 

It is interesting to note that both Guth--Iosevich--Ou--Wang and Du--Ou--Ren--Zhang make use of projection theoretic results; the former making use of Orponen's 2019 radial projection theorem \cite{OrponenRadProjSmoothness}, and the latter making use of Ren's 2023 radial projection theorem \cite{RenRadProj} (see Theorems \ref{orponenradprojtheorem} and \ref{renradprojTheorem} respectively). Motivated by these applications of projection theory, and the recent developments in the study of orthogonal and radial projections, Marshall, Senger, and I \cite{BrightMarshallSenger} studied an analogue of the Falconer distance problem for \textit{pinned dot product sets} in 2024. Given $A\subset \R^n$ and $a\in \R^n$, let 
\[
\Pi^a(A) := \{a\cdot y : y\in A\} \subset \R
\]
where $\cdot$ denotes the Euclidean dot product.

\begin{theorem}[\cite{BrightMarshallSenger}] \label{BMS-IntroDotProd}
    Let $A\subset \R^n$ be Borel with $\dim A > \frac{n+1}{2}$ and $n\geq 2$. Then, there exists an $a\in A$ such that 
    \[
    \mathcal L^1(\Pi^a(A)) >0.
    \]
\end{theorem}

Our approach to Theorem \ref{BMS-IntroDotProd} makes use of a standard, but notably sharp, orthogonal projection and radial projection estimates. Furthermore, we obtain better results by applying more recent results for both topics, see Theorem \ref{BMS-DotProd2}. It is interesting to note that, like the Falconer distance problem, we do not conjecture the above lowerbound on $\dim A$ to be sharp. At least in the plane, we conjecture that that if $A\subset \R^2$ is Borel with $\dim A>1$, there exists an $a\in A$ such that $\mathcal L^1(\Pi^a(A))>0$. However, given our approach makes use of sharp orthogonal and radial projection estimates, this leads us to make a natural conjecture, see Conjecture \ref{BMS-DotProdConj}.

\section{A Taste of Projection Theory}\label{sec:introTaste}

In this section, I give a brief taste of some topics in projection theory that we will explore throughout the thesis. Note that the below problems are interesting in higher dimensions, but for illustrative purposes we stick to $\R^2$.

\subsection{Orthogonal Projections} Suppose you are standing outside on a sunny day from dawn til dusk, observing your shadow as it changes over time. You may notice that the only time of day you cannot see a large fraction of your shadow is at noon when the sun is right above you. In this way, at almost every time of day, your shadow is relatively large. How general is this phenomena, and can we make such statements mathematically precise?

The concept of a shadow here is made precise by orthogonal projections. Given $\theta \in \mathbb{S}^{1} \subset \R^2$, let $P_\theta : \R^2 \to \R$ denote the orthogonal projection onto the line through the origin in direction $\theta$, denoted $\ell_\theta$. More precisely, we have $P_\theta(x) = \theta \cdot x$ where $\cdot$ is the Euclidean dot product. Given $X \subset \R^2$ Borel, how large could $\dim P_\theta(X)$ be? 
Using the fact that orthogonal projections are Lipschitz maps, those familiar with Hausdorff dimension can show that
\[
\dim P_\theta(X) \leq \min\{\dim X, 1\} \quad \text{for all }\theta \in \mathbb{S}^{1}.
\]
Intuitively all this is saying is that the size of the shadow is at most the size of the object casting the shadow ($\dim X$), and at most the size of what the shadow is being cast upon (in this case, the line $\ell_\theta$, which has dimension 1). In 1954, Marstrand \cite{Marstrand54} showed that
\[
\dim P_\theta(X) = \min\{\dim X, 1\} \quad \text{for almost every $\theta \in \mathbb{S}^1$}.
\]
Here, ``almost every'' is with respect to the surface measure on the circle.
In other words, \textit{almost all of the time}, the orthogonal projection of $X$ is \textit{as large as it could be}---matching our intuition from our nice sunny day outside.

There are a number of ways to probe this concept further. For instance, we could ask questions of the form: how many times like ``noon'' are there? That is, how often is our shadow \textit{smaller} than what we would typically expect? Sharp bounds on the size of such ``exceptional'' (i.e. rare) directions where the orthogonal projection is small are known as exceptional set estimates. A survey on exceptional set estimates is presented in Section \ref{sec:ESE}.

\subsection{Furstenberg Sets} 

Suppose you are a spider trying to protect your home from collapsing under the weight of raindrops collecting on strands of your spiderweb. As the drops form, you frantically move water from one part of the web to another, knowing that if more than $N$ water droplets (in total) exist on the web, your home will wash away until the sun comes out to dry up all the rain. Furthermore, you know that if the drops aren't (roughly) evenly distributed across strands in your web, your home will collapse anyways due to its precarious position atop the water spout. If your web has $\geq t$-many strands and each strand holds $\geq s$-many droplets, is it possible to arrange the water to avoid being washed out by the rain?

Naively, one may expect that given the above $(s,t)$-conditions (viewing strands of your webs as lines in $\R^2$ and droplets as points on said lines), if $st \geq N$ there will be no way to move the water to avoid being washed away. However, if you prepare appropriately and make your web highly interconnected with many strands (so that each strand only needs a few drops each, and so that your strands overlap frequently allowing you to place droplets at their intersections) you may be able to get away with having $\sim s^{3/2}t^{1/2}$ drops of water on your web at a given time (see Proposition \ref{discreteFurst}).

In this example, the drops of water on your web form a \textit{discrete $(s,t)$-Furstenberg set}. A discrete analogue of the $(s,t)$\textit{-Furstenberg set problem} then asks: given a finite set $X\subset \R^n$ such that there exists $\geq t$-many lines with $\geq s$-many points from $X$ on each line, how large must $|X|$ be? 

In the continuum setting, the Furstenberg set problem asks the same replacing the cardinality of points and lines with Hausdorff dimension. In particular, a Borel set $X\subset \R^n$ is an $(s,t)$\textit{-Furstenberg set} if there exists a $\geq t$-dimensional set of lines $\mathcal L$ with $\dim (X\cap \ell) \geq s$ for all $\ell \in \mathcal L$ (see Definition \ref{FurstDEF}). The $(s,t)$\textit{-Furstenberg set problem} then asks for sharp lowerbounds on $\dim X$ for all choices of $s$ and $t.$ Progress on the $(s,t)$-Furstenberg set problem was made rapidly over the past 25 years, and in 2023 the sharp lowerbound for all ranges of $s$ and $t$ was resolved by Orponen--Shmerkin \cite{OrponenShmerkinABC} and Ren--Wang \cite{RenWang} (see Theorem \ref{FURSTENBERG}).
Notably, it is also within these papers that sharp exceptional set estimates were obtained in the plane (Theorem \ref{OSRW}). These topics are very much related! A survey on the $(s,t)$-Furstenberg set problem is presented in Section \ref{sec:Furstenberg}.

\subsection{Radial Projections} 

Suppose you are an astronaut in space looking for a place to park your spaceship, and as a devout space explorer you are most interested in being able to see stars in as many different directions as possible. To think of this as a projection problem, you can think about the ``directions'' of these stars from your ship as the corresponding notion of ``shadows.'' What is the largest number of ``shadows'' you could hope to see? Furthermore, if you have a limited set of places you can park, can you always find a nice view?

The concept of shadow here is made precise by radial projections. Given $x\in \R^2$ let $\pi_x: \R^2 \setminus \{x\}\to \mathbb{S}^{1}$ map the point $y \in \R^2 \setminus \{x\}$ to the \textit{direction} of $y$ from center point $x$. More precisely, we have $\pi_x(y) = \frac{y-x}{|y-x|}$ where $|\cdot|$ is the Euclidean distance. Given a Borel set $Y\subset \R^2$, how small or large could $\dim \pi_x(Y\setminus \{x\})$ be? One can directly show
\[
\max\{\dim Y - 1,0\} \leq \dim \pi_x(Y\setminus \{x\}) \leq \min\{\dim Y, 1\} \quad \text{for all $x\in \R^2$.}
\]
The lowerbound on $\dim \pi_x(Y\setminus \{x\})$ heuristically comes from the fact that each fiber of the radial projection is 1-dimensional (and Hausdorff dimension is always nonnegative). Furthermore, the upperbound can be argued formally by using that radial projection is a locally Lipschitz map, but the intuition is that the size of the shadow is at most the size of the object casting the shadow ($\dim Y$), and at most the size of what the shadow is being cast upon (in this case, the sphere $\mathbb{S}^1$, which has dimension 1). 

We will postpone our discussion of specific results in radial projections for the time being as the key statements we will focus on are relatively new in the literature. In fact, in Liu's radial projection paper from 2019 \cite{Liu2020}, he wrote ``There is very little known on the Hausdorff dimension of radial projections,'' though this is certainly not the case anymore a mere five years later. A survey on radial projections is presented in Section \ref{sec:RadProj}.

\subsection{Further Remarks}
When learning about these problems for the first time, it is completely reasonable to ask why they are related to one another---let alone how they might relate to Beck-type and Falconer-type problems. A direct answer to this question may be gleaned from the proofs within Chapter \ref{ch:problems}. However, a heuristic answer can be obtained by understanding that at its core, Beck-type problems and the Falconer-type problem for dot products (in the plane) involve considering how points can lie on lines, and relatedly how lines can pass through points. The same is true for the above three topics. 

Take orthogonal projections for instance. Given some $\alpha \in \R$, notice that $P_\theta^{-1}(\alpha)$ is a line in direction $\theta^\perp$. Hence, $P_\theta^{-1}(P_\theta(X))$ is a set of lines orthogonal to $\theta$ that cover $X$, denoted $\mathcal L_\theta$. In particular, the problem of finding sharp exceptional set estimates comes down to finding configurations of points $X$ such that there exists many $\theta \in \mathbb{S}^{1}$ so that the minimum number of lines needed to cover $X$ in direction $\theta^\perp$ is small (and hence, $P_\theta(X)$ is correspondingly small). Similarly for radial projections, $\pi_x^{-1}(\pi_x(Y \setminus \{x\}))$ is a set of lines covering $Y\setminus \{x\}$ that pass through $x$. Researching how such families of lines interact is at the heart of this subfield of analysis.

The connection between these projection theory problems and Beck-type/Falconer-type problems is a two-way street. While a number of Beck-type theorems have made use of radial projections and Furstenberg sets, it is interesting to study Beck-type and Falconer-type problems on their own. Furthermore, through the pursuit of Beck-type statements, we have gained new insights and new problems in projection theory to explore---insights I will expound upon by presenting the work of Marshall and I \cite{BrightMarshall}. Similarly, while Marshall, Senger, and I \cite{BrightMarshallSenger} utilize results in orthogonal and radial projections to study the Falconer-type problem for dot products, the pursuit of sharp bounds to this problem has led to new and interesting conjectures regarding these well-studied projection maps. In the following section, I discuss what some of this work entails as I prepare a roadmap for the thesis.

\section{A Roadmap through the Thesis}

In Chapter \ref{ch:background}, we will begin with necessary background on incidence geometry and geometric measure theory. In Section \ref{sec:Incidences}, we give a brief introduction to discrete geometry. In particular, given a set of points $P$ and lines $\mathcal L$, we consider the size of the set of point-line incidences, i.e.
\[
\mathcal I(P,\mathcal L) := \{(p,\ell) \in P \times \mathcal L : p \text{ lies on } \ell\}.
\]
We will present a few bounds on $|\mathcal I(P,\mathcal L)|$ over both Euclidean space and finite fields---bounds that will allow us to motivate sharp bounds for problems in the continuum setting. Before going into such problems, however, in Section \ref{sec:HausdorffDim} we define the corresponding notion of size in the continuum setting, namely Hausdorff dimension. We will also take this time to define key metric spaces and properties thereof that we will need throughout the thesis---namely the unit sphere $\mathbb{S}^{n-1}$, the set of $k$-dimensional subspaces in $\R^n$ (i.e. the Grassmannian $\mathcal G(n,k)$), and the set of $k$-dimensional affine planes in $\R^n$ (i.e. the affine Grassmannian $\mathcal A(n,k)$).

With these tools at our disposal, in Chapter \ref{ch:Topics} we will then be able to give a few surveys on orthogonal projections, Furstenberg sets, and radial projections (Sections \ref{sec:ESE}--\ref{sec:RadProj}). We will primarily focus on giving an exposition of the literature in both the discrete and continuum settings. That said, I will highlight a few arguments for each problem, including outlines of proofs from my work with collaborators in \cite{BrightGan, BrightGan2, BrightFuRen}. In Section \ref{BGanProofs}, I present a proof of Falconer's projection theorem over $\mathbb{F}_p^n$ using the high-low method following the approach in \cite{BrightGan}. To this end, we will also introduce the geometry of affine planes and Fourier analysis over finite fields. In Section \ref{sec:DualFurstenberg}, I discuss \textit{dual} Furstenberg sets, and present a Cauchy--Schwarz bound that can be adapted to a Hausdorff dimensional result due to myself, Fu, and Ren \cite{BrightFuRen}. Lastly, in Section \ref{BGanRadProjProof}, I give a proof of a Falconer-type radial projection theorem over $\mathbb{F}_p^n$ via the high-low method. To do so we adapt methodology from the Euclidean setting due to myself and Gan \cite{BrightGan}.

Finally, in Chapter \ref{ch:problems}, we present applications of the results in Chapter \ref{ch:Topics} towards problems in projection theory. In Section \ref{SEC-BECKTYPE} we will cover Beck-type theorems for lines. We begin with the discrete setting in Section \ref{sec:discretebecklines}, with a discussion and proof of Beck's theorem and the Erd\H{o}s--Beck theorem (Theorems \ref{thm:introBeck} and \ref{thm:introerdosBeck}). We also give a generalization of Beck's theorem that considers lines containing two distinct points from two (possibly different) sets $X,Y$ (Theorem \ref{thm:ch4Beckbivariable}). In doing so, we are able to obtain a discrete radial projection theorem that echoes the work of Orponen--Shmerkin--Wang \cite{OSW}.

This will segue nicely into the continuum setting in Section \ref{sec:OSWRen}, where we prove the continuum version of Beck's theorem due to Orponen, Shmerkin, and Wang (Theorem \ref{thm:introOSWBeck}) and discuss a higher dimensional analogue due to Ren (Theorem \ref{thm:ch4RenBeck}) \cite{OSW, RenRadProj}. Then, in Section \ref{sec:BMarshall}, we prove the continuum version of the Erd\H{o}s--Beck theorem (Theorem \ref{IntroEB}) due to myself and Marshall \cite{BrightMarshall}. This will include a number of sharp examples that led to a conjectured equality for $\dim \mathcal L(X)$ (Conjecture \ref{BM-EBConjecture}). 

In Section \ref{SEC:FALCONERTYPE}, we will cover Falconer-type theorems. We begin in the discrete setting in Section \ref{erdosprob} with a discussion of the Erd\H{o}s distance problem and the work of Guth and Katz \cite{GuthKatz}. Then, in Section \ref{falconerprob} we give an introduction to the Falconer distance problem, focusing our attention to ways in which the Falconer distance problem has utilized projection theory. In Section \ref{sec:BMS}, we conclude with a Falconer-type theorem for dot products due to myself, Marshall, and Senger \cite{BrightMarshallSenger} (Theorems \ref{BMS-DotProd} and \ref{BMS-DotProd2}). Our proof methodology makes use of both standard and modern estimates for orthogonal and radial projections, and though we do not conjecture our bound to be sharp, this leads us to make a conjecture on the relation between these two projection maps (Conjecture \ref{BMS-DotProdConj}).

\chapter{Preliminaries} \label{ch:background}
Over the past few years, much progress has been made regarding problems in projection theory, harmonic analysis, and geometric measure theory. Furthermore, many of the techniques that have been used to attack these problems involve discrete analogues in incidence geometry. 

The purpose of this chapter is to provide brief introductions to tools we will need in both the discrete and continuum settings respectively. In Section \ref{sec:Incidences}, we provide a brief introduction to incidence theory, specifically regarding incidences between points and lines. Incidences between points and lines, and in particular the Szemer\'edi--Trotter theorem, have motivated key techniques, results, and conjectures in the continuum setting. To this end, in Section \ref{sec:HausdorffDim} we will give the main definitions we need from the continuum setting---namely Hausdorff measures and dimension on various metric spaces. We limit the scope of Section \ref{sec:HausdorffDim} to the foundational definitions and lemmas we will assume familiarity with throughout the thesis. 

\section{An Intro to Incidences} \label{sec:Incidences}

Broadly speaking, in \textit{discrete} geometry, an incidence is a relationship between different geometric objects. We will primarily consider incidences between points and lines---in particular, when a point lies on a line or equivalently when a line passes through a point. Consider the following problem regarding incidences between points and lines.

\begin{question}\label{sec2:q:motivatingST}
    Let $P\subset \R^n$ be a finite set of points and let $\mathcal L$ be a finite set of lines in $\R^n$ (which need not pass through the origin). Define the set of \textit{incidences} between $P$ and $\mathcal L$ to be
    \[
    \mathcal I(P,\mathcal L) = \{(p,\ell) \in P \times \mathcal L : p \text{ lies on } \ell\}.
    \]
    How large can $|\mathcal I(P,\mathcal L)|$ be in terms of $|P|$ and $|\mathcal L|$?
\end{question}

At first glance, it may seem like the trivial bound ($|\mathcal I(P,\mathcal L)| \leq |P||\mathcal L|$) is the best answer one can hope for in regards to Question \ref{sec2:q:motivatingST}. That said, if $|\mathcal I(P,\mathcal L)| = |P||\mathcal L|$, this implies that every point in $P$ lies on every line in $\mathcal L$. In fact, such configurations exist, but in very specific instances. For example, you could take $P$ to equal a single point $\{p\}$, and let $\mathcal L$ be a finite set of lines that each pass through $p$; in this case, $|\mathcal I(P,\mathcal L)| = |\mathcal L| = |P||\mathcal L|$ (as $|P| = 1$). Similarly, you could take $\mathcal L$ to be a single line $\{\ell\}$, and let $P$ be a finite set of points lying on $\ell$. Notice that in both of these examples, the reason the bound $|P||\mathcal L|$ is sharp is due to either $|P|$ or $|\mathcal L|$ being equal to 1. In particular, you can have arbitrarily many lines through a single point, or arbitrarily many points on a single line. 

Once $\min\{|P|,|\mathcal L|\} \geq2$, however, it is impossible to find a configuration of points and lines such that every point lies on every line in $\R^n$. This is as two points uniquely determine a line, and two lines intersect at at most one point. Utilizing this, one can obtain the following Proposition.

\begin{proposition}\label{sec2:prop:STPrelim}
    Let $P\subset \R^n$ be a finite set of points and let $\mathcal L$ be a finite set of affine lines in $\R^n$. Then, 
    \[
    |\mathcal I(P,\mathcal L)| \leq \min\{|P||\mathcal L|^{1/2}+ |\mathcal L|, |\mathcal L||P|^{1/2} + |P|\}.
    \]
\end{proposition}

\begin{remark}
    We include the proof of the above bound as an analogous argument will be used in the continuum setting due to work of myself, Fu, and Ren \cite{BrightFuRen}, see Section \ref{sec:DualFurstenberg}.
\end{remark}

\begin{proof}[Proof of Proposition \ref{sec2:prop:STPrelim}]
We prove the first bound, and the second follows similarly. Firstly, notice that 
\[
|\mathcal I(P,\mathcal L)| = \sum_{\ell \in \mathcal L} |P\cap \ell|,
\]
and thus by the Cauchy--Schwarz inequality we have that 
\begin{align*}
    |\mathcal I(P,\mathcal L)|^2 &= \left(\sum_{\ell \in \mathcal L} |P\cap \ell|\cdot \mathbf{1}_{\mathcal L}(\ell)\right)^2 \\
    &\leq \left(\sum_{\ell \in \mathcal L} |P\cap \ell|^2\right)\left(\sum_{\ell \in \mathcal L}\mathbf{1}_{\mathcal L}(\ell)^2\right) \\
    &= |\mathcal L| \left(\sum_{\ell \in \mathcal L} |P\cap \ell|^2\right).
\end{align*}
On the other hand, notice that we can write the sum in the last inequality as sum of indicator functions: 
\[
    \sum_{\ell \in \mathcal L} |P\cap \ell|^2 = \sum_{\ell \in \mathcal L} \left(\sum_{p\in P} \mathbf{1}_{\ell}(p)\right)^2 = \sum_{\ell \in \mathcal L} \sum_{p,p'\in P} \mathbf{1}_\ell(p)\mathbf{1}_\ell(p').
\]
At this point, we want to use that there exists at most one line through two \textit{distinct} points, so we break the above sum into two parts: one where $p = p'$ and one where $p\neq p'.$ In particular,
\[
\sum_{\ell \in \mathcal L} |P\cap \ell|^2 = \left[\sum_{\ell \in \mathcal L} \sum_{p\in P} \mathbf{1}_\ell(p)^2 \right]+ \left[\sum_{\ell \in \mathcal L} \sum_{p\neq p'} \mathbf{1}_\ell(p)\mathbf{1}_\ell(p')\right].
\]
Notice that the first term is precisely $|\mathcal I(P,\mathcal L)|$ as $\mathbf{1}_{\ell}(p)^2 = \mathbf{1}_\ell(p)$. Additionally, the second term is at most $|P|^2$ as given two points $p\neq p'$, there exists at most one line $\ell$ such that $\mathbf{1}_\ell(p)\mathbf{1}_\ell(p') = 1$. Note that this line may not be contained in $\mathcal L$, but in either case this implies we have the bound
\[
\sum_\ell |P\cap \ell|^2 \leq |\mathcal I(P,\mathcal L)| + |P|^2,\]
which implies 
\[
|\mathcal I(P,\mathcal L)|^2 \leq |\mathcal L|(|\mathcal I(P,\mathcal L)| + |P|^2).
\]
Rearranging the above inequality, it follows that 
\[
|\mathcal I(P,\mathcal L)| \leq |P||\mathcal L|^{1/2}+ |\mathcal L|.
\]
This gives the first term of the Proposition. The second follows similarly by instead using that two distinct points uniquely determine a line and
\[
|\mathcal I(P,\mathcal L)| = \sum_{p \in P} |\{\ell \in \mathcal L : p \in \ell\}|. \qedhere
\]
\end{proof}

We briefly note that the key property we used is that two lines intersect in at most one point and two points uniquely determine a line, which aren't unique to Euclidean geometry. For instance, consider the field of prime characteristic $\mathbb{F}_p$. Over $\mathbb{F}_p^2$, one can define affine lines to be sets of the form 
\[
L(\theta,b) := \{\theta t + b : t\in\mathbb{F}_p\} = \{(x,y) \in \mathbb{F}_p^2 : y = \theta x + b\}
\]
where $\theta, b\in \mathbb{F}_p^2$ and $\theta \neq 0$ (see Definition \ref{FpAffineLines}). One can show that distinct affine lines over $\mathbb{F}_p$ intersect in at most one point, and thus Proposition \ref{sec2:prop:STPrelim} also holds over $\mathbb{F}_p$. Furthermore, over $\mathbb{F}_p$, Proposition \ref{sec2:prop:STPrelim} is sharp! There are a few geometric lemmas one would need to quickly confirm to see this:
\begin{itemize}
    \item There are $p^2 + p$ unique affine lines over $\mathbb{F}_p^2$.
    \item Every affine line contains exactly $p$ points of $\mathbb{F}_p^2$.
\end{itemize}
Both of these follow from $p$ being prime. Hence, let $P = \mathbb{F}_p^2$ and $\mathcal L$ be any set of $p^2$ distinct affine lines (of which we know we have enough by the first bullet point). Since each line passes through exactly $p$-many points, we have 
\[
|\mathcal I(P,\mathcal L)| = \sum_{\ell \in \mathcal L} |\mathcal I(P,\ell)| = |\mathcal L| p = p^3,
\]
while Proposition \ref{sec2:prop:STPrelim} gives
\[
|\mathcal I(P,\mathcal L)| \lesssim (p^2)^{3/2} = p^3.
\]
Hence, up to a multiplicative constant, the Cauchy--Schwarz incidence bound \textit{cannot} be improved over $\mathbb{F}_p$ for arbitrary sets of points and lines. For a more detailed discussion of the geometry of finite fields, see Section \ref{BGanProofs}.

That said, Proposition \ref{sec2:prop:STPrelim} \textit{can} be improved over Euclidean space using the topology of $\R$ (namely, that $\R$ is ordered). The sharp bound for point-line incidences in $\R^n$ is given by the Szemer\'edi--Trotter theorem which states:

\begin{theorem}[Szemer\'edi--Trotter \cite{SzemerediTrotter}] \label{SZEMEREDITROTTER}
    Let $P\subset \mathbb{R}^n$ be a finite set of points, and $\mathcal L$ be a finite set of affine lines. Then, 
    \[
    |\mathcal I(P,\mathcal L)| \leq 4( |P|^{2/3}|\mathcal L|^{2/3} + |P| + |\mathcal L|).
    \]
\end{theorem}

\begin{remark}
    The constant 4 can be slightly improved upon, though we do not discuss this further here.
\end{remark}

Note that in the above bound, the terms $|P|$ and $|\mathcal L|$ are certainly needed given that you can have $|P|$-many points on a line, and $|\mathcal L|$-many lines through a point. There are a number of ways to prove the Szemer\'edi--Trotter theorem. One method, due to Sz\'ekely \cite{Szekely97}, uses graph theory and crossing numbers. There are a number of references which follow this graph theoretic approach, and a detailed proof can be found in an expository paper I mentored the writing of in 2024 \cite{primesSTpaper}. One can also prove the Szemer\'edi--Trotter theorem using polynomial partitioning, which can be found in Guth's \textit{Polynomial Methods in Combinatorics} \cite[Chapter 10]{GuthPolynomialMethods}. 

Given how fundamental the Szemer\'edi--Trotter theorem (Theorem \ref{SZEMEREDITROTTER}) is to this area of harmonic analysis and geometric measure theory, and thusly how often this result has been exposited in the literature, we omit the proof. However, there are a few equivalent versions of Theorem \ref{SZEMEREDITROTTER} whose statements and proofs (of equivalency) will be useful throughout. These equivalent statements regard \textit{rich points} and \textit{rich lines}, and we will spend the rest of the section on this \textit{rich} topic.

\begin{definition}\label{richpointsandlines}
Let $r\geq 2$. Given $\mathcal L$ is a finite set of affine lines in $\R^n$, we let $P_r(\mathcal L)$ denote the set of $r$-rich points:
\[
P_r(\mathcal L) := \{p \in \R^n: \text{$p$ lies on $\geq r$-many $\ell \in \mathcal L$}\}.
\]
Similarly, given $P \subset \R^n$ is finite, we let $\mathcal L_r(P)$ denote the set of $r$-rich lines:
\[
\mathcal L_r(P) := \{\ell \subset \R^n : \text{$\ell$ contains $\geq r$-many $p \in P$}\}.
\]
\end{definition}

\noindent Using this notation, we have the following proposition.

\begin{proposition}\label{STEquivalencies}
    The following statements are equivalent:
    \begin{enumerate}
        \item[a)] Theorem \ref{SZEMEREDITROTTER} in $\R^2$;
        \item[b)] Given a finite set of affine lines $\mathcal L$ in $\R^2$, for all $r\geq 2$ we have 
        \[
        |P_r(\mathcal L)| \lesssim \frac{|\mathcal L|^2}{r^3} + \frac{|\mathcal L|}{r};
        \]
        \item[c)] Given a finite set of points $P$ in $\R^2$, for all $r\geq 2$ we have
        \[
        |\mathcal L_r(P)| \lesssim \frac{|P|^2}{r^3} + \frac{|P|}{r}.
        \]
    \end{enumerate}
\end{proposition}

We spend the rest of this section sketching the equivalency of the above statements. A complete proof, on which this sketch is based, can be found in Sheffer's \textit{Polynomial Methods and Incidence Theory} \cite[Chapter 1]{ShefferPolynomialMethods}.

\begin{proof}[Proof Sketch of Proposition \ref{STEquivalencies}]
    We first show (a) implies (b). Let $P_r := P_r(\mathcal L)$. If $r$ is sufficiently small (e.g., on the same order of magnitude as the implicit constant in Theorem \ref{SZEMEREDITROTTER}), then (b) follows immediately from noting that there are $\binom{|\mathcal L|}{2}\sim |\mathcal L|^2$ intersection points in a set of $|\mathcal L|$ lines. Thus, we may assume $r$ is sufficiently large and consider $|\mathcal I(P_r, \mathcal L)|$. Given each point in $P_r(\mathcal L)$ lies in at least $r$-many lines in $\mathcal L$, it follows that 
    \[
    |P_r| r \leq |\mathcal I(P_r,\mathcal L)|.
    \]
    By the Szemer\'edi--Trotter theorem we have 
    \[
    |P_r| r \leq |\mathcal I(P_r,\mathcal L)| \lesssim |P_r|^{2/3} |\mathcal L|^{2/3} + |P_r| + |\mathcal L|.
    \]
    Given $r$ is sufficiently large, $|P_r|$ is never the dominating term on the right hand side. If first term is dominant, it follows that $|P_r(\mathcal L)|\lesssim \frac{|\mathcal L|^2}{r^3}$, and if the third term is dominant then $|P_r(\mathcal L)|\lesssim \frac{|\mathcal L|}{r}$. Hence, we have
    \[
    |P_r(\mathcal L)| \lesssim \frac{|\mathcal L|^2}{r^3} + \frac{|\mathcal L|}{r}.
    \]

    We now loosely sketch how (b) implies (a). Fix a set of $n$-many points $P$ and $m$-many affine lines $\mathcal L$. If $n$ is sufficiently large ($n\gg m^2$), then Szemer\'edi--Trotter follows immediately from the bound of $|P|$ in Proposition \ref{sec2:prop:STPrelim}. Hence, we may assume $n \lesssim m^2$. 
    
    We then dyadically decompose our points in $\mathcal P$ based on how rich they are. In particular, let $n_j$ denote the number of points of $P$ that lie on at least $2^j$-many points of $\mathcal L$ and fewer than $2^{j+1}$, and let $P_+$ denote the set of points of $P$ that are $(\sqrt{m} +1)$-rich. Then, it follows that 
    \[
    |\mathcal I(P,\mathcal L)| \leq \sum_{j = 0}^{(\log n)/2} n_j 2^{j+1} + \mathcal |\mathcal I(P_+, \mathcal L)|.
    \]
    
    Naively plugging in the upperbound of $n$ for each $n_j$ into the above inequality won't give the desired bounds, however we can adapt the idea accordingly. Let $k = \lceil \log(m^{2/3}/n^{1/3})\rceil$ and note 
    \begin{equation}\label{ch2eq1}
    \sum_{j = 0}^{(\log n)/2} n_j 2^{j+1} = \sum_{j = 0}^{k} n_j 2^{j+1} + \sum_{j = k+1}^{(\log n)/2} n_j 2^{j+1}.
    \end{equation}
    Plugging in $n_j \leq n$ into the first sum in \eqref{ch2eq1} gives 
    \[
    \sum_{j=0}^k n_j 2^{j+1} \leq \sum_{j=0}^{k} n2^{j+1} \lesssim n(m^{2/3}/n^{1/3})
    \sim n^{2/3} m^{2/3}.
    \]
    Note that we apply this naive bound to the \textit{first sum} as we cannot apply the $r$-rich points bound. However, for the second sum in \eqref{ch2eq1}, since $2\leq j \leq (\log n)/2$, it follows by ii) that $n_j \lesssim \frac{m^2}{2^{3j}}$. Applying this bound to the second sum similarly gives the bound of $\lesssim n^{2/3} m^{2/3}$. Lastly, by (b), we have $|P_+| \lesssim \sqrt{m}$ and thus applying Proposition \ref{sec2:prop:STPrelim} again gives $|\mathcal I(P_+,\mathcal L)| \lesssim m$. These three bounds complete the proof that (b) implies (a).

    Lastly, we show that (b) is equivalent to (c) by utilizing \textit{point-line duality} (which is, to be clear, the main way we are using that we are in $\R^2$). 
    
    \begin{definition}[Point-Line Duality] \label{point-lineduality}
    Let $p = (p_1,p_2) \in \R^2$ and $\ell = mx+b$. We define the \textit{dual} of $p$, denoted $\mathbf{D}(p)$, to be the line given by $y = p_1 x - p_2$. Similarly, we define the \textit{dual} of $\ell$, denoted $\mathbf{D}^\ast(\ell)$, to be the point $(m,-b)$.
    \end{definition}

    In particular, the notion of point-line duality is encapsulated in noting: $p$ is incident to $\ell$ if and only if $\mathbf{D}^\ast(\ell)$ is incident to $\mathbf{D}(p)$. Hence, it follows that $|\mathcal I(P,\mathcal L)| = |\mathcal I(\mathbf{D}^\ast(\mathcal L), \mathbf{D}(P))|$ and $p \in P_r(\mathcal L)$ if and only if $\mathbf{D}(p) \in \mathcal L_r(\mathbf{D}^\ast(\mathcal L))$. Applying point-line duality to (b) gives (c) and vice versa.
\end{proof}

We briefly note that the above statements hold in higher dimensions.

\begin{corollary}\label{genericprojCorollary0}
    Theorem \ref{SZEMEREDITROTTER}, and statements Proposition \ref{STEquivalencies}.(b) and (c) hold in $\R^n$ for all $n \geq 2$.
\end{corollary}

\begin{proof}
    We proceed by induction on $n$, with the base case of $n=2$ being covered by Proposition \ref{STEquivalencies}. Hence, suppose the statement holds for $n-1\geq 2.$

    To prove the statement for $n$, we project onto a generic $(n-1)$-plane in $\R^n$. We show why this holds for Theorem \ref{SZEMEREDITROTTER}, and the remaining two statements follow nearly identically.
    
    Given a $(n-1)$-dimensional subspace $V$ through the origin in $\R^n$, let $P_V: \R^n \to \R^{n-1}$ denote orthogonal projection onto $V$. Let $P$ be a finite set of points in $\R^n$ and $\mathcal L$ be a finite set of affine lines in $\R^n$. Then, for all by finitely many $V$, $P_V(P)$ is a set of $|P|$-many \textit{distinct} points, and $P_V(\mathcal L)$ is a finite set of $|\mathcal L|$-many \textit{distinct} lines. To see this for $P$, note that the $V$ for which this fails satisfy $V^\perp \in S(P)$, the direction set of $P$:
    \[
    S(P) := \left\{\frac{x-y}{|x-y|} : x,y\in P, x\neq y\right\} \subset \mathbb{S}^{n-1}.
    \]
    Similarly, for all but finitely many $V$, $P_V(p)$ lies on $P_V(\ell)$ if and only if $p$ lies on $\ell$. Thus, for all but finitely many $(n-1)$-dimensional subspaces $V$,
    \[
    |\mathcal I(P,\mathcal L)| = |\mathcal I(P_V(P),P_V(\mathcal L))|.
    \]
    Applying the inductive hypothesis gives the desired result.
\end{proof}

As we noted before the proof of Proposition \ref{STEquivalencies}, the ideas above will be utilized throughout the thesis. As such, we conclude with the key ideas and how we will utilize them later:
\begin{itemize}
    \item Point-line duality will be useful in discussing the notion of dual Furstenberg sets and a theorem of myself, Fu, and Ren \cite{BrightFuRen}. See Section \ref{sec:DualFurstenberg}.
    \item In the continuum setting, a number of our desired Hausdorff dimension statements will not immediately follow from the analogous $k$-planar statement and a generic projection onto a $k$-dimensional subspace, which will make a number of radial projection and Beck-type statements more challenging to prove.
    \item Lastly, the dyadic decomposition proof method will be utilized in proving Beck's theorem. There are two things worth noting now in this regard. Firstly, Beck's theorem can be obtained as a \textit{corollary} of the Szemer\'edi--Trotter theorem. In fact, the original papers of Beck \cite{Beck83} and Szemer\'edi--Trotter \cite{SzemerediTrotter} were released in the same year and journal! Secondly, while Beck's theorem can be obtained as a corollary, this wasn't how Beck originally proved his result. In particular, Beck utilized an $\epsilon$-improvement method which inspired the proof methodology of Orponen--Shmerkin--Wang \cite{OSW}! See Section \ref{SEC-BECKTYPE}.
\end{itemize}

\section{An Intro to Hausdorff Dimension} \label{sec:HausdorffDim}

Broadly speaking, in \textit{fractal} geometry, Hausdorff dimension is a notion of ``size'' that is in some sense more refined than the notion of Lebesgue measure. As such, we begin with recalling the definition of Lebesgue measure (on $\R^n$) and explain why Hausdorff dimension is an essential tool for detecting the size of a set in geometric measure theory.

\begin{definition}[Lebesgue measure]
    Let $A\subset \R^n$ be Borel. Then, we define the \textit{Lebesgue measure} of $A$ to be
    \[
    \mathcal L^n(A) := \inf \left\{ \sum_i \omega_nr_i^n : A\subset \bigcup_i B(x_i,r_i) \right\}
    \]
    where $\{B(x_i,r_i)\}_i$ is a countable collection of open balls and $\omega_n$ is the volume of a ball of radius 1 in $\R^n$.
\end{definition}

\begin{remark}
    Note that the main aspect of $\R^n$ we used here was the ``dimension'' of the ambient space, $n$, and the notion of open balls with respect to the Euclidean metric. As such, Lebesgue measure can be defined on a wide variety of ``nice enough'' (e.g. complete and separable) metric spaces, and for the purposes of this thesis we will only consider such spaces.
\end{remark}

When it comes to Hausdorff dimension, the sets we will explore will have \textit{null} Lebesgue measure. One such class of sets are \textit{Kakeya sets}---an important class of sets in harmonic analysis and geometric measure theory.

\begin{definition}[Kakeya sets] \label{def:kakeya}
    Let $K \subset \R^n$ be compact. We say $K$ is a \textit{Kakeya set} if for every direction $\theta \in \mathbb{S}^{n-1}$, there exists a unit line segment $\overline{\ell_\theta}$ in the direction of $\theta$ such that $\overline{\ell_\theta} \subset K$.
\end{definition}

One simple but motivating example of a Kakeya set is the ball of radius $1/2$, which has positive Lebesgue measure. Is it the case that all Kakeya sets have positive measure? No! In 1928, Besicovitch \cite{Besicovitch17} constructed a Kakeya set with null Lebesgue measure. Thus it's natural to ask:

\begin{question}\label{q:vagueKakeya}
How small can a Kakeya set \textit{be}? 
\end{question}
This is where Hausdorff dimension comes in.

\begin{definition}[Hausdorff measure and dimension]
    Let $(X,d)$ be a separable metric space, $A\subset X$, and $\de>0$. Then, we define \begin{itemize}
    \item the \textit{$s$-dimensional Hausdorff content at scale $\de$} via
    \[
    \mathcal H^s_\de(A) = \inf \left\{\sum_i r_i^s : A\subset \bigcup B_i \text{ such that } r_i:= \mathrm{diam}(B_i) < \de\right\},
    \]
    \item the \textit{$s$-dimensional Hausdorff measure} of $A$ as  $\mathcal H^s(A) = \lim\limits_{\de \searrow 0} \mathcal H^s_\de(A)$, and
    \item the \textit{Hausdorff dimension} of $A$ as 
    \[
    \dim A = \inf\{s : \mathcal H^s(A) = 0\} = \sup\{s : \mathcal H^s(A) = \infty\}.
    \]
    \end{itemize}
\end{definition}

\begin{remark}
    Note that the main way that we are using $(X,d)$ is separable is in knowing that there exists a countable cover of $A$ by balls of radius $< \de$ for all $\de>0$. In particular, if $(X,d)$ is not separable, there can exist a set $A\subset X$ without such countable coverings, in which case $\mathcal H^s(A) = \infty$ for all $s$. We will only consider separable metric spaces in this thesis.
\end{remark}

Notice that on $\R^n$ with the Euclidean metric, $\mathcal H^n$ is, up to some constant, equivalent to $\mathcal L^n$. Furthermore, as one should expect, $\dim \R^n = n$, and thus if $A\subset \R^n$ we have $\dim A \leq n$. Hence, for any Borel set $A\subset \R^n$ with positive Lebsegue measure, $\dim A = n$. However, it is \textit{not} the case that having full dimension implies having positive Lebesgue measure. We quickly sketch how to see this over $\R$. 

Firstly, one can note that there exists Cantor sets with any fractional Hausdorff dimension of the form $\frac{\log j}{\log k}$ for integers $j\geq 1$ and $k\geq 2$. For example, it is a standard exercise to show the standard middle-thirds Cantor set $\mathcal C \subset \R$ has $\dim \mathcal C = \frac{\log 2}{\log 3}$, see \cite[Example 2.7]{FalconerBook}. Furthermore, Hausdorff dimension interacts nicely with countable unions.

\begin{lemma}\label{countablestability}
Let $(X,d)$ be a separable metric space. Then, given a countable collection of Borel sets $A_i \subset X$,
\[
\dim \left(\bigcup_i A_i\right) = \sup_i \dim A_i.
\]
\end{lemma}

\begin{remark}
    This property is referred to as the \textbf{countable stability} of Hausdorff dimension. It is perhaps interesting to note that this property is what makes Hausdorff dimension so nice; there are other notions of dimension for which this statement fails. For example, for those familiar with Minkowski dimension, one can prove that the set 
    \[
    F = \left\{0, 1,\frac{1}{2}, \frac{1}{3}, \frac{1}{4},\cdots \right\}\subset \R
    \]
    has Minkowski dimension $\frac{1}{2}$ whereas $\dim F = 0$, see \cite[Example 3.5]{FalconerBook}.
\end{remark}

Hence, we can let $A_i \subset \R$ be Cantor sets with Hausdorff dimension $\frac{\log i}{\log(i+1)}$ for all $i\geq 1$ and let $A= \bigcup_i A_i$. By the definition of Hausdorff dimension, it follows that $\mathcal H^1(A_i) = 0$ for all $i\geq 1$. Thus, while $\dim A = \sup_i \frac{\log i}{\log(i+1)} = 1$ by Lemma \ref{countablestability}, by countable subadditivity of measures
\[
\mathcal H^1(A) \leq \mathcal \sum_i \mathcal H^1(A_i) = 0.
\]
This example shows that having full Hausdorff dimension does not imply having positive measure. In this way, Hausdorff dimension is a more refined notion of size in comparison to Lebesgue measure. 

As such, when it comes to Kakeya sets and answering Question \ref{q:vagueKakeya}, we have the following major open conjecture.

\begin{conjecture}[Kakeya Conjecture] \label{kakeyaconj}
    Every Kakeya set $K \subset \R^n$ has \[\dim K = n.\]
\end{conjecture}

\begin{remark}
    Of course, if $\mathcal H^n(K)>0$ this statement is trivial by the definition of Hausdorff dimension.
\end{remark}

One \textit{very} heuristic way to think about this conjecture is by parameter counting: we have an $(n-1)$-dimensional set of directions $\theta \in \mathbb{S}^{n-1}$, and a $1$-dimensional line segment $\ell_\theta$ in each direction. However, don't be fooled by the simplicity of this heuristic. The conjecture has only been resolved in dimensions 2 and 3---in $\R^2$ by Davies in 1971 \cite{Davies71} and in $\R^3$ by Wang--Zahl in 2025 over 50 years later \cite{WangZahlKakeya}. For a detailed background on the Kakeya conjecture, see \cite{Wolff99,KatzTaoKakeyaSurvey} as well as references within Wang and Zahl's paper.

We included the brief discussion of Kakeya sets here for a few reasons. Firstly, though we will not be discussing these sets and the tools analysts have developed to study them much further herewithin, let us suffice it to say that the study of Kakeya sets is of great importance in harmonic analysis and fractal geometry. Secondly, for the purposes of this thesis, some of the ideas seen in the above discussion will lend themselves nicely to the following chapter on projection theory. In short (over $\R^2$):

\begin{itemize}
    \item \textbf{Orthogonal projections}: We will consider \textit{exceptional set estimates}---size estimates on the set of subspaces $\theta$ such that orthogonal projection of a set onto the line in direction $\theta$ is \textit{smaller} than ``expected.'' By ``expected'' here, we mean that such \textit{exceptional} directions are quite rare, and in particular have null Lebesgue measure on $\mathbb{S}^1$. Hence, we will need Hausdorff dimension as our corresponding notion of size.
    \item \textbf{Furstenberg sets}: We will discuss a fractal version of Kakeya sets, where instead of taking a line segment in every direction $\theta \in \mathbb{S}^1$, for all $\ell$ in a $\geq t$-dimensional set of affine lines $\mathcal L$, we will take a $s$-dimensional subset of $\ell$. The question here being: for each $s$ and $t$, how small can such an ``$(s,t)$-Furstenberg set'' be?
    \item \textbf{Radial projections}: We will similarly discuss how radial projection maps distort the Hausdorff dimension of the set you are projecting. In particular, in this final section of Chapter \ref{ch:Topics} we will highlight how ideas from orthogonal projections and Furstenberg sets were used in the study of radial projections and vice versa.
\end{itemize}

\noindent As we can see from the above, we will need to define Hausdorff dimension, and thus metrics, on the sphere, subspaces of $\R^n$, and affine planes in $\R^n$. We define the necessary metric spaces now and give the main things to note about each for the purposes of the thesis.\footnote{For a more detailed exposition of these metric spaces, see \cite[Chapter 3]{Mattila95}.} Let $n\geq 2$ and $1\leq k \leq n-1.$

\begin{itemize}
    \item \textbf{The Grassmannian} $\mathcal G(n,k)$. We let $\mathcal G(n,k)$ denote the Grassmannian: the set of $k$-dimensional linear subspaces in $\R^n$. It is a standard exercise in manifold theory to show that $\mathcal G(n,k)$ is a compact smooth $k(n-k)$-dimensional manifold. In particular, $\mathcal G(n,k)$ is a separable metric space, and in fact its metric can be defined via orthogonal projections! Recall, given $V\in \mathcal G(n,k)$, let $P_V:\R^n \to \R^k$ denote the orthogonal projection onto $V$. Then, the metric on $\mathcal G(n,k)$ is given by \[d_{\mathcal G}(V,V') = \lVert P_V - P_{V'} \rVert_{\mathrm{op}} := \sup_{|x| = 1,x\in \R^n} |(P_V-P_{V'})x|.\] For intuition purposes, it may be helpful to think of $\mathcal G(n,1)$ (lines through the origin) as $\mathbb{S}^{n-1}$ up to antipodal points (as such linear subspaces intersect the unit sphere at antipodal points). Similarly, one can think of $V\in \mathcal G(n,n-1)$ (a hyperplane through the origin) in terms of its orthogonal subspace $V^\perp \in \mathcal G(n,1).$ We let $\gamma_{n,k}$ denote the invariant probability measure on $\mathcal G(n,k)$.

    \item \textbf{The Affine Grassmannian} $\mathcal A(n,k)$. We let $\mathcal A(n,k)$ denote the \textit{affine Grassmannian}: the set of $k$-dimensional affine planes in $\R^n$. Every affine flat $U\in \mathcal A(n,k)$ can be identified as a subspace $W\in \mathcal G(n,k)$ and a unique translation vector $x \in W^\perp \simeq \mathbb{R}^{n-k}$ where $U = W + x$. As such, it should be unsurprising that $\mathcal A(n,k)$ is a $\dim \mathcal G(n,k) + \dim \mathbb{R}^{n-k} = (k+1)(n-k)$-dimensional smooth manifold with metric:
    \[
    d_{\mathcal A}(U,U') = \lVert P_W - P_{W'} \rVert_{\mathrm{op}} + \lvert x- x' \rvert,
    \]
    where again $U = W + x$ and $U' = W'+x'$. We let $\lambda_{n,k}$ denote the invariant measure on $\mathcal A(n,k)$.
    
    \item \textbf{The Unit Sphere} $\mathbb{S}^{n-1}$. The unit sphere is a standard $(n-1)$-dimensional manifold, with the distance between points on the sphere being given by geodesic distance. The main reason we note the unit sphere here is that, as we mentioned on the previous bullet point, $\mathcal G(n,1)$ can be identified with $\mathbb{S}^{n-1}/\pm$. In fact, this has natural consequences in regards to Hausdorff dimension. In particular, consider $\mathcal H^s_{\mathcal G}$ is the $s$-dimensional Hausdorff measure defined on $\mathcal G(n,1)$ and similarly consider $\mathcal H^s_{\mathbb{S}}$. Then, given $A\subset \mathbb{S}^{n-1}$, it follows that $\mathcal H^s_{\mathcal G}(A/\pm) \sim \mathcal H^s_{\mathbb{S}}(A)$. The same is true in the converse direction. In particular, let $\ell_\theta \in \mathcal{G}(n,1)$ denote a line in direction $\theta \in \mathbb{S}^{n-1}$. Then, given $B\subset \mathcal{G}(n,1)$, let $B^\ast = \{\pm\theta : \ell_\theta\in B\}$. Then, $\mathcal H^s_{\mathcal G}(B) \sim \mathcal H^s_{\mathbb{S}}(B^\ast)$. In particular, as one should expect intuitively, an $s$-dimensional subset of the sphere can be naturally identified with an $s$-dimensional subset of the Grassmannian in a geometrically intuitive way. We let $\sigma^{n-1}:= \sigma$ denote the invariant probability measure on $\mathbb{S}^{n-1}$.
\end{itemize}

For our purposes, there are two main things of note. Firstly, each of these spaces (as separable metric spaces) have a notion of Borel sets and $s$-dimensional Hausdorff measures (with respect to each metric) such that $\mathcal H^{n-1}_{\mathbb{S}} \sim \sigma$, $\mathcal H^{k(n-k)}_{\mathcal G} \sim \gamma_{n,k}$, and $\mathcal H^{(k+1)(n-k)}_{\mathcal A} \sim \lambda_{n,k}$. Secondly, given a $\de$-ball $B$ (on each corresponding metric space), we have $\sigma^{n-1}(B) \sim \de^{n-1}$, $\gamma_{n,k}(B) \sim \de^{k(n-k)}$, and $\lambda_{n,k}(B)\sim \de^{(k+1)(n-k)}$. This is useful in emphasizing the connection between problems in the discrete (incidence theoretic) setting and their analogues in the continuum (measure theoretic) setting. To this end, we provide an extreme heuristic regarding Hausdorff dimension.

\begin{heuristic}[Note: Extreme] \label{hr:extreme}
    Let $(X,d)$ be a separable metric space, and let $A\subset X$ be a set with $\dim A = a$. Then, \textbf{heuristically}, for sufficiently small $\de$, it takes $\sim \de^{-a}$-many $\de$-balls to cover $A$.
\end{heuristic}

To be clear, this heuristic is not generally true, but it is true for sets with equal Hausdorff and \textit{Minkowski dimension}. In fact, this is roughly how one can define Minkowski dimension, though we don't go into this further here. Heuristic \ref{hr:extreme} (which can sometimes be made more precise in practice via $(\de,s)$-sets and $\de$-discretization) can be utilized throughout the subsequent chapter to motivate sharp conjectures and methods in projection theory. 

To see how one can do this, consider Beck's theorem (Theorem \ref{thm:introBeck}) which states that under nonconcentration conditions on a finite set $X \subset \R^2$, the set of lines that contain $\geq 2$ points of $X$, $\mathcal L(X)$, satisfies $|\mathcal L(X)| \gtrsim |X|^2$. Hence, roughly every two points of $X$ determine a line in $\mathcal L(X)\subset \mathcal A(2,1)$. Suppose instead that $X \subset \R^2$ is a Borel set. Applying Heuristic \ref{hr:extreme}, it takes $\sim \de^{-\dim X}$-many $\delta$-balls to cover $X$, and under the right nonconcentration conditions we should expect that any two $\delta$-balls covering $X$ should determine a $\delta$-neighborhood of a line in $\mathcal L(X)$. In other words, it should take roughly
\[
\binom{\de^{-\dim X}}{2} \sim \de^{-2\dim X}
\]
many $\delta$-tubes to cover $\mathcal L(X)$, and thus it is reasonable to expect $\dim \mathcal L(X) \geq 2\dim X$. Of course, this is nowhere \textit{near} a proof of the continuum version of Beck's theorem (Theorem \ref{thm:introOSWBeck}), but this heuristic of going from points and lines to $\delta$-balls and $\delta$-tubes, and thusly cardinality to Hausdorff dimension, highlights how the discrete analogues of problems in projection theory can be used to motivate conjectures in the continuum setting.

\chapter{Topics in Projection Theory} \label{ch:Topics}
In this chapter, we discuss orthogonal projections, Furstenberg sets, and radial projections. We also provide a literature review of some key breakthroughs for each of these topics over the past 30 years. 

The goal of this chapter is not to provide a detailed account of each major paper in this area. Rather, the purpose of this chapter is two-fold. The first purpose is to emphasize the connection between incidence geometry and these topics by providing key examples from the discrete world that motivate proof techniques and conjectures in the continuum. The second is to state the main results and Hausdorff-dimensional tools we will need in the study applications of projection theory in Chapter \ref{ch:problems}. 

\section{Orthogonal Projections} \label{sec:ESE}

\begin{heuristic} \label{hr:Marstrand}
    Large objects typically cast large shadows.
\end{heuristic}

This heuristic statement is a driving force for a number of problems in projection theory, and orthogonal projections are a prime example. Let us begin by making this ``Marstrand-type'' heuristic precise. Recall that given $V \in \mathcal G(n,k)$, we let $P_V: \R^n \to V\simeq \R^k$ denote the orthogonal projection onto $V$. If $k = n-1$, we will sometimes use $\theta^\perp$ and $\mathbb{S}^{n-1}$ in place of $V$ and $\mathcal{G}(n,n-1)$ respectively. More precisely, given $\theta \in \mathbb{S}^{n-1}$, $P_{\theta^\perp}: \R^n \to \theta^\perp$ projects onto the $(n-1)$-dimensional subspace with normal direction $\theta.$ 

Suppose $X \subset \R^n$ is a finite set of points. Notice (in alignment with Heuristic \ref{hr:Marstrand}), for all but finitely many $\theta \in \mathbb{S}^{n-1}$, $|P_{\theta^\perp}(X)| = |X|$---as large as it \textit{can be}. To see this, recall the definition of the \textit{direction set} of $X$:
\[
S(X) := \left\{\frac{x-y}{|x-y|} : x,y\in X, x\neq y\right\} \subset \mathbb{S}^{n-1}.
\]
Of course, $S(X)$ is a finite set since $X$ is. Furthermore, given $\theta \in S(X)^c$ it follows that $|P_{\theta^\perp}(X)| = |X|$. If this weren't the case, this would imply that there is some fiber of $P_{\theta^\perp}$ that is at least 2-rich which contradicts $\theta \in S(X)^c.$ This argument implies the following discrete statement.

\begin{proposition} \label{prop:discreteMarstrand}
    Given a finite set of points $X\subset \R^n$, for all but finitely many $\theta \in \mathbb{S}^{n-1}$,
    \[
    |P_{\theta^\perp}(X)| = |X|.
    \]
    Moreover, $|P_{\theta^\perp}(X)|<|X|$ if and only if $\theta \in S(X)$. 
\end{proposition}

In the Hausdorff-dimensional setting, the continuum analogue of Proposition \ref{prop:discreteMarstrand} is given by Marstrand's projection theorem.

\begin{theorem}[Marstrand's Projection Theorem \cite{Marstrand54,Mattila75MarstrandMattilaProj}] \label{MARSTRAND}
    Let $X \subset \R^n$ be  Borel. Then, for $\gamma_{n,k}$-almost every $V\in \mathcal G(n,k)$
    \begin{equation}\label{eqn:MarstrandEqual}
    \dim P_V(X) = \min\{\dim X, k\}.
    \end{equation}
\end{theorem}

\begin{remark}
The above result was obtained in the plane by Marstrand in 1954 \cite{Marstrand54} and generalized to higher dimensions by Mattila in 1975 \cite{Mattila75MarstrandMattilaProj}. As such, Theorem \ref{MARSTRAND} is also referred to as the Marstrand--Mattila projection theorem. Furthermore, given $\dim X >k$, they more strongly showed that $\mathcal L^k(P_V(X)) >0$ for $\gamma_{n,k}$-almost every $V \in \mathcal G(n,k)$.
\end{remark}

Note that $\min\{\dim X, k\}$ is as large as $\dim P_V(X)$ \textit{could be}. The $\dim X$ term follows as $P_V$ is a Lipschitz map (see \cite[Corollary 2.4.a]{FalconerBook}); the $k$ term follows as $P_V(X) \subset V \simeq \R^k$. In other words, Marstrand's projection theorem gives that $\dim P_V(X)$ is \textit{typically} as large as it \textit{can be}. We refer to such projection theoretic statements as being \textit{Marstrand-type}. 

Though Theorem \ref{MARSTRAND} holds for almost every $V \in \mathcal G(n,k)$, it is an interesting question to quantify how large the set of $V$ such that \eqref{eqn:MarstrandEqual} fails \textit{can be}. This question lies at the heart of \textit{exceptional set estimates}.

\begin{heuristic}\label{hr:exceptional}
    It is \textit{exceptional} for large objects to cast small shadows.
\end{heuristic}

Let's begin yet again with the discrete setting. Let $X \subset \R^n$ be finite and
\[
E_s(X) = \{\theta \in \mathbb{S}^{n-1}: |P_{\theta^\perp}(X)| < s\}.
\]
Notice by Proposition \ref{prop:discreteMarstrand}, if $s>|X|$, $E_s(X)$ is infinite. Otherwise, $E_s(X)$ is finite (as $E_s(X) \subset S(X)/\pm$ for all $s\leq |X|$) and moreover we have:

\begin{proposition}\label{discreteESE}
Let $X \subset \R^n$ be a finite set of points.
Then,
\begin{itemize}
    \item[a)] (Product Set) $| E_{|X|^{1/2}}(X)| \leq 1$,
    \item[b)] (Cauchy--Schwarz) $|E_s(X)| \lesssim s$ for $s\leq |X|/2$, and
    \item[c)] (Szemer\'edi--Trotter) $|E_s(X)| \lesssim \max\{s^2 |X|^{-1}, 1\}$ for $s\leq |X|/2$.
\end{itemize}
\end{proposition}

\begin{remark}
    Note that the third bound is the strongest and sharp for all $s$. Even so, we note all three bounds as their continuum analogues are important results towards the sharp bound that will be discussed later.
\end{remark}

\begin{proof}[Proof of Proposition \ref{discreteESE}]
Notice by definition that for all $\theta \in E_s(X)$, it takes at most $s$-many lines in direction $\theta$ to cover $X$. Denote this set of fibers $\mathcal L_\theta := P_\theta^{-1}(P_\theta(X))$, and let 
\[
\mathcal L = \bigcup_{\theta \in E_s(X)} \mathcal L_\theta.
\]

To see statement (a), suppose $\#E_{|X|^{1/2}}(X) \geq 2$, meaning that we have two distinct $\theta,\theta'$ and corresponding fibers $\mathcal L_\theta$, $\mathcal L_{\theta'}$ covering $|X|$ in $<|X|^{1/2}$-many parallel lines in directions $\theta^\perp$ and $(\theta')^\perp$ respectively. It follows that 
\begin{align*}
|X| &\leq |\mathcal L_\theta| |\mathcal L_{\theta'}| \\
&= |P_\theta^{-1}(P_\theta(X))| |P_{\theta'}^{-1} (P_{\theta'}(X))| \\
&< |X|^{1/2} |X|^{1/2} = |X|.
\end{align*}
This gives a contradiction.

We now set up the incidence arguments that give statements (b) and (c). For ease of notation, let $|E_s(X)| = t$. It follows that $|\mathcal L| < st$. We then consider the set of incidences between $X$ and $\mathcal L$. Firstly, notice that for each $x\in X$, there are $t$-many lines in $\mathcal L$ that contain $x$ given that i) $\mathcal L_\theta$ covers $X$, ii) $\mathcal L_\theta \cap \mathcal L_{\theta'} = \emptyset$ for all $\theta\neq \theta'$, and iii) $|E_s(X)| = t$. Hence, 
\begin{equation} \label{ESESetUp}
|X|t = \sum_{x\in X} |\mathcal I(x, \mathcal L)| = |\mathcal I(X,\mathcal L)|.
\end{equation}

From here, Proposition \ref{discreteESE}.(b) follows from Proposition \ref{sec2:prop:STPrelim} and (c) follows from Szemer\'edi--Trotter. As the proofs are similar, we only present the latter. Firstly, note that if $|X|/12 < s < |X|/2$, then $s \sim |X|$ and thus Proposition \ref{discreteESE}.(c) follows by (b). Hence, we assume $s \leq |X|/12.$ Applying Theorem \ref{SZEMEREDITROTTER} with $|\mathcal L|< st$ to \eqref{ESESetUp}, we obtain
\[
|X|t < 4(|X|^{2/3}s^{2/3}t^{2/3} + |X| +st).
\]
Notice that the third term can never be dominant if $s\leq |X|/12$. Doing case analysis on the first two terms gives
\[
|E_s(X)| = t \leq \max\{12^3 s^{-2} |X|,12\}
\]
proving the desired result.
\end{proof}

Each of the above bounds has an interesting continuum counterpart, the third of which is sharp for all ranges of $s$ and was resolved when $n = 2$ by Orponen--Shmerkin \cite{OrponenShmerkinABC} and Ren--Wang \cite{RenWang} last year. We discuss this further as we turn our attention to the continuum setting for the remainder of the section.

We begin by setting up some notation we will utilize throughout. Fix integers $n$ and $k$ with $1\leq k \leq n-1$, and let $X \subset \R^n$ be Borel. Then, given $0 < s \leq \min\{\dim X, k\}$, let 
\[
E_{s}^{n,k}(X) := \{V \in \mathcal G(n,k) : \dim P_V(X) < s\}.
\]
We will often drop the $n,k$ from the superscript if it is clear from context or unimportant for the discussion. Note that it is also interesting to explore variants of $E_s(X)$ where one replaces $\dim P_V(X) < s$ with another sense of ``size,'' i.e. null $\mathcal H^k$-measure, empty interior, or other notions of dimension (see e.g. \cite[Theorem 5.12]{Mattila15}), though we do not discuss this further here.

By Marstrand's projection theorem, we have $\gamma_{n,k}(E_{s}^{n,k}(X)) = 0$. (One might a priori be concerned about measurability of $E_s(X)$, but rest assured: given $X$ is Borel, so too is $E_s(X)$. See e.g. \cite[Statement A]{Kaufman68}.) As such, how else might we quantify the size of $E_s(X)$? Just as we do with Kakeya sets, we turn to Hausdorff dimension! Exceptional set estimates seek to find bounds on $\dim E_{s}^{n,k}(X)$ (as a subset of the Grassmannian) in terms of $\dim X,s,n,$ and $k$. By non-trivial here, we mean upperbounds \textit{strictly less than} the dimension of the ambient manifold, i.e. $\dim \mathcal G(n,k) = k(n-k)$. We note here that such bounds imply Marstrand's projection theorem as a corollary, and in this way Heuristic \ref{hr:exceptional} is stronger than Heuristic \ref{hr:Marstrand}.

\begin{proposition} \label{MarstrandFromNontrivialESE}
    Theorem \ref{MARSTRAND} follows as a corollary from nontrivial exceptional set estimates. In particular, if $X \subset \R^n$ is Borel and \[\dim E_{s}^{n,k}(X) < k(n-k) \quad \quad \text{for all } 0< s< \min\{\dim X, k\},\] 
    then 
    \[
    \dim P_V(X) = \min\{\dim X, k\} \quad \quad \text{for } \gamma_{n,k}\text{-a.e. $V \in \mathcal G(n,k).$}
    \]
\end{proposition}

\begin{proof}
Let $\dim X = a$. We want to show that 
\[
E_{\min\{a,k\}}(X) := \{V : \dim P_V(X) < \min\{a,k\}\}.
\]
is $\gamma_{n,k}$-null. To do so, note that 
\[
E_{\min\{a,k\}}(X) = \bigcup_{i = \lceil \min\{a,k\}^{-1}\rceil}^\infty\{V : \dim P_V(X) < \min\{a,k\} - i^{-1}\}.
\]
In particular, each set on the right hand side is of the form $E_{s_i}(X)$ were $s_i = \min\{a,k\}-i^{-1}$. Given that $\dim E_{s_i}(X) < k(n-k)$ by assumption, it follows by the definition of Hausdorff dimension that $\mathcal H^{k(n-k)}(E_{s_i}(X)) = 0.$ Recalling that $\gamma_{n,k}\sim \mathcal H^{k(n-k)}$ and using countable subadditivity of measures, the claim follows.
\end{proof}

One of the first exceptional set estimates is due to Kaufman in 1968 over $\R^2$ \cite{Kaufman68}, which was generalized by Mattila to higher dimensions in 1975 \cite{Mattila75MarstrandMattilaProj}.

\begin{theorem}[Kaufman's Projection Theorem \cite{Kaufman68, Mattila15}] \label{KAUFMAN}
    Let $X\subset \R^n$ be Borel and fix $1\leq k \leq n-1$. Then for $0< s \leq \min\{\dim X, k\}$,
    \[
    \dim E_{s}^{n,k}(X) \leq k(n-k) - k + s.
    \]
\end{theorem}

We briefly make a few remarks on Kaufman's projection theorem. Firstly, when $k=n-1$, Theorem \ref{KAUFMAN} gives $\dim E_{s}(X) \leq s$ akin to the bound $\#E_s(X) \lesssim s$ we saw in the discrete setting. In this way, Proposition \ref{discreteESE}.(b) is a discrete analogue of Kaufman's projection theorem, though of course the proof of the continuum statement is harder to obtain. Secondly, notice that the upperbound from Kaufman's projection theorem is monotone increasing in the $s$ variable. On the one hand, we should expect as much from an exceptional set estimate given $E_s(X) \subset E_{s'}(X)$ if $s\leq s'$. On the other hand, by Heuristic \ref{hr:Marstrand}, one should also expect that (for fixed $s$) the larger $\dim X$ is the smaller $\dim E_s(X)$ should be. While we don't see this in Kaufman's projection theorem (which is uniform in $\dim X$), we see this in Falconer's exceptional set estimate from 1982 \cite{Falconer82}.

\begin{theorem}[Falconer's Projection Theorem \cite{Falconer82}] \label{FALCONER}
    Let $X\subset \R^n$ be Borel and fix $1\leq k \leq n-1$. Then, for $0 < s \leq \min\{\dim X, k\}$,
\[
\dim E_{s}^{n,k}(X) \leq \max\{k(n-k) + s - \dim X, 0\}.
\]
\end{theorem}

Standard proofs of Marstrand's, Kaufman's and Falconer's projection theorems are potential theoretic---making use of Frostman measures and energy integrals; see \cite[Theorems 4.1 and 5.1]{Mattila15} and \cite[Theorem 1.(ii)]{Falconer82}. In Section \ref{BGanProofs}, I present a Fourier analytic proof of Falconer's projection theorem over finite fields. The argument over finite fields is an adaptation of the argument due to myself and Gan \cite{BrightGan}. We discuss this further in the following section, though for now carry on with the literature review.

In 2012, D. Oberlin \cite[(1.8)]{Oberlin2s-aConjecture} conjectured that for Borel $X\subset \R^2$, 
\begin{equation} \label{sharpESEeqn}
\dim E_s(X) \leq \max\{2s-\dim X, 0\}.
\end{equation}
He obtained the statement for $\dim X \leq 1$ and $s = \frac{\dim X}{2}$ using restricted Radon transforms, which is what motivates this conjecture. Also notice that when $s=\dim X$ or $1$, this upperbound agrees with that from Kaufman's and Falconer's projection theorems respectively. One can further motivate this conjecture by thinking of Proposition \ref{discreteESE}.(c) as a discrete analogue to \eqref{sharpESEeqn}. Additionally, one can think of the case when $s = \frac{\dim X}{2}$ as the continuum counterpart to Proposition \ref{discreteESE}.(a). This particular special case was obtained by Bourgain in 2010 \cite{BourgainProjectionTheorem} for all ranges of $\dim X.$

\begin{theorem}[Bourgain's Projection Theorem \cite{BourgainProjectionTheorem}] \label{BOURGAIN}
    Let $X\subset \R^2$ be Borel with $0 < \dim X < 2$. Then,
    \[
    \dim \{\theta : \dim P_\theta(X) \leq \dim X/2\} = 0.
    \]
\end{theorem}

\begin{remark}
In fact, the above result is a corollary of a discretized result \cite[Theorem 3]{BourgainProjectionTheorem}. For the purposes of this thesis, and with the unfamiliar reader in mind, we omit the discretized theorems (presenting their Hausdorff dimensional consequences instead) and simply/heuristically note that such discretized results are more ``quantitative.'' Furthermore, note that if $\dim X = 0$ or $2$, $\dim P_\theta(X) = \dim X/2$ for all $\theta$. Hence, the range on $\dim X$ is natural in the above statement.
\end{remark}

As we will see throughout this chapter, Bourgain's projection theorem has led to massive breakthroughs in this area of geometric measure theory. As such, let's spend some time discussing the history of Theorem \ref{BOURGAIN}. A key result towards Bourgain's projection theorem is the following result proved by Edgar and Miller \cite{EdgarMiller} and conjectured by Erd\H{o}s and Volkmann \cite{ErdosVolkmann}:

\begin{theorem}[\cite{EdgarMiller}] \label{ctmsubringstatement}
    A Borel subring of $\R$ cannot have Hausdorff dimension strictly between 0 and 1.
\end{theorem}

We note that Theorem \ref{ctmsubringstatement} does not hold over $\mathbb{C}$ and $\mathbb{F}_{p^2}$ (for a prime $p$) which have ``half-dimensional'' subrings (namely $\R$ and $\mathbb{F}_p$ respectively). The existence of such subrings often leads to challenging and important obstacles in the study of incidence geometry and geometric measure theory, as results that hold in one geometry may not (naively) hold in another.

Around the same time as Edgar and Miller's result in 2003, Bourgain  \cite{BourgainSubring} independently proved Theorem \ref{ctmsubringstatement} from a discretized statement conjectured by Katz and Tao \cite{KatzTaoSubringConjecture}. A few years later in 2010, Bourgain proved Theorem \ref{BOURGAIN}, the discretized statement of which recovers his discretized subring result from 2003 and much more.

Katz and Tao showed in \cite{KatzTaoSubringConjecture} that their conjecture (that Bourgain \cite{BourgainSubring} resolved) has implications for Furstenberg sets (a fractal analogue of Kakeya sets, see Section \ref{sec:Furstenberg}). This connection garnered much attention, and in 2022, Orponen and Shmerkin \cite{OrponenShmerkinEpsilonImprove} used Bourgain's projection theorem to generalize the result of Katz--Tao and Bourgain. We cannot give justice to the importance of this result until the subsequent section, but for now note that as a consequence of Orponen and Shmerkin's work they obtained the following $\epsilon$-improvement to Kaufman's projection theorem.

\begin{theorem}[\cite{OrponenShmerkinEpsilonImprove}] \label{OSEpsilonProj}
    For every $0 < s < t \leq 1$, there exists $\epsilon:= \epsilon(s,t)>0$ such that for all $X\subset \R^2$ Borel with $\dim X = t$,
    \[
    \dim E_s(X) \leq s-\epsilon.
    \]
\end{theorem}

This $\epsilon$-improvement to Kaufman's projection theorem may seem like a small step towards proving \eqref{sharpESEeqn}, but in fact it was one of the key breakthroughs towards Oberlin's conjecture for exceptional set estimates. The $\epsilon$-improvement in \cite{OrponenShmerkinEpsilonImprove} ultimately led to the sharp estimates for exceptional sets in the plane in work of Orponen--Shmerkin and Ren--Wang \cite{OrponenShmerkinABC,RenWang}.

\begin{theorem}[\cite{OrponenShmerkinABC,RenWang}] \label{OSRW}
    Let $X\subset \R^2$ Borel and $s \leq \min\{\dim X,1\}$. Then,
    \[
    \dim \{\theta : \dim P_\theta(X) \leq s\} \leq \max\{2s-\dim X,0\}.
    \]
\end{theorem}

We conclude this section with a brief discussion of select higher dimensional results. After the work of Orponen--Shmerkin--Wang \cite{OSW}, Ren obtained a higher-dimensional $\epsilon$-improvement to Kaufman's projection theorem (for hyperplanes) in 2023 \cite{RenRadProj}.

\begin{theorem}[\cite{RenRadProj}] \label{RenTheoremKaufman}
    For every $k<s<t\leq n$, there exists $\epsilon:= \epsilon(s,t)>0$ such that for all $X\subset \R^n$ Borel with $\dim X = t$, 
    \[
    \dim \{\theta \in \mathbb{S}^{n-1} : \dim P_{\theta^\perp}(X) \leq s\} \leq s-\epsilon.
    \]
\end{theorem}

Theorem \ref{RenTheoremKaufman} was then used to obtain radial projection results in higher dimensions following the methodology of Orponen--Shmerkin--Wang (see Section \ref{sec:RadProj}). A year later in 2024, Cholak et al. \cite{CholakEtAl} obtained sharp exceptional set estimates for orthogonal projections onto lines and hyperplanes.

\begin{theorem}[\cite{CholakEtAl}] \label{cholak}
    Let $X\subset \R^n$ be Borel. Then, for all $s\leq \min\{\dim X,i\}$ (for $i=1$ and $n-1$ respectively),
    \begin{align*}
        \dim E_{s}^{n,1}(X) &\leq \max\{n-2-\lfloor \dim X-s\rfloor, n-2+2s-\dim X \} \\
        \dim E_{s}^{n,n-1}(X) &\leq \max\{2s-\dim X, \lceil s \rceil -1\}.
    \end{align*}
    Moreover, these bounds are sharp.
\end{theorem}

\begin{remark}\label{miniOPremark}
    As noted in \cite[Remark 1.5]{CholakEtAl}, $\dim E_{s}^{n,k}(X) = 0$ for all $s\leq \max\{\dim X-n+k,0\}$. In particular, notice that for all $X\subset \R^n$ and $V\in \mathcal G(n,k)$, $X\subset P_V(X) \times V^\perp$, so $\dim P_V(X) \geq \dim X -(n-k)$ for all $V\in \mathcal G(n,k)$. I.e., if $s \leq \max\{\dim X-n+k,0\}$, we have $E_s(X) =\emptyset.$
\end{remark}

To prove the upperbounds in Theorem \ref{cholak}, Cholak et al. used an inductive approach, making use of slicing-type theorems and Kaufman's projection theorem. To construct sharp examples for $k = 1$ and $n-1$, the authors embedded sharp examples from $A\subset \R^{n-1}$ into $\R^n$ by considering $A\times \{0\}$ and $A \times \R$. In fact, this method gives lowerbounds on the quantity 
\[
T_{n,k}(a,s) = \sup_{X\subset \R^n, \dim X = a} \dim \{V \in \mathcal G(n,k) : \dim P_V(X) < s\}
\]
that they conjecture to be sharp for \textit{all} $k$; see \cite[Conjecture 1.3]{CholakEtAl}. 

To this end, it is worth noting that the sharpness of Theorem \ref{OSRW} can be shown by considering sharp examples for Szemer\'edi--Trotter. Roughly speaking, one can take sharp examples from Szemer\'edi--Trotter, thicken every point to a $\de$-ball, and (heuristically) apply Proposition \ref{discreteESE}.(c). Of course, making such a sketch rigorous can be a bit technical---one of my favorite write ups can be seen in \cite[Section 3]{FraserdeOrellana}---but this is just to show yet another reason why incidence geometry is so fruitful in the study of geometric measure theory (and vice versa).  On the one hand, insight into open problems in geometric measure theory can be gained by grappling with difficult open problems in incidence geometry. On the other hand, incidence geometry can be used to create and prove discrete analogues of problems in geometric measure theory, and sharp examples from incidence geometry can be discretized to prove the sharpness of results in the continuum.

\subsection{The High-Low Method: Falconer's Bound} \label{BGanProofs}

In this section, we prove a finite field analogue of Falconer's projection theorem (Theorem \ref{FALCONER}) using a Fourier analytic method known as the high-low method. Before getting into the proof, we begin with a few remarks about harmonic analysis over finite fields $\mathbb{F}_p^n$ for $p$ prime.

As we have tried to elucidate in the previous subsection, there is a very close tie between discrete problems in incidence geometry and continuum problems in geometric measure theory. Exploring the interplay within these two areas of mathematics has proven deeply fruitful, and finite fields have served as a very helpful playground for such exploration. 

Studying harmonic analytic problems over $\mathbb{F}_p^n$ is natural for a number of reasons. Perhaps most importantly, the necessary geometric objects are defined over finite fields (e.g. lines and planes), allowing us to create analogues of problems in $\R^n$ over $\mathbb{F}_p^n.$ Furthermore, studying such problems over $\mathbb{F}_p^n$ allows us to ask if our proofs over $\R^n$ truly require properties of Euclidean space or simply require properties of geometric objects that hold in vast generality. E.g., are we using facts that are specific to $\R$, or geometric properties such as ``two lines intersect in at most one point''? Questions like this implore us to explore our proof methods quite closely and look for new ideas, proof methods, and results in \textit{both} directions. To this end, it is important to note that there are open problems in geometric measure theory and harmonic analysis over $\R^n$ that are resolved over $\mathbb{F}_p^n$ (such as the Kakeya problem \cite{Dvir08}), and open problems over $\mathbb{F}_p^n$ that are resolved over $\R^n$ (such as the Furstenberg set problem and sharp exceptional set estimates for orthogonal projections when $n = 2$ \cite{RenWang}).

Whether or not an argument from geometric setting transfers over to the other (perhaps with slight modification), it is interesting to ask \textit{why}. On the one hand, if a proof doesn't transfer over it may be interesting to search for a  counterexample to the analogous statement, and if you can't find a counterexample, it may be interesting to search for a new proof entirely. On the other hand, if a proof does transfer over, it can lead the researcher to find improvements to their own methods.

Thankfully, when it comes to the high-low method approach to Falconer's projection theorem over $\mathbb{F}_p^n$, the proof methodology of myself and Gan over $\R^n$ \cite{BrightGan} carries over quite nicely and is is much technically simpler due to the geometry of $\mathbb{F}_p^n$. The proof is also recorded and refined in a subsequent paper of myself and Gan over $\mathbb{F}_q^n$ for prime powers $q = p^r$ \cite[Lemma 21]{BrightGan2}, though we present the proof over $\mathbb{F}_p^n$ as the main geometric concepts needed for the proof are conveyed in this setting.

We will begin with presenting the necessary geometry and Fourier analysis over $\mathbb{F}_p^n$, and then prove a finite field version of Falconer's projection theorem (Theorem \ref{DISCFalconer}). The high-low method can similarly obtain a radial projection result over $\mathbb{F}_p^n$ (Theorem \ref{highlowLPT}) which we present in Section \ref{BGanRadProjProof}.
\vspace{.2cm}

\textbf{\textit{Geometry over $\mathbb{F}_p$}}: We begin with defining affine lines.

\begin{definition}[Lines in $\mathbb{F}_p^n$] \label{FpAffineLines}
An \textit{affine line} in $\mathbb{F}_p^n$ is a set of the form
\[
L(\theta,b) = \{\theta t + b : t \in \mathbb{F}_p\}
\]
for some $\theta ,b\in \mathbb{F}_p^n$ with $\theta \neq 0$.
\end{definition}

If the point $0 \in L(\theta,b)$, then we say that $L(\theta,b)$ is a 1-dimensional linear subspace of $\mathbb{F}_p^n$. It is a standard counting exercise to show there are 
\[ \frac{p^n - 1}{p-1} = p^{n-1} + p^{n-2} + \cdots + p + 1 \sim_n p^{n-1}
\] 
unique 1-dimensional subspaces of $\mathbb{F}_p^n$, and thus the number of unique affine lines is $\sim_n p^{2(n-1)}$. 

\begin{remark}
Note that here, building on Notation \ref{SimNotation}, the statement $A\sim_n B$ implies that there exists a constants $C_1(n),C_2(n)>0$ such that 
\[
C_1(n) B \leq  A \leq C_2(n) B,
\]
and these constants are, importantly, independent of other variables in the context of the statement. For instance, letting $C_1 = 1$ and $C_2 = n$, we have
\[
C_1 p^{n-1} \leq p^{n-1} + p^{n-2} + \cdots + p + 1 \leq C_2 p^{n-1},
\]
and notably $C_1,C_2$ are independent of $p$ for all $n.$
\end{remark}

The former statement can be obtained by noting that $L(\theta,b) = L(\theta',b)$ if and only if $\theta = \lambda \theta'$ for some $\lambda \in \mathbb{F}_p$. The latter statement can be obtained by noting that each affine line $L(\theta,b) = L(\theta, 0) + b$ for some unique translation vector $b \in \theta^\perp$ where 
\[
\theta^\perp = \{x \in \mathbb{F}_p^n : x\cdot \theta = 0\} \simeq \mathbb{F}_p^{n-1}
\]
and
\[
x\cdot y = x_1 y_1 + x_2 y_2 + \cdots + x_n y_n \pmod p.
\]
Let $\mathcal{G}(\mathbb{F}_p^n, 1)$ denote the set of $1$-dimensional linear subspaces in $\mathbb{F}_p^n$. For each $\theta \in \mathcal G (\mathbb{F}_p^n, 1)$, we can associate a $0\neq \tilde{\theta} \in \mathbb{F}_p^n$ such that $\theta = L(\tilde{\theta},0)$. We use this identification throughout, and by abuse of notation we let $\tilde{\theta} = \theta$.

\begin{remark} \label{FiniteFieldSubspaces}
More generally, the number of $k$-dimensional linear subspaces of $\mathbb{F}_p^n$ is 
\[
\binom{n}{k}_p := \frac{(p^n - 1) (p^n - p)\cdots (p^n - p^{k-1})}{(p^k - 1) (p^k - p) \cdots (p^k - p^{k-1})} \sim \frac{p^{nk}}{p^{k^2}} = p^{k(n-k)},
\]
and the number of $k$-dimensional subspaces that contain a fixed $\ell$-dimensional subspace in $\mathbb{F}_p^n$ is $\binom{n-\ell}{k - \ell}_p \sim p^{(k-\ell)(n-k)}$ for $\ell \leq k$. Let $\mathcal{G}(\mathbb{F}_p^n, 1)$ denote the set of $k$-dimensional linear subspaces in $\mathbb{F}_p^n$. We will need this notation when discussing Falconer's projection theorem over $\mathbb{F}_p^n$ and orthogonal projections onto $k$-planes in $\mathcal{G}(\mathbb{F}_p^n,k)$ for $1 \leq k \leq n-1$. Similarly, let $\mathcal A(\mathbb{F}_p^n,k)$ denote affine $k$-planes over $\mathbb{F}_p^n$.
\end{remark}

We now define orthogonal projection onto 1-dimensional subspaces.

\begin{definition}
    Given $\theta \in \mathcal G(\mathbb{F}_p^n, 1)$, let $P_\theta : \mathbb{F}_p^n \to \mathbb{F}_p$ be given by 
    \[
    P_\theta(x) = x\cdot \theta
    \]
    where $x\cdot \theta := x_1 \theta_1 + \cdots + x_n \theta_n.$
\end{definition}

When $n = 2$, it may be instructive for the reader to parameterize $\theta \in \mathcal{G}(\mathbb{F}_p^n,1)$ via the directions $\{(1,t) : t \in \mathbb{F}_p\}\cup \{(0,1)\}$ and show that the level sets of $P_{(1,t)} : \mathbb{F}_p^2 \to \mathbb{F}$ are affine lines in direction $(1,-t^{-1})$, and the level sets of $P_{(0,1)}$ are affine lines in direction $(1,0)$.\footnote{It may be instructive for the reader to generalize this to higher dimensions, i.e. to find a set $\Theta \subset \mathbb{F}_p^n \setminus 0$ such that $L(\Theta,0) := \bigcup_{\theta \in \Theta}L(\theta, 0 ) = \mathcal G(\mathbb{F}_p^n, 1)$ and $|\Theta| = p^{n-1} + p^{n-2} + \cdots + 1$.} Geometrically, what this is saying is that the orthogonal complement to a line in direction $\theta$ in $\mathbb{F}_p^2$ is a 1-dimensional linear subspace $\theta^\perp := \{x\in \mathbb{F}_p^2 : x \cdot \theta = 0\}$. For $n\geq 2,$ 
\[
\theta^\perp:= \{x\in \mathbb{F}_p^n : x \cdot \theta = 0\} \in \mathcal G(\mathbb{F}_p^n, n-1).
\]
In this way, for $\theta \in \mathcal G(\mathbb{F}_p^n,1)$, $P_\theta: \mathbb{F}_p^n \to \mathbb{F}_p$ is the natural analogue to orthogonal projection in $\R^n.$

\begin{remark}
    More generally, given a $k$-plane $V$ of $\mathbb{F}_p^n$, we can define 
    \[
    V^\perp  := \{x\in \mathbb{F}_p : x\cdot v = 0 \text{ for all } v\in V\} \in  \mathcal{G}(\mathbb{F}_p^n, n-k),\]
    and let $P_V: \mathbb{F}_p^n\to V$ be such that $P_V(x) = y$ if and only if $x\in V^\perp + y$.
\end{remark}

\textbf{\textit{Fourier Analysis over $\mathbb{F}_p$}}: We first define the Fourier transform.

\begin{definition}[Fourier Transform in $\mathbb{F}_p^n$]
    Let $f: \mathbb{F}_p^n \to \mathbb{C}$. We define the \textit{Fourier transform} of $f$ to be the function $\hat{f}:\mathbb{F}_p^n\to \mathbb{C}$ given by 
    \[
    \hat{f}(\xi) := \sum_{x\in \mathbb{F}_p^n} f(x) e^{-2\pi i x\cdot \xi/p}.
    \]
\end{definition}

The Fourier transform is a deeply important tool throughout mathematics, so much so that we do not attempt to convey the weight of this object within this thesis. For our purposes, let us suffice to say that the Fourier transform nicely encapsulates symmetries of $\mathbb{F}_p^n$ as an abelian group. For instance, the Fourier transform behaves nicely with translations as can be see in the following lemma.

\begin{lemma} \label{FourierTranslation}
    Let $f: \mathbb{F}_p^n \to \mathbb{C}$, $v\in \mathbb{F}_p^n$, and let $T_vf(x) = f(x+v)$. Then, 
    \[
    \widehat{T_vf}(\xi) = e^{2\pi i v\cdot \xi/p} \hat{f}(\xi).
    \]
\end{lemma}

\begin{proof}
    This is a direct calculation via a change of variables. In particular, 
    \begin{align*}
    \widehat{T_vf}(\xi) &= \sum_{x\in \mathbb{F}_p^n} f(x+v) e^{-2\pi i x\cdot \xi/p} \\
    &= \sum_{y\in \mathbb{F}_p^n}f(y) e^{-2\pi i (y-v)\cdot \xi/p} \\
    &= e^{2\pi i v\cdot \xi/p}\hat{f}(\xi)
    \end{align*}
    where the second line follows by the change of variables $y = x+v.$ 
\end{proof}

We can also begin to see how the Fourier transform behaves with respect to affine $k$-planes. 

\begin{proposition} \label{FourierOfPlane}
    Let $V$ be an affine $k$-dimensional plane of $\mathbb{F}_p^n$. Then,
    \[
    \hat{\mathbf{1}_V}(\xi) = p^{k} e^{2\pi i b\cdot \xi/p}\mathbf{1}_{V^\perp}(x)
    \]
    where $b\in V^\perp$ is such that $V - b \in \mathcal G(\mathbb{F}_p^n,k).$
    In particular, $\supp\widehat{\mathbf{1}_V} \subset \mathbf{1}_{V^\perp}$.
\end{proposition}

\begin{proof}
By Lemma \ref{FourierTranslation}, it suffices to prove the statement for $k$-dimensional \textit{subspaces} $V\in \mathcal G(\mathbb{F}_p^n, k)$. We begin by proving the statement with $k = 1$. 

Let $\theta \in \mathcal G(\mathbb{F}_p^n,1)$ and let $L:= L(\theta, 0)$. Then, by definition, 
\[
\hat{\mathbf{1}_L}(\xi) := \sum_{x\in \mathbb{F}_p^n} \mathbf{1}_L(x) e^{-2\pi i x\cdot \xi/p} = \sum_{t\in \mathbb{F}_p} e^{-2 \pi i (\theta t) \cdot \xi /p}.
\]
Here, we use that $\mathbf{1}_L(x) = 0$ unless $x = \theta t$ for some $t\in \mathbb{F}_p$. If $\theta \cdot \xi = 0$, the claim is clear. Now, notice that if $\theta\cdot \xi \neq 0$, the above sum equals zero. This follows as if $\theta \cdot \xi \neq 0$, one can apply the change of variables $t \mapsto t' (\theta\cdot\xi)^{-1}$ (which is invertible) and obtain 
\[
\sum_{t\in \mathbb{F}_p}e^{-2 \pi i (\theta t) \cdot \xi /p} = \sum_{t' \in \mathbb{F}_p} e^{-2\pi i t'/p} = 0
\]
as we are summing over the $p$-th roots of unity. Hence, 
\[
\supp \hat{\mathbf{1}_V}:= \{\xi \in \mathbb{F}_p^n : \hat{\mathbf{1}_L}(\xi) \neq 0\}\subseteq \{\xi \in \mathbb{F}_p^n : \theta \cdot \xi = 0\} := \theta^\perp.
\]
The same proof idea works for larger values of $k$, which can be seen by noting that for $V\in \mathcal G(\mathbb{F}_p^n, k)$, then $V$ is spanned by $k$-many linearly independent vectors $v_1,\dots, v_k \in \mathbb{F}_p^n$. By definition we have
\[
\hat{\mathbf{1}_V}(\xi) = \sum_{x\in \mathbb{F}_p^n} \mathbf{1}_V(x) e^{-2\pi i x\cdot \xi} = \sum_{x\in V} e^{-2\pi i x\cdot \xi}.
\]
If $\xi \in V^\perp$, the claim is clear. Now, suppose $\xi \notin V^\perp$. Without loss of generality, assume $v_1\cdot \xi \neq 0$. It therefore follows that $\xi \notin v_1^\perp$. Furthermore, let $W = \mathrm{span}\{v_2,\dots, v_k\}$, and notice that for all $x\in V$, $x = w + tv_1$ for some $w\in W$ and $t\in \mathbb{F}_p$. Then,
\[
\sum_{x\in V} e^{-2\pi i x\cdot \xi} = \sum_{w\in W} e^{-2\pi i w\cdot \xi} \sum_{t\in \mathbb{F}_p} e^{-2\pi i tv_1 \cdot \xi} := \sum_{w\in W} \left(e^{-2\pi i w\cdot \xi} \hat{\mathbf{1}}_{L(v_1,0)}(\xi)\right) = 0.
\]
The last line follows by the case $k = 1$ and $\xi \notin v_1^\perp$, completing the proof. 
\end{proof}

This is one of the nice properties of dealing over finite fields: given $f : \mathbb{F}_p^n \to \mathbb{C}$ with $\supp f$ contained in a $k$-plane $V$, it follows that $\supp \hat{f}\subset V^\perp$. Over Euclidean space, one can define the Fourier transform of a (nice enough, e.g. Schwartz) function $f : \R^n \to \mathbb{C}$ as a function $\hat{f}:\R^n \to \mathbb{C}$ as 
\[
\hat{f}(\xi) := \int f(x) e^{-2\pi i x\cdot \xi}\,\mathrm dx.
\]
Then, if $f$ is compactly supported on a $\delta$-neighborhood of a $k$-plane $V$ (which can be thought about as a $1\times \dots \times 1 \times \delta$ rectangular prism centered around $V$), then $\hat{f}$ is \textit{essentially} supported on a $\delta^{-1}$-neighborhood of $V^\perp$ (i.e. a $1\times \dots \times 1 \times \de^{-1}$ rectangular prism centered around $V^\perp$). This is made precise by \textit{dual ellipsoids} which are discussed in Wolff's lecture notes \cite{WolffLectureNotes}.

This unfortunately is not precisely true without the word ``essentially'' here---in fact, it is a standard exercise in harmonic analysis to show that if $f,\hat{f}\in L^2(\R^n)$ are compactly supported,\footnote{Note that it is nontrivial to show that the Fourier transform is well-defined on $L^2(\R^n)$, let alone $L^p(\R^n)$ for $1\leq p < \infty$ (which requires the theory of tempered distributions).} then $f = 0$. This is the main reason why we present the high-low method over $\mathbb{F}_p^n$ as opposed to $\R^n$.

The last Fourier analytic statement we need is Plancherel's theorem.

\begin{theorem}[Plancherel] \label{PLANCHEREL}
    Let $f : \mathbb{F}_p^n \to \mathbb{C}$. Then, 
    \[
    \sum_{x\in \mathbb{F}_p^n} |f(x)|^2 = p^{-n} \sum_{\xi \in \mathbb{F}_p^n} |\hat{f}(\xi)|^2.
    \]
\end{theorem}

It is quite instructive to work through this proof, using that 
\[
\sum_{\xi \in 
\mathbb{F}_p^n} |\hat{f}(\xi)|^2 = \sum_{\xi \in \mathbb{F}_p^n} \sum_{x,y\in \mathbb{F}_p^n} f(x) \overline{f(y)} e^{-2\pi i (x - y)\cdot \xi/p}
\]
and considering the cases where $x = y$ and $x\neq y$. Alternatively, one can use the Fourier inversion formula which gives 
\[
f(x) = p^{-n}\sum_{\xi \in \mathbb{F}_p^n} \hat{f}(\xi) e^{2\pi i x\cdot \xi/p}.
\]
For a complete proof of both Plancherel and Fourier inversion, albeit with a different normalization of the Fourier transform, see \cite[Section 6.1]{ZhaoGTAC}. 

\textbf{\textit{Falconer's Projection Theorem over $\mathbb{F}_p^n$}}: We are now ready to state and prove the $\mathbb{F}_p^n$ analogue of Falconer's projection theorem  (Theorem \ref{FALCONER}).

\begin{theorem}[Falconer's Projection Theorem over $\mathbb{F}_p^n$] \label{DISCFalconer}
    Let $p$ be prime, $X\subset \mathbb{F}_p^n$, and fix $1\leq k \leq n-1$. Furthermore, let 
    \[
    E_s(X) := \{V\in \mathcal G(\mathbb{F}_p^n, k) : |P_V(X)|< s\}.
    \]
    Then, for all $0 < s \leq \min\{|X|, \frac{1}{2}p^k\}$,
    \[
    |E_s(X)| \lesssim p^{k(n-k)}s|X|^{-1}.
    \]
\end{theorem}

\begin{proof}
    We will prove the above statement when $k = 1$ in detail, and then briefly highlight what changes for arbitrary $k$. 
    
    Let $k = 1$ and fix some $0 < s \leq \min\{|X|, \frac{1}{2}p^k\}$. Notice that if $|X|\geq p^{n-1}s$, then $E_s(X) = \emptyset$. This follows as for all $\theta$,
    \[
    X\subset P_{\theta}^{-1}(P_\theta(X)) \implies |X| \leq p^{n-1} |P_\theta(X)|.
    \]
    Hence, if $\theta \in E_s(X)$, it'd follows that $|X| < p^{n-1} s$ reaching a contradiction. Therefore, we only need deal with the case $|X| \leq p^{n-1} s$ in proving Falconer's projection theorem, a fact we use near the end of the proof. It may be helpful to note that this reduction is akin to Remark \ref{miniOPremark}.

    For all $\theta \in \mathcal G(\mathbb{F}_p^n, 1)$, the level sets of $P_\theta(X)$ are hyperplanes that are translations of $\theta^\perp.$ For each $\theta$, let $B_\theta \subset \theta := L(\theta,0)$ be such that 
    \[
    P_\theta^{-1}(P_\theta(X)) = \bigcup_{b\in B_\theta} \theta^\perp + b.
    \]
    Then, we let 
    \[
    f_\theta = \sum_{b \in B_\theta} \mathbf{1}_{\theta^\perp + b} \quad \quad \text{ and } f = \sum_{\theta \in E_s(X)} f_\theta.
    \]
    We prove Theorem \ref{DISCFalconer} by bounding the $L^2$-norm of $f\mathbf{1}_X$ above and below. 

    We firstly claim that
    \begin{equation} \label{HighLowOP0}
    |X| |E_s(X)|^2= \sum_{x\in \mathbb{F}_p^n} |f(x) \mathbf{1}_X(x)|^2.
    \end{equation}
    To see this, notice that for all $x\in X$ (for which the argument of the above sum is nonzero) and each $\theta \in E_s(X)$, there exists a unique $b \in B_\theta$ such that $x \in \theta^\perp + b.$ Then, applying Plancherel (Theorem \ref{PLANCHEREL}) to \eqref{HighLowOP0}, we have 
    
    \begin{equation} \label{HighLowOP1}
    |X||E_s(X)|^2 = p^{-n}\sum_{\xi \in \mathbb{F}_p^n} |\widehat{f\mathbf{1}_X}(\xi)|^2 = p^{-n}\sum_{\xi \in \mathbb{F}_p^n} |\sum_{\theta \in E_s(X)} \widehat{f_\theta \mathbf{1}_X}(\xi)|^2.
    \end{equation}
    The last line follows by the linearity of the Fourier transform. By Proposition \ref{FourierOfPlane}, notice that $\supp \hat{f_\theta} \subset L(\theta,0)$. In particular, notice that when $\theta \neq \theta'$, 
    \[
    L(\theta,0) \cap L(\theta',0) = \{0\}.
    \]
    Hence, it follows that when $\xi \neq 0$, 
    \begin{equation} \label{HighLowOP2}
    |\sum_{\theta \in E_s(X)} \widehat{f_\theta \mathbf{1}_X}(\xi)|^2 = \sum_{\theta \in E_s(X) } | \widehat{f_\theta \mathbf{1}_X}(\xi)|^2.
    \end{equation}
    As such, we split the sum in \eqref{HighLowOP1} into cases where $\xi \neq 0$ and $\xi = 0$, or the \textit{high} frequencies and the \textit{low} frequencies respectively. By \eqref{HighLowOP2}, we have 
    \begin{align*}
        p^{-n} \sum_{\xi \neq 0} |\sum_{\theta \in E_s(X)}\widehat{f_\theta \mathbf{1}_X}(\xi)|^2 &= p^{-n} \sum_{\xi \neq 0} \sum_{\theta \in E_s(X)} |\widehat{f_\theta\mathbf{1}_X}(\xi)|^2 \\
        &\leq p^{-n} \sum_{\xi \in \mathbb{F}_p^n} \sum_{\theta \in E_s(X)} |\widehat{f_\theta\mathbf{1}_X}(\xi)|^2 \\
        &= \sum_{\theta \in E_s(X)} \sum_{x\in \mathbb{F}_p^n} |f_\theta(x) \mathbf{1}_X(x)|^2.
    \end{align*}
    Here, the second inequality comes from adding back in the zero frequency, and the last line follows by applying Plancherel again. In particular, notice  
    \[
    \sum_{\theta \in E_s(X)} \sum_{x\in \mathbb{F}_p^n} |f_\theta(x) \mathbf{1}_X(x)|^2 \leq \sum_{\theta \in E_s(X)}  p^{n-1} s \leq |E_s(X)| p^{n-1} s
    \]
    as there are $p^{n-1}$-many points in each fiber of $P_\theta$. We now deal with the 0 frequency term. By definition of the Fourier transform and the Cauchy--Schwarz inequality, we have
    \[
        p^{-n} |\widehat{f\mathbf{1}_X}(0)|^2 := p^{-n} \left|\sum_{x\in X} \sum_{\theta \in E_s(X)} f_\theta(x)\cdot 1 \right|^2 \leq  p^{-n} |X| \sum_{x\in X} |\sum_{\theta \in E_s(X)} f_\theta(x)|^2.
    \]
    As we argued earlier, for each $x\in X$ and each $\theta \in E_s(X)$, there is a unique $b\in B_\theta$ such that $x\in \theta^\perp + b$. In particular, we have 
    \[
    p^{-n} |X| \sum_{x\in X} |\sum_{\theta \in E_s(X)} f_\theta(x)|^2 \leq p^{-n} |X|^2 |E_s(X)|^2 \leq \frac{1}{2}|X| |E_s(X)|^2.
    \]
    The last line follows by the assumption that $|X| \leq p^{n-1}s$ and $s\leq \min\{|X|, p/2\}$. Hence, combining \eqref{HighLowOP0} and \eqref{HighLowOP1} with the above inequalities, we see
    \[
    |X||E_s(X)|^2 \leq |E_s(X)| p^{n-1}s + \frac{1}{2} |X| |E_s(X)|^2.
    \]
    Rearranging the above inequality, we obtain 
    \[
    |E_s(X)| \leq 2p^{n-1} s|X|^{-1}.
    \]
    which concludes the proof for $k=1$. The above argument goes largely unchanged for general values of $k$. In particular, for general $k$, we let 
\[
f_V = \sum_{b \in B_V} \mathbf{1}_{V^\perp + b} \quad \quad \text{ and } \quad \quad f= \sum_{V\in E_s(X)} f_V
\]
where $B_V$ is such that 
\[
P_V^{-1}(P_V(X)) = \bigcup_{b\in B_V} V^\perp + b.
\]
Considering the $L^2$-norm of $f\mathbf{1}_X$ again, we see 
\[
|X| |E_s(X)|^2 \leq \sum_{x\in \mathbb{F}_p^n} |f(x)|^2 = p^{-n} \sum_{\xi\in \mathbb{F}_p^n} |\widehat{f\mathbf{1}_X}(\xi)|^2.
\]
To deal with the right hand side, we again split the sum into when $\xi = 0$ and $\xi \neq 0$. The low frequency term, when $\xi = 0$, can be dealt with in the same manner as in the previous argument. In particular, we assume $|X| < p^{n-k} s$ (as otherwise $E_s(X) = \emptyset$), and use that $s\leq \min\{|X|, \frac{1}{2}p^{k}\}$ to see 
\[
p^{-n} |X|^2 |E_s(X)|^2 \leq \frac{1}{2} |X||E_s(X)|^2.
\]
The high frequency term, when $\xi \neq 0$, is a bit harder to deal with. In particular, it may no longer be the case that $\hat{f}_V$ and $\hat{f}_{V'}$ have disjoint support away from the origin if $V\neq V'$ as they are $k$-planes for some $k\geq 2$. Notice that when $\xi \neq 0$ and if $f_V(\xi) \neq 0$, there are $\leq\binom{n-1}{k-1}_p \sim p^{(k-1)(n-k)}$ choices of $V' \neq V$ such that $f_{V'}(\xi) \neq 0$. This follows as for $\xi \in \supp \widehat{f}_{V'}$, it must be the case that $L(\xi,0) \subset V'$. In particular, as we noted in Remark \ref{FiniteFieldSubspaces}, there are at $\binom{n-1}{k-1}_p$ many $k$-planes that contain the fixed line $L(\xi,0)$. Thus, 
\begin{align*}
    p^{-n} \sum_{\xi \neq 0} |\sum_{V\in E_s(X)} \widehat{f_V \mathbf{1}_X}(\xi)|^2 &\lesssim p^{-n} \sum_{\xi \in \mathbb{F}_p^n} \sum_{V\in E_s(X)} p^{(k-1)(n-k)} |\widehat{f_V1_X}(\xi)|^2 \\
    &\leq p^{(k-1)(n-k)} \sum_{V\in E_s(X)} \sum_{x\in \mathbb{F}_p^n}|f_V(x)\mathbf{1}_X(x)|^2 \\
    &\leq p^{(k-1)(n-k)} |E_s(X)|p^{n-k}s \\
    &= |E_s(X)| p^{k(n-k)} s.
\end{align*}
Combining this with the low frequency term, we obtain the desired bound: 
\[
|E_s(X)|\lesssim p^{k(n-k)} s|X|^{-1}. \qedhere
\]
\end{proof}

We make a few concluding remarks about orthogonal projections over finite fields. Firstly, using $p^{(k-1)(n-k)}$ to deal with the upperbound for the high frequency terms can be quite lossy. If one is more careful in their analysis of the high term, as Gan and I were in \cite{BrightGan2}, one can improve Theorem \ref{DISCFalconer}. Furthermore, as we noted earlier and as is discussed within \cite{BrightGan2}, the above arguments also hold over $\mathbb{F}_q^n$ where $q$ is a prime \textit{power}. This is an important distinction to make, as there are some exceptional set estimates that are known to fail over $\mathbb{F}_q^n$. For example, based on the statement of Falconer's projection theorem over $\mathbb{F}_p^n$, it is reasonable to conjecture that there is a finite field analogue to Theorem \ref{OSRW} as Chen did in \cite{Chen2018}:

\begin{conjecture} \label{finitefieldchenconjecture}
    Let $X\subset \mathbb{F}_p^2$ for $p$ prime. Then, for all $\epsilon>0$,
    \[
    |E_s(X)| \lesssim p^\epsilon \max\{s^2 |X|^{-1}, 1\}.
    \]
\end{conjecture}

The sharpness of this conjecture is proven in \cite[Example 13]{BrightGan2}, and it is worth noting that it is known to \textit{fail} over $\mathbb{F}_q^n$ where $q = p^r$ for some $r\geq 2.$ If $q=p^2$ and $X \subset \mathbb{F}_q^2$ is the subgroup isomorphic to $\mathbb{F}_p^2$, then 
\[
|\{\theta \in \mathcal G(\mathbb{F}_q^2, 1) : |P_\theta(X)|\leq p\}| = p \gg \max\{p^2 |X|^{-1},1\}.
\]
Hence, in the study of exceptional set estimates over finite fields, it is important to note if results hold for $\mathbb{F}_q^n$ for any prime power $q$ or if they specifically require prime order fields. The best progress towards proving Conjecture \ref{finitefieldchenconjecture} over $\mathbb{F}_p^2$ was obtained in 2023 by Lund, Pham, and Vinh \cite{LundPhamVinh} who proved
\[
|E_s(X)| \lesssim \min\{s^{5/2}|X|^{-1}, s^6|X|^{-3}, s\}.
\]
For more on exceptional set estimates over finite fields, see \cite{LundPhamVinh,BrightGan2,GanESEandFurst}.

\newpage 
\section{Furstenberg Sets} \label{sec:Furstenberg}

Recall that a Kakeya set (Definition \ref{def:kakeya}) is a compact set $K \subset \R^n$ with a unit line segment in every direction $\theta \in \mathbb{S}^{n-1}$. In this Section, we will discuss $(s,t)$-Furstenberg sets---a fractal version of Kakeya sets---and the associated $(s,t)$-Furstenberg set problem.

\begin{definition}[$(s,t)$-Furstenberg set] \label{FurstDEF}
Let $n\geq 2$, $0 < s \leq 1$, and $0 < t\leq 2(n-1)$. We say a Borel set $F\subset \R^n$ is an $(s,t)$\textit{-Furstenberg set} if there exists a Borel set of lines $\mathcal L \subset \mathcal A(n,1)$ such that $\dim \mathcal L \geq t$ and $\dim (F\cap \ell) \geq s$ for all lines $\ell \in \mathcal L$.
\end{definition}

\begin{remark}
    The ranges on $s$ and $t$ are natural as $\dim \mathcal A(n,1) = 2(n-1)$ and for every line $\ell \in \mathcal A(n,1)$, $\dim \ell = 1$. Furthermore, notice that a Kakeya set $K\subset \R^n$ is a $(1,n-1)$-Furstenberg set.
\end{remark}

Just as the Kakeya problem (Conjecture \ref{kakeyaconj}) seeks to find lowerbounds to the dimension of Kakeya sets, the $(s,t)$-Furstenberg set problem seeks to find lowerbounds to the dimension of $(s,t)$-Furstenberg sets. We motivate the sharp bounds for Furstenberg sets by once again studying a discrete model example. 

\begin{proposition} \label{discreteFurst}
    Let $t>0$ and $s\geq 2$. Let $F\subset \R^n$ be a finite set such that there exists a finite set of affine lines $\mathcal L \subset \mathcal A(n,1)$ with $|\mathcal L|\geq t$ and $|F\cap \ell| \geq s$ for all $\ell \in \mathcal L$. Then, for $s\geq 2$,
    \begin{itemize}
        \item[a)] (Cauchy--Schwarz I) $|F|\gtrsim st^{1/2}$,
        \item[b)] (Cauchy--Schwarz II) $|F| \gtrsim \min\{s^2,st\}$, and
        \item[c)] (Szemer\'edi--Trotter) $|F| \gtrsim \min\{st, s^{3/2}t^{1/2}\}$ for $s\geq 2$.       
    \end{itemize}
\end{proposition}

\begin{proof}
    We once again proceed by applying point-line incidences. Consider the set of points
    \[
    P = \left(\bigcup_{\ell \in \mathcal L} F\cap \ell\right) \subset F
    \]
    and lines $\mathcal L$ as in the statement of the Proposition. We obtain lowerbounds on $P$ and thusly $F$. Notice that 
    \begin{equation}\label{FurstSetUp}
    |\mathcal L| s \leq |\mathcal I(P,\mathcal L)|
    \end{equation}
    as every line $\ell \in \mathcal L$ intersects $P$ in $\geq s$-many points.

    To obtain Statement (a), we apply the first Cauchy--Schwarz bound on point line incidences (Proposition \ref{sec2:prop:STPrelim}) to \eqref{FurstSetUp}, getting 
    \[
    |\mathcal L|s \leq |\mathcal I(P,\mathcal L)| \leq |P| |\mathcal L|^{1/2} + |\mathcal L| \implies  |\mathcal L|^{1/2}(s-1) \leq |P|.
    \]
    Using that $s\geq 2$ and $|\mathcal L|\geq t$ gives $st^{1/2} \lesssim |P|$.

    To obtain Statement (b), we instead apply the second Cauchy--Schwarz bound to get
    \[
    |\mathcal L| s \leq |\mathcal I(P,\mathcal L)| \leq |\mathcal L| |P|^{1/2} + |P|.
    \]
    If the first term dominates, it follows that $s^2 \lesssim |P|$. If the second term dominates, we can use that $|\mathcal L| \geq t$ to obtain $st \lesssim |P|$.

    We lastly prove Statement (c). Note that if $2\leq s < 13$, the claim follows by (a); hence, we suppose $s\geq 13.$ Applying Szemer\'edi--Trotter (Theorem \ref{SZEMEREDITROTTER}) to \eqref{FurstSetUp}, we have 
    \[
    |\mathcal L| s \leq |\mathcal I(P,\mathcal L)| \leq 4(|P|^{2/3} |\mathcal L|^{2/3} + |P| + |\mathcal L|).
    \]
    Note that the third term cannot dominate as $s\geq 13$. If the second term dominates, then $st\lesssim |P|$. If the first term, then $s^{3/2} t^{1/2} \lesssim |P|$. Hence, Statement (c) follows by taking the minimum of these two bounds.
\end{proof}

Just as was the case with orthogonal projections, the continuum counterpart of the Szemer\'edi--Trotter bound (Proposition \ref{discreteFurst}.(c)) is sharp and was obtained in $\R^2$ by Orponen--Shmerkin \cite{OrponenShmerkinABC} and Ren--Wang \cite{RenWang} last year.

\begin{theorem}[\cite{OrponenShmerkinABC,RenWang}] \label{FURSTENBERG}
    Let $0< s \leq 1$ and $0< t\leq 2$. Then, any $(s,t)$-Furstenberg set $F\subset \R^2$ satisfies
    \[
    \dim F \geq \min\left\{s + t, \frac{3s+t}{2}, s+1\right\}.
    \]
\end{theorem}

For the remainder of this section, we present an abridged history of the Furstenberg set problem, starting just 25 years prior to the above result. Some of the first Furstenberg set estimates were recorded in a survey paper on Kakeya sets by Wolff in 1999 \cite{Wolff99}. 

Technically, Wolff considers a special class of $(s,1)$-Furstenberg sets with a line in every direction. More precisely, he considers sets $F \subset \R^2$ such that for every direction $\theta \in \mathbb{S}^1$, there exists a line $\ell_\theta$ in direction $\theta$ with $\dim (F\cap \ell_\theta) \geq s$. In the literature \cite{OrponenShmerkinEpsilonImprove}, this class of $(s,1)$-Furstenberg sets is sometimes referred to simply as \textit{$s$-Furstenberg sets} and obtaining sharp lowerbounds to such sets is correspondingly called the $s$-Furstenberg set problem. Though Wolff considered $s$-Furstenberg sets, a modification of his approach more generally gives:

\begin{theorem}[\cite{Wolff99}] \label{WolffsBounds}
For $s\leq 1$, every $(s,1)$-Furstenberg set $F\subset \R^2$ has 
\[
\dim F \geq \max\left\{2s, s + \frac{1}{2}\right\}.
\]
\end{theorem}

These bounds are akin to the maximum of Proposition \ref{discreteFurst}.(a) and (b) with $s\leq t$. In fact, Wolff proved the above bounds via a Cauchy--Schwarz type argument; for proofs we refer the reader to \cite{Wolff99}.

Before moving on, we make a few historical remarks about the content of \cite{Wolff99} and the origins of the name ``Furstenberg sets.'' Firstly, the study of Furstenberg sets is motivated by Furstenberg's work on $\times 2$, $\times 3$-invariant sets \cite{Furstenberg70}. Furthermore, Wolff states that ``it is unlikely that the author [himself] was the first to observe these bounds; in all probability they are due to Furstenberg and Katznelson.'' Lastly, in unpublished work recalled in \cite[Remark 1.5]{Wolff99}, Furstenberg conjectured that the special case of $(s,1)$-Furstenberg sets that Wolff considers satisfy $\dim F\geq \frac{3s+1}{2}$. In the survey, Wolff presents the sharpness of the Furstenberg set conjecture ($t= 1$) based on a grid example and Szemer\'edi--Trotter. The construction generalizes to $t\leq 1$ proving the sharpness of $\frac{3s+t}{2}$ in Theorem \ref{FURSTENBERG}.

The result and methods of Wolff were generalized in 2010 by Molter--Rela \cite{MolterRela} considering a class of $(s,t)$-Furstenberg sets with lines in a \textit{fractal} set of directions $\Theta \subset \mathbb{S}^1$ (with $\dim \Theta \geq t$). For Furstenberg sets $F\subset \R^2$ within this class, they obtain 
\[
\dim F \geq \max\left\{2s + t - 1, s + \frac{t}{2}\right\}.
\]
Instead of considering lines in a $\geq t$-dimensional set of \textit{directions}, one can more generally consider a $\geq t$-dimensional set of \textit{affine lines} (as per Definition \ref{FurstDEF}). In the plane, these considerations are more or less equivalent for $t\leq 1$ (see \cite[Section 5]{BrightDhar}), but for the remainder of this section we only consider Furstenberg sets associated to affine subspaces. 

One paper to make this distinction for the Furstenberg set problem is due to H\'era--Keleti--M\'ath\'e in 2019 \cite{HeraKeletiMathe}, obtaining that an $(s,t)$-Furstenberg set $F\subset \R^n$ satisfies
\[
\dim F \geq 2s + \min\{t,1\} - 1.
\]
Under the same assumptions, a subsequent result of H\'era in 2019 \cite{HeraFurstenberg} gives
\[
\dim F \geq s + \frac{t}{2},
\]
proving the continuum analogue of Proposition \ref{discreteFurst}.(a) in all dimensions. We note here that these two papers also consider Furstenberg sets associated to \textit{$k$-dimensional} affine subspaces, though we don't discuss this further here. For a survey on Furstenberg sets associated to $k$-planes, and new results on this topic, we refer the reader to a recent paper of myself and Dhar \cite{BrightDhar}.

The continuum counterpart to Proposition \ref{discreteFurst}.(b) was obtained in the plane by Lutz--Stull in 2020 \cite{LutzStull} and in all dimensions by H\'era--Shmerkin--Yavicoli in 2021 \cite{HeraShmerkinYavicoli}. They show an $(s,t)$-Furstenberg set $F\subset \R^n$ satisfies 
\[
\dim F \geq s + \min\{s,t\}.
\]
We make a few remarks about their proofs before moving on. Regarding the proof of Lutz and Stull, we remark that they apply a Kolmogorov complexity argument which has since proven fruitful for various problems in geometric measure theory (see e.g. \cite{KolmogorovEx1,KolmogorovEx2,KolmogorovEx3,KolmogorovEx4}). Regarding the proof of H\'era--Shmerkin--Yavicoli in 2021, we remark that they apply a Cauchy--Schwarz argument and $(\de,s)$-set machinery. Though we won't show their argument here, an analogous argument for \textit{dual} Furstenberg sets (due to myself, Fu, and Ren \cite{BrightFuRen}) will be presented in Section \ref{sec:DualFurstenberg}. Lastly, notice that when $0 < t\leq s \leq 1$, the bound of Lutz--Stull and H\'era--Shmerkin--Yavicoli gives the bound of $s + t$ as in Theorem \ref{FURSTENBERG}. The sharpness of this result can be seen by taking the product of an $s$-dimensional Cantor set and a $t$-dimensional Cantor set, see e.g. the ``Cantor target'' in \cite{MolterRela}.

In summary, the above bounds give:

\begin{theorem}[\cite{MolterRela,HeraFurstenberg, LutzStull,HeraKeletiMathe}] \label{ElementaryFurstenberg}
    Let $0< s\leq 1$ and $0 < t \leq 2(n-1)$, and let $F\subset \R^n$ be an $(s,t)$-Furstenberg set. Then,
    \[
    \dim F \geq \max\left\{s + \min\{s,t\}, s + \frac{t}{2}\right\}.
    \]
\end{theorem}

Notice that when $t = 2s$, both terms in Theorem \ref{ElementaryFurstenberg} gives $\dim F \geq 2s$ whereas Theorem \ref{FURSTENBERG} gives $\dim F \geq 2s + \epsilon$ for some $\epsilon:= \epsilon(s,t)>0.$ As it turns out, beating Wolff's 1999 bound of $2s$ was a big feat---and for good reason too: this bound is sharp over $\mathbb{C}$! 

To draw this comparison, we define the \textit{complex dimension} of a set $X \subset \mathbb{C}^2$, denoted $\dim_{\mathbb{C}}X$, to be \textit{half} the Hausdorff dimension of $X\subset \R^4$ (identifying $\mathbb{C}^2$ with $\R^4$). We similarly define the complex dimension of a set of affine complex lines $\mathcal L$ in $\mathbb{C}^2$ as half $\dim \mathcal L$ (again identifying $\mathbb{C}^2\simeq\R^4$). 

\begin{example}
Consider $\R^2$ embedded into $\mathbb{C}^2$ given by $F=\mathbb{C} \times \{0\}\subset \mathbb{C}^2$, and consider the set of all (complex) lines contained in $F$ denoted $\mathcal{L}$. For each line $\ell \subset \R^2$, $\dim (\R^2 \cap \ell) = 1$. Thus, $\dim_{\mathbb{C}}(F\cap \ell) = 1/2$ for all $\ell \in \mathcal L$. Furthermore, $\dim \mathcal A(\R^2,1) = 2$, and thus $\dim_{\mathbb{C}}\mathcal L = 1$. Hence, $F$ is a (complex analogue of) an $(1/2,1)$-Furstenberg set with $\dim_{\mathbb{C}}F = 1$!
\end{example}

The main obstacle here is that there exists a half-dimensional subring of $\mathbb{C}$; namely $\dim_{\mathbb{C}} \R = 1/2.$ However, recall from Theorem \ref{ctmsubringstatement} that no Borel subring of $\R$ has Hausdorff dimension strictly between 0 and 1! 

This is where Bourgain's discretized statement of Theorem \ref{ctmsubringstatement}, which resolved a conjecture of Katz and Tao, enters the scene. In particular, the prior work of Katz and Tao from 2001 \cite{KatzTaoSubringConjecture} in conjunction with Bourgain's result from 2003 \cite{BourgainSubring} gives that $(1/2,1)$-Furstenberg sets $F\subset \R^2$ satisfy $\dim F\geq 1 + c$ for an absolute constant $c>0.$ 

This result cascaded into a number of papers in 2021 on so-called ``$\epsilon$-improvements'' over the $2s$ bound for the $(s,t)$-Furstenberg sets in the plane (for $s< t$). In \cite{HeraShmerkinYavicoli}, H\'era, Shmerkin, and Yavicoli ``simplify, clarify, adapt and quantify many of the steps'' of Katz and Tao's approach to obtain an $\epsilon$-improvement for $(s,2s)$-Furstenberg sets. Later that year, the $\epsilon$ obtained by H\'era--Shmerkin--Yavicoli was further quantified in work of Di Benedetto and Zahl \cite{BenedettoZahl}. Lastly, utilizing the discretized statement of Bourgain's projection theorem from 2010 (Theorem \ref{BOURGAIN}), in 2021 Orponen--Shmerkin obtained:

\begin{theorem}[\cite{OrponenShmerkinEpsilonImprove}] \label{OSEpsilonFurst}
    Let $0 < s<t\leq 2$. Then, there exists an $\epsilon:= \epsilon(s,t)>0$ such that any $(s,t)$-Furstenberg set $F \subset \R^2$ satisfies
    \[
    \dim F \geq 2s + \epsilon.
    \]
\end{theorem}

We make a few historical remarks and connections before continuing on. Firstly, the above result was first obtained by Orponen for the \textit{packing} dimension of $F$ in 2019 \cite{OrponenPackingEpsilonImprove}, the discretized statement of which was utilized in the proof of Theorem \ref{OSEpsilonFurst}. Secondly, as was noted in \cite[Corollary 1.6]{OrponenShmerkinEpsilonImprove}, we note that for $s > \frac{1}{2}$ Theorem \ref{OSEpsilonFurst} gives an improvement to Wolff's bound (Theorem \ref{WolffsBounds}) for the $s$-Furstenberg set problem. To this end, note that an $\epsilon$-improvement to Wolff's bound for $s = \frac{1}{2}$ follows from the work of Bourgain and Katz--Tao \cite{BourgainSubring,KatzTaoSubringConjecture}, and for $s<\frac{1}{2}$ from work of Shmerkin--Wang \cite{ShmerkinWang}. The last remark we make on this variant of the $(s,t)$-Furstenberg set problem is that an explicit improvement for all $s$ was obtained by Orponen--Shmerkin last year \cite{OrponenShmerkinABC}. Lastly, it is worth noting for context that both Theorem \ref{OSEpsilonProj} and Theorem \ref{OSEpsilonFurst} follow from the same discretized statement \cite[Theorem 1.3]{OrponenPackingEpsilonImprove}, highlighting an explicit connection between Furstenberg sets and the exceptional set estimates.

We conclude this section by discussing some of the connections between Furstenberg sets and radial projections (the focus of Section \ref{sec:RadProj}). One paper that makes this connection especially clear is the 2022 paper of Orponen--Shmerkin--Wang which proved ``Kaufman and Falconer estimates for radial projections'' using (discretized) Furstenberg set results. The Kaufman-type estimate makes use of Orponen and Shmerkin's $\epsilon$-improvement (Theorem \ref{OSEpsilonFurst}) and implies sum-product phenomena (see e.g. \cite[Corollary 1.17]{OSW}). Such sum-product estimates were further studied utilizing Orponen, Shmerkin, and Wang's radial projection estimates in \cite{OrponenShmerkinABC}, which in turn led to the sharp exceptional set and Furstenberg set estimates in the plane (Theorems \ref{OSRW} and \ref{FURSTENBERG}) within the following year. 

The Falconer-type estimate makes use of another $(s,t)$-Furstenberg set estimate in the plane due to Fu and Ren in 2021 \cite{FuRen}. In particular, they proved the sharp bound of $s+1$ in Theorem \ref{FURSTENBERG} when $s+t \geq 2$:

\begin{theorem}[\cite{FuRen}]
Let $0 < s \leq 1$, $1\leq t\leq 2$, and let $F\subset \R^2$ be an $(s,t)$-Furstenberg set. Then, 
\[
    \dim F \geq \min\{2s+t-1,s+1\}.
\]
\end{theorem}

The sharpness of this bound can be seen by taking the product of an $s$-dimensional Cantor set and an interval. Fu and Ren's proof employs the high-low method to obtain (discretized) bounds on the number of incidences between $\delta$-thickened points and lines (\textit{balls} and \textit{tubes}). Their incidence bound was utilized by Orponen, Shmerkin, and Wang to imply a radial projection estimate in the plane \cite[Theorem 1.7]{OSW}. A subsequent paper of myself, Fu, and Ren from 2024 \cite{BrightFuRen} obtains an analogous estimate for incidences between balls and tubes in $\R^n$ for all $n\geq 2$. In turn, the proof method of Orponen--Shmerkin--Wang can thus be used to generalize \cite[Theorem 1.7]{OSW} to higher dimensions; see Theorems \ref{BGan-LundPhamThu} and \ref{brightfurenradproj}.

\subsection{Dual Furstenberg Sets} \label{sec:DualFurstenberg}

In this section, we focus our attention to \textit{dual} $(s,t)$-Furstenberg sets. In particular, we will discuss the sense in which they are ``dual,'' state an estimate for such sets due to myself, Fu, and Ren \cite{BrightFuRen} (Theorem \ref{BFR-DUALFURST}), and prove a discrete analogue of this result that is morally similar to its continuum counterpart. We will use Theorem \ref{BFR-DUALFURST} in the subsequent chapter. For more on the dual Furstenberg set problem, we refer the reader to \cite{BrightDhar,OrponenShmerkinEpsilonImprove,RenRadProj,DabrowskiOrponenVilla,LiLiuDualFurstenberg}.

\begin{definition}[Dual $(s,t)$-Furstenberg set] Let $n\geq 2$, $0 < s \leq n-1$, and $0<t\leq n$. We say a Borel set $\mathcal L \subset \mathcal A(n,1)$ is a \textit{dual} $(s,t)$\textit{-Furstenberg set} if there exists a Borel set of pins $X\subset \R^n$ such that $\dim X\geq t$ and $\dim \{\ell \in \mathcal L : x\in \ell\} \geq s$ for all $x\in X$.    
\end{definition}

\begin{remark}
    Here, it is clear that the ranges on $s$ and $t$ are natural as $X\subset \R^n$ and a set of lines $\mathcal L_x \subset \mathcal A(n,1)$ through a point $x\in \R^n$ can be identified (up to antipodal points) with where they intersect $$\mathbb{S}^{n-1} + x := \{y \in \R^n : |y-x| = 1\}.$$ The corresponding \textit{dual} $(s,t)$-Furstenberg set problem asks for sharp lowerbounds on $\dim \mathcal L$ for all ranges of $s$, $t$, and $n$.
\end{remark}

Notice that when $n = 2$, the ranges on $s$ and $t$ are the same as that for the $(s,t)$-Furstenberg set problem, and for good reason: in the plane, these problems are \textit{dual} via point-line duality. We discussed point-line duality in Definition \ref{point-lineduality}, but present the definition here again for convenience:

\begin{definition*}[Point-Line Duality] 
    Let $p = (p_1,p_2) \in \R^2$ and $\ell = mx+b$. We define the \textit{dual} of $p$, denoted $\mathbf{D}(p)$, to be the line given by $y = p_1 x - p_2$. Similarly, we define the \textit{dual} of $\ell$, denoted $\mathbf{D}^\ast(\ell)$, to be the point $(m,-b)$.
\end{definition*}

In Section \ref{sec:Incidences}, it was important to note that point-line duality preserved incidences. In particular, given a point $p$ and an affine line $\ell$ in $\R^2$,
\[
p \in \ell \quad \text{ if and only if } \quad \mathbf{D}^\ast(\ell) \in \mathbf{D}(p).
\]
For the purposes of Hausdorff dimension, the maps $\mathbf{D}$ and $\mathbf{D}^\ast$ have nice metric properties too. These properties are carefully recorded in \cite[Sections 6.1 and 6.2]{DabrowskiOrponenVilla}, but the key thing to note is that $\mathbf{D}$ and $\mathbf{D}^\ast$ are locally bi-Lipschitz maps between $\R^2$ and $\mathcal A(2,1)$ (under the standard metrics).\footnote{In fact, for this reason, there are a number of papers that define the Hausdorff dimension of a set of lines $\mathcal L \subset \mathcal A(2,1)$ as being the Hausdorff dimension of $\mathbf{D}^\ast(\mathcal L)$ rather than using the metric on the affine Grassmannian. This can be a bit more approachable to a reader unfamiliar with the affine Grassmannian, and is ultimately equivalent in the plane.} In particular, given an $(s,t)$-Furstenberg set $F\subset \R^2$, it follows that $\mathbf{D}(F)$  is a dual $(s,t)$-Furstenberg set with $\dim F = \dim \mathbf{D}(F)$. In this way, when $n = 2$, the $(s,t)$-Furstenberg set problem is precisely equivalent to its dual analogue, and lowerbounds for one imply lowerbounds for the other. In fact, some of the bounds we have presented thus far have been proven by considering a discretization of the \textit{dual} Furstenberg set problem; see \cite{OrponenShmerkinEpsilonImprove,RenRadProj}.

However, while these problems are equivalent in the plane, they aren't in higher dimensions.\footnote{There is a point-hyperplane duality that holds in all dimensions (and thus explains why point-line duality holds in $\R^2$); see \cite{DabrowskiOrponenVilla}. This duality can be used in the study of Furstenberg sets of hyperplanes, see \cite{BrightDhar, DabrowskiOrponenVilla}. For this reason, I believe that dual Furstenberg sets should perhaps be more aptly referred to as \textit{pinned} Furstenberg sets, but I digress.} As such, when $n\geq 3$, new proof methods need to be employed to obtain bounds for the dual $(s,t)$-Furstenberg set problem. One such bound was obtained by myself, Fu, and Ren, giving a dual analogue of H\'era--Shmerkin--Yavicoli and Lutz--Stull's result \cite{HeraShmerkinYavicoli,LutzStull} in all dimensions.

\begin{theorem}[\cite{BrightFuRen}] \label{BFR-DUALFURST}
    Let $n\geq 2$, $0< s \leq n-1$, and $0 < t \leq n$. Given $\mathcal L\subset \mathcal A(n,1)$ is a dual $(s,t)$-Furstenberg set, it follows that 
    \[
    \dim \mathcal L \geq s + \min\{s,t\}.
    \]
\end{theorem}

\begin{remark}
    The above bound can be obtained as a special case of \cite[Corollary 2.10]{RenRadProj} and a Hausdorff content argument based on \cite{HeraShmerkinYavicoli}. We include the above bound for use in proving the continuum version of Beck's theorem and the Erd\H{o}s--Beck for lines in Chapter \ref{ch:problems}.
\end{remark}

It was natural to conjecture Theorem \ref{BFR-DUALFURST} based on the work of H\'era--Shmerkin--Yavicoli and Lutz--Stull, and it similarly follows via a Cauchy--Schwarz argument. To illustrate this point, note the following discrete analogue of Theorem \ref{BFR-DUALFURST}.

\begin{proposition}\label{discDFurstExample}
    Let $\mathcal L \subset \mathcal A(n,1)$ be finite, and suppose there exists a finite set of points $X \subset \R^n$ such that $|X|\geq t$ and \[
    |\mathcal L_x := \{\ell \in \mathcal L : x\in \ell\}| \geq s
    \]
    for all $x\in X.$ Then, 
    \[
    |\mathcal L| \gtrsim \min\{s^2,st\}.
    \]
\end{proposition}

\begin{proof}
    Firstly, by removing lines as necessary, we may assume without of generality that $|\mathcal L_x| = s$ for all $x\in X$.
    We prove the above lowerbound by obtaining the same bound for the set $\bigcup_{x\in X} \mathcal L_x.$
    By abuse of notation, we call this set $\mathcal L$ as well.
    Consider the set of triples 
    \[
    J(X,\mathcal L) := \{(x,x',\ell) \in X^2 \times \mathcal L : \ell \in \mathcal L_x \cap \mathcal L_{x'}\}.
    \]
    By definition, we have 
    \[
    |J(X,\mathcal L)| = \sum_{x\in X} \sum_{x,\ell} \mathbf{1}_\ell(x) \mathbf{1}_{\ell}(x') \leq \sum_{x\in X} \sum_{x'\in X} |\mathcal L_x\cap \mathcal L_{x'}|.
    \]
    Notice that if $x\neq x'$, there is at most one line through $x,x'$. Thus,
    \begin{align*}
    \sum_{x\in X} \sum_{x'\in X} |\mathcal L_x\cap \mathcal L_{x'}| &= \sum_{x \in X} |\mathcal L_x| + \sum_{x\neq x'} |\mathcal L_x \cap \mathcal L_{x'}| \\
    &\leq \sum_{x \in X} |\mathcal L_x| + \sum_{x\neq x'} 1 \\
    &\leq |X|s + |X|^2 \\
    &\lesssim \max\{s|X|,|X|^2\}.
    \end{align*}
    On the other hand, 
    \begin{align*}
        |J(X,\mathcal L)| &= \sum_{\ell \in \mathcal L} |\{(x,x') \in X^2 : \ell \in \mathcal L_x \cap \mathcal L_{x'}\} \\
        &= \sum_{\ell \in \mathcal L} |\{x \in X : \ell \in \mathcal L_x\}|^2 \\
        &\geq |\mathcal L|^{-1} \left(\sum_{\ell \in \mathcal L} |\{x \in X : \ell \in \mathcal L_x\}|\right)^2 \\
        &= |\mathcal L|^{-1} \left(\sum_{x\in X} |\mathcal L_x|\right)^2 \geq |\mathcal L|^{-1} (|X|s)^2.
    \end{align*}
    Hence, in total we have 
    \[
    |\mathcal L|^{-1} |X|^2 s^2 \lesssim \max\{s|X|,|X|^2\}.\]
    Rearranging this bound with $|X|\geq t$ proves the desired result.
\end{proof}

As it turns out, Proposition \ref{discDFurstExample} is a dual analogue of Proposition \ref{discreteFurst}.(b). In fact, following the proof methodology of Proposition \ref{discreteFurst}, one obtains:

\begin{proposition} \label{discDFurstTotalProp}
    Let $s,t>0$ and $n\geq 2$. Let $\mathcal L\subset \mathcal A(n,1)$ be a finite set such that there exists a finite set of points $X\subset \R^n$ with $|X|\geq t$ and \[
    |\mathcal L_x:= \{\ell \in \mathcal L : x\in \ell\}|\geq s
    \]
    for all $x\in X$. Then, 
    \begin{itemize}
        \item[a)] (Cauchy--Schwarz I) $|\mathcal L|\gtrsim st^{1/2}$ for $s\geq 2$, 
        \item[b)] (Cauchy--Schwarz II) $|\mathcal L| \gtrsim \min\{s^2,st\}$ for $s \geq 2$, and
        \item[c)] (Szemer\'edi--Trotter) $|\mathcal L| \gtrsim \min\{st, s^{3/2}t^{1/2}\}$ for $s\geq 2$.  
    \end{itemize}
\end{proposition}

Still, we highlight the proof of Proposition \ref{discDFurstExample} for a few reasons. Firstly, the double-counting argument involving triples $|J(X, \mathcal L)|$ is common and useful in this area of incidence geometry. Secondly, we presented this argument with notation similar to that of \cite{BrightFuRen} so that the interested reader (who may perhaps be unfamiliar with $(\de,s)$-sets) may find the latter, more intricate proof, easier to follow. Lastly, while Theorem \ref{discreteFurst} is morally the same as Theorem \ref{discDFurstTotalProp}, they are only logically equivalent when $n = 2$ by point-hyperplane duality. For the discrete problems, one can deduce one statement from the other by a generic projection argument without loss in cardinality. However, for the continuum problems, a generic projection can drastically drop the Hausdorff dimension of your set as you are projecting onto a smaller dimensional space!

\newpage 
\section{Radial Projections} \label{sec:RadProj}

It is a natural and interesting question to ask whether or not Heuristics \ref{hr:Marstrand} and \ref{hr:exceptional} (that large objects typically cast relatively large shadows) apply to other families of projection maps---especially nonlinear projections. One such family that served as a driving force in this field are \textit{radial projections}.

\begin{notation}
    Given $x\in \R^n$, let $\pi_x:\R^n \setminus \{x\}\to \mathbb{S}^{n-1}$ be defined by
\[
\pi_x(y) := \frac{y-x}{|y-x|}.
\]
\end{notation}

\noindent As one might expect, the nonlinear setting can pose problems that didn't arise for orthogonal projections. Regardless, it is reasonable to think that point-line incidences may still be fruitful for radial projections as the fibers of $\pi_x$ are lines through $x$. We will see this echoed throughout the theory of radial projections which we turn to now. 

\subsection{The Discrete Setting}

The theory of \textit{discrete} radial projections is of highly important for Beck's theorem---both historically and for the purposes of this thesis. As such, we postpone any discussion of the continuum setting until Section \ref{ss:radprojctm}. We will use the following notation throughout this section.

\begin{notation}
    Given $Y\subset \R^n$, let 
    \[
    \mathcal L(Y) = \{\ell : |Y\cap \ell|\geq 2\} \subset \mathcal A(n,1). 
    \]
    Additionally, given $x\in \R^n$, let
    \[
    \mathcal L_x(Y) = \pi_{x}^{-1}(\pi_x(Y\setminus \{x\})) \subset \mathcal A(n,1)
    \]
    be the set of lines passing through $x$ containing a point in $Y\setminus \{x\}.$
    Lastly, given $\mathcal L \subset \mathcal A(n,1)$, we let $\tilde{\mathcal L}$ denote the embedding of $\mathcal L$ into $\R^n$ i.e. 
    \[
    \tilde{\mathcal L} := \{x \in \R^n : x\in \ell \text{ for some } \ell \in \mathcal L\},
    \]
    and similarly define and denote $\tilde{\mathcal L}(Y)$ and $\tilde{\mathcal L}_x(Y).$
\end{notation}

We begin with a discrete Marstrand-type statement for radial projections.

\begin{proposition}
    Given a finite set $Y \subset \R^n$ with $|Y| \geq 2$, $|\pi_x(Y\setminus \{x\})| < |Y|$ only if $x\in \tilde{\mathcal L}(Y)$. Hence, there are infinitely many $x\in \R^n$ such that 
    \[
    |\pi_x(Y\setminus \{x\})| = |Y|.
    \]
\end{proposition}

\begin{proof}
    We do case work on whether or not $x\in Y.$
    If $x\in Y$, notice that $|\pi_x(Y\setminus\{x\})|< |Y|$ is always true, and by definition $x\in \tilde{\mathcal L}(Y).$ Otherwise, if $x\notin Y$ and $|\pi_x(Y\setminus \{x\})| < |Y|$, there are at least two distinct points $y_1,y_2 \in Y\setminus \{x\}$ such that $\pi_x(y_1) = \pi_x(y_2)$. Therefore, $x$ is contained in the line that contains $y_1$ and $y_2$, and thus $x\in \tilde{\mathcal L}(Y)$. In either case, if $|\pi_x(Y\setminus \{x\})| < |Y|$, then $x\in \tilde{\mathcal{L}}(Y)$. Hence, for all $x\in \R^n \setminus \tilde{\mathcal L}(Y)$, $|\pi_x(Y \setminus \{x\})| = |Y|$.
\end{proof}

\begin{remark}
    It's worth noting that the first statement would be an if and only if statement if we instead considered $$\tilde{\pi}_x: \R^n \setminus \{x\}\to \mathcal A(n,1),$$ mapping $y\in \R^n \setminus \{x\}$ to the unique line passing through $x$ and $y$. In the study of radial projections where we are interested in the size of $\pi_x(Y\setminus \{x\})$ up to constants, this alternative definition doesn't make a difference. Indeed,
    \[
    |\pi_x(Y\setminus \{x\})| \sim |\tilde{\pi}_x(Y\setminus \{x\})| \quad \text{ and } \quad \dim \pi_x(Y\setminus \{x\}) = \dim \tilde{\pi}_x(Y\setminus \{x\})
    \]
    for all $x\in \R^n$, see \cite[Lemma 20]{BrightMarshall}.
\end{remark}

Based on what happened for the study of exceptional set estimates for orthogonal projections, one might be inclined to define 
\[
E_s(Y):= \{x\in \R^n : |\pi_x(Y \setminus \{x\})|<s\},
\]
for $0 < s \leq |Y|$ and try to immediately obtain bounds on $|E_s(Y)|$. However, there is one key difference between the discrete Marstrand-type statement for orthogonal projections and that for radial projections. Namely, for orthogonal projections, $E_s(X)$ is a subset of the direction set $S(X)$ which is \textit{always} finite; for radial projections, $E_s(Y) \subset \tilde{\mathcal L}(Y)$, which contains infinitely many points in $\R^n$. This begs the question of whether or not exceptional sets for radial projections are finite in the first place.

The following proposition gives examples of $Y\subset \R^n$ such that $E_s(Y)$ is infinite for both small and large values of $s$.

\begin{proposition} \label{ESE BAD FOR RAD PROJ}
    If $Y$ is collinear, there are infinitely many $x\in \R^n$ such that $|\pi_x(Y\setminus \{x\})| = 1$.
    Moreover, if $Y$ is noncollinear, there are infinitely many $x\in \R^n$ such that 
    \[
    |\pi_x(Y \setminus \{x\})| = |Y|-1.
    \]
\end{proposition}

\begin{proof}
    The first statement is rather immediate by taking $x$ to be contained on the line containing $Y$ (outside of a bounded interval). To prove the second statement, we note the Sylvester--Gallai theorem.\footnote{See \cite[Exercise 7.9]{GuthPolynomialMethods} for an outline of how to prove the Sylvester--Gallai theorem using Euler's formula following the methodology in \cite{melchior}.}

    \begin{lemma}[Sylvester--Gallai] \label{sylvestergallai}
        If $Y\subset \R^n$ is finite and noncollinear, then there exists an affine line that contains exactly two points of $Y.$
    \end{lemma}

    Using the Sylvester--Gallai theorem, suppose $Y$ is noncollinear and let $\ell \in \mathcal A(n,1)$ denote a line that contains exactly two distinct points $y_1,y_2\in Y$. Then, notice that there are only finitely many points in $$A := \ell \cap \tilde{\mathcal L}(Y\setminus  \{y_1,y_2\}),$$
    i.e. points $x\in \ell$ such that $|\pi_x(Y\setminus \ell)| < |Y\setminus \ell| = |Y|-2$. Hence, for all $x\in \ell \setminus A$ and not contained in the line segment with endpoints $y_1,y_2$, it follows that $|\pi_x(Y)| = |Y|-1.$
\end{proof}

Given Proposition \ref{ESE BAD FOR RAD PROJ}, searching for bounds on $E_s(Y)$ is a moot point if $Y$ is contained on a line, and even \textit{if} $Y$ is noncollinear, there are \textit{some} values of $s$ such that $E_s(Y)$ is infinite. However, depending on how noncollinear $Y$ is, we can show that $E_s(Y)$ is finite for a sharp range of $s$.

\begin{notation}
    Let $Y \subset \R^n$ be a finite noncollinear set. Then, let $c(Y)$ be the maximal number such that
    \[
    |Y\setminus \ell | \geq c(Y) |Y| \quad \quad \forall \ell \in \mathcal A(n,1).
    \]
    Note, $c(Y) \in [\frac{1}{|Y|}, \frac{|Y|-2}{|Y|}]$ by the Sylvester--Gallai theorem.
\end{notation}

We can now show that $E_s(Y)$ is finite depending on the value of $c(Y).$

\begin{proposition} \label{prop:ESEfinite}
    Let $Y\subset \R^n$ be finite and noncollinear. Then, $E_s(Y)$ is finite for all $0< s \leq c(Y)|Y|$. In fact, for all such $s$, $$E_s(Y) \subset P_2(\mathcal L(Y)),$$
    the set of 2-rich points for the lines $\mathcal L(Y)$ (see Definition \ref{richpointsandlines}).
\end{proposition}

We first prove a lemma towards Proposition \ref{prop:ESEfinite}.

\begin{lemma}\label{lem:ESEfinite1}
    Let $Y\subset \R^n$ be finite and noncollinear. Then, for all $x\in \R^n$,
    \[
    |\pi_x(Y \setminus \{x\})/\pm| \geq 2.
    \]
    In particular, if $Y$ is noncollinear, $Y \subset P_2(\mathcal L(Y))$.
\end{lemma}

\begin{proof}
Suppose for the sake of contradiction there exists an $x\in \R^n$ such that $|\pi_x(Y\setminus \{x\})/\pm| = 1$. Note that $|\pi_x(Y\setminus \{x\})/\pm|$ cannot be zero, as $|Y| \geq 3$ by noncollinearity. It follows that 
\[
Y\subset \pi_x^{-1}(\pi_x(Y\setminus \{x\})).
\]
In particular, note that the right hand side is a line as $|\pi_x(Y\setminus \{x\})/\pm| = 1$, which contradicts the noncollinearity of $Y.$

Hence, given $y\in Y$, $|\pi_x(Y\setminus \{y\})/\pm| \geq 2$, meaning there are at least two distinct lines $\ell_1,\ell_2$ through $y$ containing at least one point of $Y \setminus \{y\}$. Notice that $\ell_1,\ell_2 \in \mathcal L(Y)$ by definition. Therefore, $y\in P_2(\mathcal L(Y)).$    
\end{proof}

Notice that Lemma \ref{lem:ESEfinite1} implies that
\[
(E_s(Y) \cap Y) \subset P_2(\mathcal L(Y)),
\]
so we won't need to worry about such points when proving Proposition \ref{prop:ESEfinite}.

\begin{proof}[Proof of Proposition \ref{prop:ESEfinite}]
    By Lemma \ref{lem:ESEfinite1}, we may restrict our attention to $x\in E_s(Y) \setminus Y$. For ease of notation, let
    \[
    \mathcal L_x = \mathcal L_x(Y):= \pi_x^{-1}(\pi_x(Y)),
    \]
    notably with $|\mathcal L_x| < s \leq c(Y)|Y|$. Also note that given $Y$ is noncollinear, by Lemma \ref{lem:ESEfinite1}, $|\mathcal L_x| \geq 2$. We claim that at least two of these lines are at least 2-rich of points in $Y$. Given $\mathcal L_x$ covers $Y$ with strictly fewer than $|Y|$-many lines, it is clear that one of these fibers, say $\ell_1$, must be at least 2-rich. 
    
    We claim $\mathcal L_x \setminus \ell_1$ contains at least one other 2-rich line, $\ell_2$. Note that by definition of $c(Y)$, $|Y\setminus \ell_1| \geq c(Y)|Y|$, and by definition of $E_s(Y)$, $$|\mathcal L_x\setminus \ell_1| <s-1 \leq c(Y)|Y|-1,$$ In particular, $\mathcal L_x\setminus \ell_1$ covers $Y\setminus \ell_1$, and thus the existence of such an $\ell_2$ follows. Notably, $\ell_1,\ell_2 \in \mathcal L(Y)$ and $x\in \ell_1\cap \ell_2$. Thus, $x\in P_2(\mathcal L(Y)).$ 
\end{proof}

To see that the range of $s$ is sharp, consider the following two examples.

\begin{example}
    Let $Y\subset \R^2$ contain $(1-c(Y))|Y|$-many points on the $x$-axis, more specifically $[-1,1]\times \{0\}$ The $x$-axis will then be our line $\ell^\ast$. Then, there are infinitely many points $x\in \ell^\ast$ such that \begin{equation}\label{radprojex1}
    |\pi_x(Y\setminus \{x\})| \geq c(Y)|Y|.
    \end{equation}
    Firstly, note for all $x\in (\R \setminus [-1,1])\times \{0\}$, $|\pi_x(Y\cap \ell^\ast)| = 1$. Therefore, if $x$ fails \eqref{radprojex1}, then $|\pi_x(Y\setminus \ell^\ast)| \leq c(Y)|Y| -1$, i.e. $x\in P_2(\mathcal L(Y\setminus \ell^\ast))$, of which there are only finitely-many points.
\end{example}

\begin{example}
    Let $Y\subset \R^2$ contain $(1-c(Y))|Y|$-many points in $[-1,1]\times \{0\}$ and $c(Y)|Y|$-many points in $\{0\}\times [-1,1]$. Then, every point $x\in (\R \setminus [-1,1]) \times \{0\}$ satisfies $|\pi_x(Y)| = c(Y)|Y| + 1$.
\end{example}

We hope that by the above discussion it is clear that the study of discrete radial projections is trickier than that for orthogonal projections. This is partly due to the question of if $E_s(Y)$ is even finite to begin with, but even if we let $X\subset E_s(Y)$ be finite, bounding $|X|$ nontrivially (in terms of $s$ and $|Y|$) is tricky for other geometric considerations. In particular, given how we approached the discrete exceptional set estimates for orthogonal projections (Proposition \ref{discreteESE}), one might first be inclined to consider the set of lines 
\[
\mathcal L = \bigcup_{x\in X} \mathcal L_x(Y)
\]
and then bound $|\mathcal I(Y,\mathcal L)|$ above and below. While we can certainly bound $|\mathcal I(Y,\mathcal L)|$ from above by the Szemer\'edi--Trotter theorem (noting that $|\mathcal L|< |X|s$), bounding it from below is quite trickier as there may exist lines $\ell$ that lie in $\mathcal L_x \cap \mathcal L_{x'}$ for numerous distinct pairs $x,x'\in X$. However, assuming that $X$ is significantly noncollinear this should rarely happen. To this end, we will say that $X$ satisfies the $C$-\textit{nonconcentration condition} if for there are at most $C$-many points of $X$ on any line, and note the following Proposition.

\begin{proposition}\label{discreteradprojESE}
    Let $X, Y\subset \R^n$ be finite, $0 < s \leq |Y|$, and $C>0$. 
    Then,
    \begin{itemize}
        \item[a)] (Product Set) If $X$ is noncollinear, then $X\not\subset E_{\frac{1}{2}|Y|^{1/2}}(Y)$.
        \item[b)] (Cauchy--Schwarz) If $X\subset E_s(Y)$ satisfies the $C$-nonconcentration condition, then for $0 < s\leq (2C)^{-1}|Y|$,
        \[
        |X| \lesssim_C s.
        \]
        \item[c)] (Szemer\'edi--Trotter) If $X\subset E_s(Y)$ satisfies the $C$-nonconcentration condition, then for $0 < s\leq (2C)^{-1}|Y|$,
        \[
        |X| \lesssim_C \max\{s^2|Y|^{-1},1\}.
        \]
    \end{itemize}
\end{proposition}

\begin{proof}
Throughout the proof, let $$\mathcal L_x = \mathcal L_x(Y) := \pi_{x}^{-1}(\pi_x(Y\setminus \{x\})).$$ 

To see Statement a), suppose for the sake of contradiction that $X\subset  E_{\frac{1}{2}|Y|^{1/2}}(Y)$ and let $x_1,x_2,x_3 \in X$ be noncollinear. Then, if $\ell$ is the line containing $x_1,x_2$, it follows that 
\[
Y \setminus \ell \subset (\tilde{\mathcal L}_{x_1} \cap \tilde{\mathcal L}_{x_2}).
\]
In particular, $|Y\setminus \ell| <|Y|/4$. Given $x_3 \notin \ell$ by assumption, we have
\[
Y\cap \ell \subset (\tilde{\mathcal L}_{x_3} \cap \ell) \implies |Y\cap \ell| < \frac{1}{2}|Y|^{1/2}.
\]
Hence, it follows that 
\[
|Y| = |Y\setminus \ell| + |Y\cap \ell| < |Y|
\]
which is a contradiction, proving statement (a).

We now set up the incidence arguments that give statements (b) and (c), again letting $|X| = t$. Notice by definition that for all $x\in X\subset E_s(Y)$, it takes at most $s$-many lines through $x$ to cover $Y$. Hence, let 
\[
\mathcal L = \bigcup_{x\in X} \mathcal L_x,
\]
and note that $|\mathcal L| < |X| s.$
Consider the set of incidences between $Y$ and $\mathcal L$. Given that there are at most $C$-many points of $X$ on any line $\ell \in \mathcal A(n,1)$, and $\mathcal L_x$ covers $Y\setminus \{x\}$ for all $x\in \R^n,$ it follows that 
\[
C^{-1} |X||Y| \leq C^{-1} \sum_{x\in X} |\mathcal I(Y, \mathcal L_x)| \leq \sum_{\ell \in \mathcal L} |\mathcal I(Y,\ell)| = |\mathcal I(Y,\mathcal L)|,
\]
where here the factor of $C^{-1}$ is using that for all $\ell \in \mathcal L$ there are at most $C$-many $x\in X$ with $\ell \in \mathcal L_x$. From here, Proposition \ref{discreteradprojESE}.(b) follows immediately from Proposition \ref{sec2:prop:STPrelim} and (c) follows from Szemer\'edi--Trotter.
\end{proof}

\noindent Theorem \ref{discreteradprojESE} also implies \textit{pinned} results. Recalling Notation \ref{SimNotation}, we have:

\begin{proposition}\label{discretepinnedradproj}
    Let $X,Y\subset \R^n$ be finite sets and $C>0$. Then, 
    \begin{itemize}
        \item[i)] (Product Set) If $X$ is noncollinear, $\max_{x\in X} |\pi_x(Y\setminus \{x\})| \gtrsim |Y|^{1/2}$.
        \item[ii)] (Cauchy--Schwarz) If $X$ satisfies the $C$-nonconcentration condition,\footnote{We are able to weaken this nonconcentration condition to obtain the same bound from a Beck-type theorem. See Corollary \ref{corollaryBeckRadProj}.}
        \[
        \max_{x\in X} |\pi_x(Y\setminus \{x\})| \gtrsim_C \min\{|X|, |Y|\}.
        \]
        \item[iii)] (Szemer\'edi--Trotter) If $X$ satisfies the $C$-nonconcentration condition, 
        \[
        \max_{x\in X} |\pi_x(Y\setminus \{x\})| \gtrsim_C \min\{|X|^{1/2}|Y|^{1/2}, |Y|\}.
        \]
    \end{itemize}
\end{proposition}

\begin{proof}
We claim that Statement i) holds with implicit constant 1/2. To see this, note that $X$ has at least 3 noncollinear points, say $x_1,x_2,x_3$, and suppose for the sake of contradiction that 
\[
\max_{x\in X} |\pi_x(Y\setminus \{x\})| < \frac{|Y|^{1/2}}{2}.
\]
Then, let $\ell = \ell(x_1,x_2)$ denote the line through $x_1,x_2$. It follows from the above bound that $Y$ has strictly less than $|Y|/4$-many points off of $\ell$, as 
\[
Y \subset \ell \cup (\tilde{\mathcal L}_{x_i} \setminus \ell) \quad \text{for }i=1,2.
\]
However, note that this would be an immensely bad bound if $Y$ is contained in $\ell$ itself. However, given $x_3 \notin \ell$, it follows that 
\[
(Y\cap \ell) \subset \tilde{\mathcal L}_{x_3},
\]
and thus 
\[
|Y| \leq |Y\cap \ell| + \frac{|Y|}{4} < |Y|
\]
which is a contradiction.

We claim that Statements ii) and iii) hold with, say, implicit constants $(2C)^{-1}$ and $(12C)^{-3/2}$ and we prove the latter here. Firstly, note that if $X$ is sufficiently small (e.g. $|X|\leq 12 C$), the claim follows from Statement i), so we suppose $|X|> 12C$. Then, let
\[
s:= (12C)^{-3/2} \min\{|X|^{1/2}|Y|^{1/2},|Y|\}
\]
and suppose for the sake of contradiction that for all $\epsilon >0$
\[
X\subset (E_{s-\epsilon}(Y) \cap X).
\]
Thus, by Proposition \ref{discreteradprojESE} (which we can apply as $s-\epsilon \leq C^{-1}|Y|/2$), 
\begin{align*}
|X| &\leq \#(E_{s-\epsilon}(Y)\cap X)  \\
&\leq \max\{(12C)^3 (s-\epsilon)^2|Y|^{-1}, 12C\} \\
&< |X|
\end{align*}
which gives a contradiction for $\epsilon$ sufficiently small. Here, we use the implicit constants that we dropped in the proof of Proposition \ref{discreteradprojESE}.
\end{proof}

To conclude our discussion of discrete radial projections, we briefly note a few historical results. Firstly, notice that if $X = Y$ is noncollinear and $X$ satisfies the $C$-nonconcentration condition, then $|S(Y)|\gtrsim_C |Y|$. This follows by Proposition \ref{discretepinnedradproj}.(ii) by noting $\pi_y(Y\setminus \{y\}) \subset S(Y)$ for all $y\in Y.$ A stronger result of Ungar from 1982 \cite{Ungar82} gives the following.

\begin{proposition}[\cite{Ungar82}] \label{UngarResult}
    If $Y$ is noncollinear, then 
    \[
    |S(Y)| \geq |Y|-1,
    \]
    where $S(Y)$ is the direction set of $Y.$
\end{proposition}

A year later in 1983, Beck \cite{Beck83} obtained Proposition \ref{discretepinnedradproj}.(ii) when $X = Y$ as a consequence of his titular theorem (Theorem \ref{thm:introBeck}). In fact, Beck's theorem gives the conclusion of the Proposition under a significantly weaker condition than $C$-nonconcentration, see Corollary \ref{corollaryBeckRadProj}. As a corollary, Beck's theorem implies Ungar's result up to implicit constants, partially resolving a conjecture of Dirac \cite{DiracConj} and Motzkin \cite{MotzkinConj}.

\subsection{The Continuum Setting} \label{ss:radprojctm}

The continuum theory of radial projections has seen rapid developments over the past decade, which is where the majority of our survey takes place. However, it is worth briefly noting that radial projections have been classically studied, having first been explored in the plane by Marstrand in the same 1954 paper \cite{Marstrand54} proving his eponymous projection theorem (Theorem \ref{MARSTRAND}).

In this paper, Marstrand obtains \textit{visibility} results for radial projections. Given a Borel set $Y \subset \R^n$ and $x\in \R^n$, we say $Y$ is \textit{invisible} from $x$ if $$\mathcal H^{n-1}(\pi_x(Y\setminus \{x\})) = 0.$$ Notice that if $\mathcal H^{n-1}(Y) = 0$, which for instance holds if $\dim Y< n-1$ and fails if $\dim Y>n-1$, it follows that $Y$ is invisible from \textit{every} $x\in \R^n$. Hence, if one seeks to show that $Y$ is rarely invisible, it is reasonable to assume $\dim Y>n-1$, or (as a special case) $0< \mathcal H^{s}(Y)< \infty$ for some $s> n-1$ as Marstrand originally did in the plane \cite{Marstrand54}.

\begin{remark}
One can study the visibility of $Y$ when $0 < \mathcal H^{n-1}(Y) <\infty$, the results of which depend on the rectifiability of $Y$, though we do not discuss this topic further here. See \cite{OrponenRadProjSmoothness,OrponenSahlsten} and references therein for more.
\end{remark}

Marstrand showed that all Borel sets $Y \subset \R^2$ with positive and finite $\mathcal H^s$-measure (for some $s>1$) are visible from $\mathcal L^2$-almost every $x\in \R^2$ and visible from $\mathcal H^s$-almost every $x\in Y$. Nearly 60 years later, this result was unified and generalized by Mattila and Orponen \cite{MattilaOrponenRadProjVisibility,OrponenRadProjVisibility} whose work collectively shows the following:

\begin{theorem}[\cite{MattilaOrponenRadProjVisibility,OrponenRadProjVisibility}] \label{mattila-orponen-visibility}
    For $Y\subset \R^n$ Borel with $\dim Y>n-1$,
\begin{equation} \label{MattilaOrponenSharp}
\dim \{x\in \R^n : Y \text{ is invisible from } x\} \leq 2(n-1) - \dim Y.
\end{equation}
Moreover, \eqref{MattilaOrponenSharp} is sharp in the following sense: for all $s>n-1$, there exists Borel sets $E,Y \subset \R^{n}$ with $\dim Y = s$ and $\dim E = 2(n-1)-s$ such that $\mathcal H^{n-1}(\pi_x(Y\setminus \{x\}))>0$ for all $x\in \R^n \setminus E$.
\end{theorem}

It may be insightful to the reader (who may perhaps be more familiar with orthogonal projections) to note that this exceptional set estimate for visibility immediately implies the following Marstrand-type corollary: if $\dim Y>n-1$, then $\mathcal H^{n-1}(\pi_x(Y\setminus \{x\}))>0$ for $\mathcal L^n$-almost every $x\in\R^n$. In general, \textit{visibility problems} seek to understand how often the image of a projection map has positive measure. Such problems are classically studied in geometric measure theory; see \cite[Section 6]{Mattila04Survey} for a survey. 

For the remainder of this thesis, we will discuss Hausdorff dimensional results for radial projections. In particular, given $X,Y \subset \R^n$ Borel, when does there exist an $x\in X$ such that $\dim \pi_x(Y\setminus \{x\})$ is ``large''? In 2018, Orponen \cite{OrponenRadProjSmoothness} considered this problem, obtaining one of the first dimensional lowerbounds for radial projections which states:

\begin{theorem}[Orponen's Radial Projection Theorem \cite{OrponenRadProjSmoothness}] \label{orponenradprojtheorem}
    Given Borel sets $X,Y \subset \R^2$, if $\dim X >0$ and $X$ is not collinear, then there exists an $x\in X$ such that
    \[
    \dim \pi_x(Y \setminus \{x\}) \geq \frac{\dim Y}{2}.
    \]
\end{theorem}

Note that the noncollinearity condition on $X$ is necessary. Indeed, if $X,Y$ are contained in a common line $\ell$ then $\dim \pi_x(Y \setminus \{x\}) = 0$ for all $x\in X$, as was the case in the discrete setting. Speaking of, notice that Orponen's radial projection theorem gives the continuum analogue of Proposition \ref{discretepinnedradproj}.(i).

A year later in 2019, Liu utilized Orponen's work to obtain an exceptional set estimate for radial projections in $\R^n$. In particular, he showed:

\begin{theorem}[\cite{Liu2020}] \label{liusbound}
Given a Borel set $Y\subset \R^n$ with $\dim Y \in (n-2,n-1]$,
\[
\dim\{x \in \R^n : \dim \pi_x(Y\setminus \{x\}) < \dim Y\} \leq 2(n-1) - \dim Y.
\]
\end{theorem}

We make a number of remarks about Theorem \ref{liusbound} before moving on. Firstly, as Liu noted, this bound is sharp when $\dim Y = n-1$. 

\begin{example} \label{exLiu}
Let $Y$ to be a $(n-1)$-plane $V \subset \R^n$. Then, for all $x\in V$, $\pi_x(Y\setminus \{x\}) \simeq \mathbb{S}^{n-2}$, and thus $\dim \pi_x(Y\setminus \{x\}) = n-2< \dim Y.$ Furthermore, for all $x\notin V$, $\dim \pi_x(Y) = n-1=\dim Y$.\footnote{There are a number of ways to see this, but one nice way is to use that for all $x\in \R^n \setminus V$, $\pi_x:V\to \mathbb{S}^{n-1}$ is a locally bi-Lipschitz map and preserves dimension; see Lemma \ref{LemRadProjLocBiLips}.} Hence, we have 
\[
\{x \in \R^n : \dim \pi_x(Y\setminus \{x\})< \dim Y\} = V,
\]
which is $(n-1)$-dimensional, proving the sharpness of Liu's exceptional set estimate when $\dim Y = n-1$.
\end{example}

Though the above example is sharp at the endpoint, Liu did not believe that the above bound is sharp for all ranges of $\dim Y$. He did, however, conjecture that a Borel set $Y \subset \R^n$ satisfies
\begin{equation}\label{liusconjecture}
\dim \{x \in \R^n : \dim \pi_x(Y) < \dim Y\} \leq \lceil \dim Y\rceil.
\end{equation}
The sharpness of this bound when $\dim Y = k$ for some integer $1\leq k \leq n-1$ can be seen by taking $Y$ to be a $k$-plane, similar to Example \ref{exLiu}. It's also worth noting that if $\dim Y\in (n-1,n]$, it is trivially true that $\dim \pi_x(Y\setminus \{x\}) < \dim Y$. If $\dim Y\in (n-1,n]$,
one can instead apply Orponen and Mattila--Orponen's visibility bound (Theorem \ref{mattila-orponen-visibility}) to get that\footnote{This is an immediate consequence of Theorem \ref{mattila-orponen-visibility} as given $\mathcal H^{n-1}(\dim \pi_x(Y\setminus \{x\}))>0$, it follows that $\dim \pi_x(Y\setminus \{x\}) = n-1.$ That said, while the sharpness of the exceptional set bound for \textit{visibility} is known, the author is currently unaware of if the sharpness of the exceptional set bound for \textit{dimensionality} is known for $\dim Y>n-1$.}
\[
\dim\{x\in \R^n : \dim \pi_x(Y\setminus \{x\}) < n-1\} \leq 2(n-1) -\dim Y.
\]
We discuss Liu's conjecture \eqref{liusconjecture}, and the resolution thereof, momentarily. Before doing so, we discuss the work of Lund--Pham--Vu \cite{LundPhamThu} from 2022, which studies radial projections over finite fields $\mathbb{F}_q^n$ (for prime powers $q$).\footnote{As a historical remark, in \cite{Liu2020}, Liu acknowledges Orponen for reminding him to put the above conjecture into the literature. For this reason, what we henceforth refer to as Liu's conjecture is sometimes attributed as being due to Liu \textit{and} Orponen, see e.g. \cite{LundPhamThu}.}

Given $x,y\in \mathbb{F}_q^n$ with $x\neq y$, let $\pi_x(y)$ denote the unique affine line that contains $x$ and $y$ (see Definition \ref{FpAffineLines} for the definition of affine lines over finite fields). Using this notation, we can now state some of their results. 

\begin{theorem}[\cite{LundPhamThu}] \label{lundphamthuTheorem}
    Let $Y\subset \mathbb{F}_q^n$. If $|Y|\geq 6q^{n-1}$ and $s \leq \frac{1}{4}q^{n-1}$, then 
    \begin{equation} \label{lund-pham-thubound}
\#\{x\in \mathbb{F}_q^n : |\pi_x(Y\setminus \{x\})| < s\} \leq 12 q^{n-1}s |Y|^{-1}.
\end{equation}
    If $|Y|\leq \frac{1}{100}q^{n-1}$, then 
    \begin{equation} \label{lund-pham-thubound2}
\#\{x\in \mathbb{F}_q^n : |\pi_x(Y\setminus \{x\})| < |Y|/10\} < 8q^{n-1}.
\end{equation}
\end{theorem}

\begin{remark}
Notice that when $s = \frac{1}{4}q^{n-1}$, \eqref{lund-pham-thubound} is akin to Mattila--Orponen and Orponen's visibility result (Theorem \ref{mattila-orponen-visibility}). Furthermore, when $|Y|\geq q^{n-2}$, \eqref{lund-pham-thubound2} is the finite field analogue of \eqref{liusconjecture} for $\dim Y \in (n-2,n-1].$ This not only bolstered Liu's conjecture, but led to an analogous conjecture over $\mathbb{F}_q^n$, see \cite[Conjecture 1.6]{LundPhamThu} and Remark \ref{lundphamthuremark}.
\end{remark}

Though Theorem \ref{lundphamthuTheorem} is over finite fields, \eqref{lund-pham-thubound} immediately motivates a conjecture of theirs (\cite[Conjecture 1.2]{LundPhamThu}) in the continuum, namely: if $Y\subset\R^n$ is Borel with $\dim Y >n-1$ and $s\leq n-1$, then
\begin{equation} \label{lundphamthuRnConjecture}
\dim \{x\in \R^n : \dim \pi_x(Y\setminus \{x\}) < s\} \leq n-1 + s- \dim Y.
\end{equation}
Within the year, both \eqref{liusconjecture} and \eqref{lundphamthuRnConjecture} were proved by three groups of authors: Orponen and Shmerkin (in $\R^2$), myself and Gan, and Orponen--Shmerkin--Wang (as a corollary of their main results) \cite{BrightGan, OrponenShmerkin22RadProj,OSW}:

\begin{theorem}[\cite{BrightGan,OrponenShmerkin22RadProj,OSW}] \label{BGan-LundPhamThu}
    Given a Borel set $Y\subset \R^n$, if $\dim Y \in (k,k+1]$ for some integer $1\leq k \leq n-1$, and $0 < s \leq k$, then 
    \[
    \dim \{x\in \R^n : \dim \pi_x(Y\setminus \{x\}) < s\} \leq \max\{k + s - \dim Y,0\}.
    \]
\end{theorem}

\begin{theorem}[\cite{BrightGan,OrponenShmerkin22RadProj,OSW}] \label{BGan-Liu}
    Given a Borel set $Y\subset \R^n$, if $\dim Y \in (k-1,k]$ for some integer $1\leq k \leq n-1$, and $0 < s \leq k$, then 
    \[
    \dim \{x\in \R^n : \dim \pi_x(Y\setminus \{x\}) < \dim Y\} \leq k.
    \]
\end{theorem}

We make a few remarks about \cite{BrightGan} and \cite{OrponenShmerkin22RadProj} before moving on. Firstly, the statements of Theorems \ref{BGan-LundPhamThu} and \ref{BGan-Liu} \textit{in} \cite{BrightGan,OrponenShmerkin22RadProj} are, in some sense, cosmetically weaker. In particular, 1) the bounds are instead obtained for 
\[
\{x \in \R^n \setminus Y : \dim \pi_x(Y) < s\} \quad \quad \text{and} \quad \quad \{x \in \R^n \setminus Y : \dim \pi_x(Y) < \dim Y\}
\]
respectfully,\footnote{That said, I believe the proof methodology of Orponen--Shmerkin and myself and Gan should adapt immediately to the larger sets considered in Theorems \ref{BGan-LundPhamThu} and \ref{BGan-Liu}.} and 2) Theorem \ref{BGan-LundPhamThu} is proven for $s<k$, though the end point case with $s = k$ follows immediately (as is noted in \cite[Corollary 1.8]{OSW}):

\begin{lemma} \label{reducetos<k}
    Let $Y \subset \R^n$ be Borel, with $\dim Y \in (k,k+1]$ for some integer $1\leq k \leq n-1$. Suppose for all $0<s<k$,
    \[
    \dim \{x\in \R^n : \dim \pi_x(Y\setminus \{x\}) < s\} \leq \max\{k + s - \dim Y,0\}.
    \]
    Then, the same bound holds with $s = k$.
\end{lemma}

\begin{proof}
    The proof of this lemma follows by noting
    \[
    \{x \in \R^n : \dim \pi_x(Y\setminus \{x\}) < k\} = \bigcup_{j=1}^{\infty} \{x \in \R^n : \dim \pi_x(Y\setminus \{x\}) < k - 1/j\},
    \]
    and applying countable stability of Hausdorff dimension (Lemma \ref{countablestability}).
\end{proof}

Secondly, besides that \cite{BrightGan} holds in all dimensions where as \cite{OrponenShmerkin22RadProj} holds in the plane, the main way the methodology of myself and Gan differs from that of Orponen--Shmerkin is in the proof of Theorem \ref{BGan-LundPhamThu}. To obtain Theorem \ref{BGan-LundPhamThu}, Gan and I apply the high-low method to obtain Theorem \ref{BGan-LundPhamThu} for $\dim Y>n-1$. Note that this is analogous to the condition that $|Y|\geq 6q^{n-1}$ in Theorem \ref{lundphamthuTheorem}. Similar to Section \ref{BGanProofs}, we present a high-low method proof of Theorem \ref{lundphamthuTheorem} adapted to finite fields in Section \ref{BGanRadProjProof}. To obtain Theorem \ref{BGan-LundPhamThu} for $\dim Y \in (k-1,k]$, Gan and I apply a (discretized) Marstrand projection theorem, projecting $Y$ onto $k$-plane $V$ such that $\dim P_V(Y) = \dim Y\in (k-1,k]$. This allows us to apply the $\dim Y>n-1$ exceptional set estimate directly to $P_V(Y) \subset V\simeq \R^k$. The statements and proofs that make this heuristic argument precise can be found in \cite[Propositions 17 and 19]{BrightGan}. In the proof of Theorem \ref{BGan-Liu}, the methodologies of 
\cite{OrponenShmerkin22RadProj} and \cite{BrightGan} converge (the latter being motivated by the former). In particular, we deduce Theorem \ref{BGan-Liu} as a \textit{corollary} of Theorem \ref{BGan-LundPhamThu} by employing a ``swapping trick'' based on Liu's 2019 paper \cite{Liu2020}.

\begin{remark} \label{lundphamthuremark}
Given Theorem \ref{lundphamthuTheorem}, it is thus reasonable to expect that a Marstrand-type argument should imply a finite field analogue of Theorem \ref{BGan-LundPhamThu} for all ranges of $|Y|$. This was done in a subsequent paper of myself, Lund, and Pham in 2023 \cite{BrightLundPham}. This paper also proves that a Marstrand-type argument recovers the finite field analogue of Liu's conjecture from Theorem \ref{lundphamthuTheorem} for all ranges of $|Y|$ (thus resolving \cite[Conjecture 1.2]{LundPhamThu}).
\end{remark}

Lastly, it is worth noting that around the same time as my paper with Gan, the work of Orponen and Shmerkin \cite{OrponenShmerkin22RadProj} was subsumed into a paper of Orponen, Shmerkin, and Wang \cite{OSW} which we discuss now. To be more specific, Orponen and Shmerkin's planar argument for Theorem \ref{BGan-LundPhamThu} utilized Fu and Ren's bound for incidences between balls and tubes. This argument became the basis of \cite[Theorem 1.7]{OSW} which states:

\begin{theorem}[\cite{OSW}] \label{oswThm1.2}
    Given $X,Y\subset \R^2$ Borel with $X\neq \emptyset$ and $\dim Y>1$, 
    \[
    \sup_{x\in X} \dim \pi_x(Y\setminus \{x\}) \geq \min\{\dim X + \dim Y - 1, 1\}.
    \]
\end{theorem}

As we discussed in Section \ref{sec:Furstenberg}, Fu and Ren's incidence estimate was generalized to higher dimensions in a paper of myself, Fu, and Ren \cite{BrightFuRen}. In fact, we were motivated to write this paper to obtain a higher dimensional version of Theorem \ref{oswThm1.2}, which gives:

\begin{theorem}[\cite{BrightFuRen}] \label{brightfurenradproj}
    Given $X,Y\subset \R^n$ Borel with $X\neq \emptyset$ and $\dim Y \in (k,k+1]$ for some integer $1\leq k \leq n-1$,
    \[
    \sup_{x\in X} \dim \pi_x(Y\setminus \{x\}) \geq \min\{\dim X + \dim Y - k, k\}.
    \]
\end{theorem}

Furthermore, as was shown in \cite{BrightFuRen}, Theorems \ref{BGan-LundPhamThu} and \ref{brightfurenradproj} are formally equivalent. We note that one direction of this equivalence was proven in the plane in \cite[Corollary 1.8]{OSW}, and the argument directly carries over to $\R^n$.

\begin{proof}[Proof that Theorems \ref{BGan-LundPhamThu} and \ref{brightfurenradproj} are equivalent]
    We first suppose Theorem \ref{brightfurenradproj} holds and prove Theorem \ref{BGan-LundPhamThu}. Let $Y\subset \R^n$ with $\dim Y \in (k,k+1]$ for some $1\leq k \leq n-1$. By Lemma \ref{reducetos<k}, we may reduce to the case with $s< k$. In particular, given $s< k$, let 
    \[
    X = \{x \in \R^n : \dim \pi_x(Y\setminus \{x\}) < s\},
    \]
    and suppose for the sake of contradiction that 
    \[
    \dim X > \max\{k + s - \dim Y, 0\}.
    \]
    Then, it's clear that $X\neq \emptyset$ and $\dim X>0$. Hence, applying Theorem \ref{brightfurenradproj},
    \[
    \sup_{x\in X} \dim \pi_x(Y\setminus \{x\}) \geq \min\{\dim X + \dim Y - k, k\} > s.
    \]
    This directly contradicts the definition of $X$.

    Now we suppose that Theorem \ref{BGan-LundPhamThu} holds, and we prove Theorem \ref{brightfurenradproj}. Let $Y \subset \R^n$ have $\dim Y\in (k,k+1]$. Firstly, we note the following trivial bound, which will handle the case of Theorem \ref{brightfurenradproj} when $\dim X = 0.$

    \begin{lemma} \label{radprojtrivialbound}
        Let $Y\subset \R^n$ be Borel. Then, for all $x\in \R^n$, 
        \[
        \dim \pi_x(Y\setminus \{x\}) \geq \dim Y-1.
        \]
    \end{lemma}

    This bound is trivial, and essentially boils down to $Y$ being contained in the (1-dimensional) fibers of $\pi_x(Y\setminus \{x\})$ for all $x\in \R^n.$ We omit the proof here, as though the proof is short, it requires slightly more Hausdorff dimensional machinery (e.g. packing dimension or Frostman measures) than we wish to go into here. For a proof, see e.g. \cite[Proposition 2.6]{BrightMarshall} or \cite{BrightFuRen}. In either case, notice that Lemma \ref{radprojtrivialbound} gives 
    \[
    \sup_{x\in X} \dim \pi_x(Y\setminus \{x\}) \geq \dim Y-1\geq \dim Y - k
    \]
    implying Theorem \ref{brightfurenradproj} for $\dim X = 0$. Hence, assume $\dim X>0$, let 
    \[
    s = \min\{\dim X + \dim Y - k, k\},
    \]
    and let $\epsilon>0$ be sufficiently small such that $s- \epsilon >0$ and $\dim X - \epsilon>0$. Notably, such an $\epsilon$ exists since $\dim Y >k$ and $\dim X>0$. Thus, we can apply Theorem \ref{BGan-LundPhamThu} to obtain 
    \[
    \dim \{x\in \R^n : \dim \pi_x(Y\setminus \{x\}) < s-\epsilon\} \leq \max\{k + s- \epsilon - \dim Y, 0\} < \dim X.
    \]
    Thus, for all $\epsilon$ sufficiently small, there exists an $x\in X$ such that 
    \[
    \dim \pi_x(Y\setminus \{x\}) \geq s-\epsilon := \min\{\dim X + \dim Y - k,k\}-\epsilon.
    \]
    Since $\epsilon$ can be made arbitrarily small, this proves the desired claim.
\end{proof}

We briefly discuss how Orponen, Shmerkin, and Wang obtained Theorem \ref{BGan-LundPhamThu} in higher dimensions. Like the paper of myself and Gan \cite{BrightGan}, they first prove the case where $\dim Y >n-1$ via a potential-theoretic and slicing argument. To deal with $\dim Y \in (k,k+1]$ for arbitrary $1\leq k\leq n-1$, they apply an integralgeometric argument that is (conceptually) similar to the Marstrand-type argument in \cite{BrightGan}. I find the argument beautiful and concise, and it boils down to set containments and Lipschitz maps. As such, for the purposes of discussion we include the proof here with commentary.

\begin{proposition}[Proposition 4.3 in \cite{OSW}] \label{integralgeometricproof}
    Theorem \ref{oswThm1.2} follows from the special case of the same statement with $k = n-1.$
\end{proposition}

\begin{proof}
Suppose that Theorem \ref{oswThm1.2} holds in the special case that $X,Y\subset \R^n$ with $X\neq \emptyset$ and $\dim Y >n-1$ and all $n\geq 2.$ 

Fix such an $n\geq 2$ and suppose there exists Borel sets $X,Y\subset \R^n$ with $X\neq \emptyset$ and $\dim Y \in (k,k+1]$ (for some $1\leq k \leq n-2$) such that 
\[
\min\{\dim X+ \dim Y - k, k\}> \sup_{x\in X} \dim \pi_x(Y\setminus \{x\}) . 
\]
The key idea for this argument is to orthogonally project $Y$ onto a $(k+1)$-dimensional subspace $V\in \mathcal G(n,k+1)$ to reach a contradiction to the special case with $n = k+1$. By Marstrand's projection theorem (Theorem \ref{MARSTRAND}), for $\gamma_{n,k+1}$-almost every $V\in \mathcal G(n,k+1)$, $\dim P_V(Y) = \dim Y$ and $\dim P_V(X) = \min\{\dim X,k+1\}$. We claim that for all $x\in \R^n$ and all $V\in \mathcal G(n,k+1)$,
\[
\dim \pi_{P_V(x)}(P_V(Y)\setminus \{P_V(x)\}) \leq \dim \pi_x(Y\setminus \{x\}).
\]
To see why this completes the proof, notice that $P_V(Y) \subset V \simeq \R^{k+1}$ and $\dim P_V(Y) = \dim Y >k$. Hence, we can apply the special case to $P_V(X),P_V(Y) \subset V$. Doing so, we obtain
\begin{align*}
\min\{\dim X+ \dim Y - k, k\} &> \sup_{x\in X} \dim \pi_x(Y\setminus \{x\}) \\
&\geq \sup_{x\in X} \dim \pi_{P_V(x)}(P_V(Y)\setminus \{P_V(x)\}) \\
&\geq \min\{\dim P_V(X) + \dim P_V(Y) - k, k\}\\
&= \min\{\min\{\dim X,k+1\} + \dim Y - k,k\}.
\end{align*}
Using that $\dim Y>k$, notice that by case work on whether or not $\dim X\geq k+1$, we obtain a contradiction. 

Hence, we have reduced the Proposition to the claim 
\[
\dim \pi_{P_V(x)}(P_V(Y)\setminus \{P_V(x)\}) \leq \dim \pi_x(Y\setminus \{x\}).
\]
To prove this, we show that 
\begin{equation} \label{setcontainments}
\pi_{P_V(x)}(P_V(Y)\setminus \{P_V(x)\}) \subset \pi_0(P_V(\pi_x(Y\setminus \{x\})))
\end{equation}
where here $\pi_0$ is radial projection onto the origin. To see this, let $\theta \in \pi_{P_V(x)}(P_V(Y)\setminus \{P_V(x)\})$. Then, by the linearity $P_V$, we have 
\begin{align*}
\theta &= \frac{P_V(y) - P_V(x)}{|P_V(y)-P_V(x)|} \\
&= P_V\left(\frac{y-x}{|P_V(y-x)|}\right) \\
&= \frac{1}{|P_V((y-x)/|y-x|)|} \cdot P_V\left(\frac{y-x}{|y-x|}\right) = \pi_0(P_V(\pi_x(y))).
\end{align*}

Though this may seem like an arbitrary string of manipulations, the reader may find it instructive to draw a picture here. Notice that 
\[
\pi_x(Y\setminus \{x\}) \subset \mathbb{S}^{n-1} \quad \quad \text{and}\quad \quad V\cap \mathbb{S}^{n-1}\simeq \mathbb{S}^k.
\]
Thus, one can heuristically about $\pi_0\circ P_V$ as ``orthogonal projection'' projecting $\mathbb{S}^{n-1}$ onto the $k$-dimensional circle $V\cap \mathbb{S}^{n-1}$. In fact, this heuristic can be made precise by applying a projective transformation that maps $\mathbb{S}^{n-1}$ to $\R^{n-1}$ and $V\cap \mathbb{S}^{n-1}$ to a $k$-dimensional subspace in $\mathcal{G}(n-1,k)$. In this way, the integralgeometric argument of Orponen--Shmerkin--Wang can be throught of as conceptually similar to the (Marstrand-type) orthogonal projection argument Gan and I use in \cite{BrightGan}.

In either case, using \eqref{setcontainments} and that radial and orthogonal projections are Lipschitz maps, it follows that 
\[
\dim \pi_{P_V(x)} (P_V(Y)\setminus \{P_V(x)\}) \leq \dim \pi_0(P_V(\pi_x(Y\setminus\{x\}))) \leq \dim \pi_x(Y\setminus \{x\})
\]
for all $x\in \R^n$ and all $V\in \mathcal G(n,k+1)$.
This completes the proof.
\end{proof}

Theorem \ref{oswThm1.2} is one of \textit{two} key planar radial projection results in the paper of Orponen--Shmerkin--Wang. Their other main result utilizes Orponen and Shmerkin's $\epsilon$-improvement \cite{OrponenShmerkinEpsilonImprove} to the $(s,t)$-Furstenberg set problem (Theorem \ref{OSEpsilonFurst}), obtaining a continuum analogue to Proposition \ref{discretepinnedradproj}.(ii).

\begin{theorem}[\cite{OSW}] \label{oswThm1.1}
    Let $X,Y\subset \R^2$ be Borel with $X$ not contained in any line. Then,
    \[
    \sup_{x\in X} \dim \pi_x(Y\setminus \{x\}) \geq \min\{\dim X, \dim Y, 1\}.
    \]
\end{theorem}

Both Theorems \ref{oswThm1.2} and \ref{oswThm1.1} make use of a bootstrapping scheme introduced by Orponen and Shmerkin \cite{OrponenShmerkin22RadProj}. We briefly discuss the bootstrapping technique in the context of Beck's theorem in Section \ref{sec:OSWRen}. While we cannot give full justice to the technicalities of the method here, in \cite{BrightBushlingMarshallOrtiz}, Bushling, Marshall, Ortiz, and I wrote a study guide (mentored by Josh Zahl) on \cite{OSW} at the University of Pennsylvania Study Guide Writing Workshop 2023.


Orponen, Shmerkin, and Wang also explored a higher dimensional version of Theorem \ref{oswThm1.1}. In particular, for Borel sets $X,Y\subset \R^n$ with $X$ not contained on a $k$-plane, $\dim X >k-1$, and $\dim Y>k-1/k - \eta$ for a sufficiently small $\eta = \eta(n,k,\dim X)>0$,
\begin{equation} \label{oswThm4.2.ii}
\sup_{x\in X} \dim \pi_x(Y\setminus \{x\}) \geq \min\{\dim X, \dim Y,k\}.
\end{equation}
In the proof of this statement, they are able to apply an integralgeometric argument as in Proposition \ref{integralgeometricproof} to reduce to the case $k= n-1$. However, it's within the potential-theoretic proof of the case of $k=n-1$ that the $\eta$ parameter is needed (due to an application of \cite[Proposition 6.8]{ShmerkinWang}). Orponen, Shmerkin, and Wang in fact conjectured that \eqref{oswThm4.2.ii} should hold for $\dim Y >k-1$. This conjecture was obtained by Ren \cite{RenRadProj} a year later:

\begin{theorem}[\cite{RenRadProj}] \label{renradprojTheorem}
    Let $X,Y\subset \R^n$ be Borel sets with $X$ not contained in a $k$-plane. Then, 
    \[
    \sup_{x\in X} \dim \pi_x(Y\setminus \{x\}) \geq \min\{\dim X, \dim Y, k\}.
    \]
\end{theorem}

We make a few remarks on Ren's radial projection theorem to round out the section. Firstly, one may note that Theorem \ref{renradprojTheorem} is slightly different from \cite[Theorem 1.1]{RenRadProj} which supposes $\dim X,\dim Y\leq k$. However, this is a natural reduction for a number of reasons. Notice that if $\dim X,\dim Y>k$, this bound follows by Theorem \ref{BGan-LundPhamThu}. Furthermore, if $\dim X>k$ and $\dim Y\leq k$, this bound follows by Theorem \ref{BGan-Liu}. Hence, it's reasonable to assume that $\dim X\leq k$. Additionally, if $\dim X \leq k$, and $\dim Y>k$, the result follows by \eqref{oswThm4.2.ii}. As such, the only case left to consider for Theorem \ref{renradprojTheorem} was $\dim X,\dim Y \leq k$. This shows that one only needs consider the case of $\dim X,\dim Y\leq k$. That said, the full generality of Theorem \ref{renradprojTheorem} can be gleaned from Ren's proof methodology. In particular, to see that Ren's argument adapts for $\dim X,\dim Y>k$, simply note that his proof is Frostman-theoretic and lower the Frostman exponents throughout the proof. 

All of that said, we state Theorem \ref{renradprojTheorem} in this way for the purposes of proving Beck-type theorems in the subsequent chapter. The last remark we make is that the key ingredient towards proving Theorem \ref{renradprojTheorem} was a generalization of Orponen and Shmerkin's $\epsilon$-improvement on a Furstenberg set problem in the plane \cite{OrponenShmerkinEpsilonImprove} to higher dimensions \cite[Theorem 1.8]{RenRadProj}. More precisely, Ren obtains an $\epsilon$-improvement to the \textit{dual} $(s,t)$-Furstenberg set problem in all dimensions, see Section \ref{sec:DualFurstenberg}. After obtaining this $\epsilon$-improvement, the bootstrapping machinery of Orponen and Shmerkin \cite{OrponenShmerkinABC} utilized in \cite{OSW} immediately carries over to higher dimensions to obtain the desired result.

Lastly, as a corollary to Theorem \ref{renradprojTheorem}, one obtains a continuum analogue to Ungar's direction set bound (Proposition \ref{UngarResult}).\footnote{In the plane, this statement was conjectured by Orponen \cite[Conjecture 1.9]{OrponenRadProjSmoothness} and the corollary was noted in \cite[Corollary 1.4]{OSW}.}

\begin{corollary}
If $Y\subset \R^n$ is not contained in a $k$-plane, then 
\[
\dim S(Y) \geq \min\{\dim Y,k\}.
\]
\end{corollary}

\begin{proof}
    Applying Theorem \ref{renradprojTheorem}, for all $\epsilon>0$ there exists a $y\in Y$ such that 
    \[
    \dim \pi_y(Y\setminus \{y\}) \geq \min\{\dim Y, k\}-\epsilon.
    \]
    Thus, the statement follows from noting that for all $y\in Y$, 
    \[
    S(Y) \supset \pi_y(Y\setminus \{y\}). \qedhere
    \]
\end{proof}

\subsection{The High-Low Method: Lund--Pham--Vu's Bound} \label{BGanRadProjProof}

In this section, we provide a proof of Lund, Pham, and Vu's exceptional set estimate for radial projections (up to a worse implicit constant) via the high-low method. We state the result here for sake of exposition. Again, given $x\in \mathbb{F}_p^n$, we let $\pi_x: \mathbb{F}_p^n \setminus \{x\}\to \mathcal A(\mathbb{F}_p^n, 1)$ be such that $\pi_x(y)$ is the unique affine line through both $x,y$.
\vspace{-.2cm}
\begin{theorem}[Falconer-type Radial Projection Theorem over $\mathbb{F}_p^n$] \label{highlowLPT}
    Let $p$ be prime and let $Y\subset \mathbb{F}_p^n$ be such that $|Y| > 8n p^{n-1}$. Furthermore, let 
    \[
    E_s(Y) := \{y \in \mathbb{F}_p^n : |\pi_y(Y \setminus \{y\})| < s\}.
    \]
    Then, for all $0 < s \leq \min\{|Y|, \frac{1}{\sqrt{2}} p^{n-1}\}$,
    \[
    |E_s(Y)| \leq 8n p^{n-1} s |Y|^{-1}.
    \]
\end{theorem}

\vspace{-.3cm}

\begin{remark}
    Note that the $8n$ coefficients are worse than those in the combinatorial proof due to Lund--Pham--Vu \cite{LundPhamThu} (see Theorem \ref{lundphamthuTheorem}).
\end{remark}

\vspace{-.2cm}

The reason for including this proof is two-fold. Firstly, the proof echoes that of Falconer's projection theorem over $\mathbb{F}_p$ (see Section \ref{BGanProofs}). It may perhaps then be unsurprising that the methodology follows that of myself and Gan over $\R^n$ \cite{BrightGan}. That said, secondly, the finite field analogue of \cite{BrightGan} has yet to be recorded in the literature, and hence we include it here.

\vspace{-.2cm}
\begin{proof}[Proof of Theorem \ref{highlowLPT}]
    For ease of notation, we let $E := E_s(Y)$. Given $x\in \mathbb{F}_p^n$, let $\mathcal L_x = \pi_x(Y) \subset\mathcal A(\mathbb{F}_p^n, 1)$. Then, we let 
    \[
    f_y = \sum_{\ell \in \mathcal L_y} \mathbf{1}_\ell \quad \text{ and } \quad f = \sum_{y\in E} f_y. 
    \]
    \vspace{-.3cm}

    We now again consider the $L^2$-norm of $f\mathbf{1}_Y$. Notice that for $x\in \mathbb{F}_p^n$, 
    \[
    f(x) := |\{y \in E : \text{ there exists an } \ell \in \mathcal L_y \text{ with } x\in \ell\}|.
    \]
    Hence, notice that if $x\in Y \setminus E$, it follows that $f(x) = |E|$ as for each $y\in E$ there exists a unique $\ell \in \mathcal L_y$ containing both $x,y$. Relatedly, if $x\in Y\cap E$, $f(x) = |E|-1$. Thus, we have    
    \vspace{-.1cm}
    \begin{equation} \label{highlowRP0}
    |Y| |E|^2 \lesssim \sum_{x \in Y} |f(x)|^2.
    \vspace{-.1cm}
    \end{equation}
At this point, we need be mildly more careful about our analysis of the high and low terms. In particular, we notice that
\[
f = \left(\sum_{y\in E} \sum_{\ell \in \mathcal L_y} \mathbf{1}_\ell - \frac{1}{p^{n-1}}\right) + \left(\sum_{y\in E} \sum_{\ell \in \mathcal L_y} \frac{1}{p^{n-1}}\right) := f_{\text{high}} + f_{\text{low}}.
\]
We do this so that $\widehat{f}_{\text{high}} \subset \mathbb{F}_p^n \setminus 0$ and $\widehat{f}_{\text{low}} \subset \{0\}$. Hence by \eqref{highlowRP0}, we have 
\begin{equation} \label{highlowRP1}
|Y||E|^2 \lesssim \sum_{x\in Y} |f_{\text{high}}(x)|^2 + |f_{\text{low}}(x)|^2.
\end{equation}
We first see the low term does not dominate by the calculation:
\[
\sum_{x\in Y} |\sum_{y\in E} \sum_{\ell \in \mathcal L_y} \frac{1}{p^{n-1}}|^2 < |Y||E|^2s^2p^{-2(n-1)}.
\]
Given $s\leq \frac{1}{\sqrt{2}}p^{n-1}$, it follows that the low term is half the left hand side of \eqref{highlowRP0}. Hence, we only need bound the high term. By Plancherel,
\begin{align*}
    \sum_{x\in Y} |f_{\text{high}}(x)|^2 &:= \sum_{x\in Y} |\sum_{y\in E} \sum_{\ell \in \mathcal L_y} \mathbf{1}_\ell(x) - \frac{1}{p^{n-1}}|^2 \\
    &\leq \sum_{x\in \mathbb{F}_p^n} |\sum_{y\in E} \sum_{\ell \in \mathcal L_y} \mathbf{1}_\ell(x) - \frac{1}{p^{n-1}}|^2\\
    &= p^{-n} \sum_{\xi \in \mathbb{F}_p^n}|\sum_{y\in E} \sum_{\ell \in \mathcal L_y} \widehat{\mathbf{1}_\ell}(\xi) - \frac{1}{p^{n-1}}\widehat{\mathbf{1}_{\mathbb{F}_p^n}}(\xi)|^2. \numberthis\label{highlowRP2}
\end{align*} 
Notice that $\widehat{1_{\mathbb{F}_p^n}}= p^n\mathbf{1}_{\{0\}}$ via an identical proof to Proposition \ref{FourierOfPlane}. In particular, since $\widehat{\mathbf{1}_\ell}\xi) = p e^{2\pi i b\cdot \xi/p} \mathbf{1}_{\ell^\perp}$, it follows that $0\notin \supp \widehat{f}_{\text{high}}$ as claimed. Hence it is at this point that we'd like to use orthogonality to pull out one of the sums. However, we are now in the predicament that there can exist many $y,y'\in E$ such that $\mathcal L_y \cap \mathcal L_{y'} \neq \emptyset$. To deal with this, we introduce the following notation, grouping lines in terms of their directions $\Theta$:
\[
\mathcal L = \bigcup_{y\in E} \mathcal L_y := \bigcup_{\theta \in \Theta} \mathcal L_\theta
\]
where 
\begin{align*}
&\Theta = \{\theta \in \mathcal G(\mathbb{F}_p^n, 1) : L(\theta, y) \in \mathcal L \text { for some } y\in E\}\\ &\text{ and }\,\,
\mathcal L_\theta = \{\ell \in \mathcal L : \ell = L(\theta, b) \text{ for some } b \in \mathbb{F}_p^n\}.
\end{align*}
Lastly, let 
\[
n_\ell = |\{y \in E : \ell \in \mathcal L_y\}|.
\]
Using this notation in \eqref{highlowRP2}, we have
\begin{align*}
    \sum_{x\in \mathbb{F}_p^n} |f_{\text{high}}(x)|^2 &= p^{-n} \sum_{\xi \in \mathbb{F}_p^n} |\sum_{\theta \in \Theta} \sum_{\ell \in \mathcal L_\theta} n_\ell (\widehat{\mathbf{1}_\ell}(\xi) - p\mathbf{1}_{\{0\}}(\xi))|^2.
\end{align*}
For each $\xi \neq 0$, there exist $\leq \binom{n-1}{n-2}\leq np^{n-2}$-many $\theta^\perp \in \mathcal G(n,n-1)$ containing $L(\xi, 0)$. Hence,
\begin{align*}
\sum_{x\in \mathbb{F}_p^n} |f_{\text{high}}(x)|^2&\leq np^{n-2} \sum_{\theta \in \Theta} p^{-n} \sum_{\xi \in \mathbb{F}_p^n}|\sum_{\ell \in \mathcal L_\theta} n_\ell (\widehat{\mathbf{1}_\ell}(\xi) - p\mathbf{1}_{\{0\}}(\xi))|^2.
\intertext{By Plancherel, and using that the lines in $\mathcal L_\theta$ are pairwise disjoint, we have}
&\leq p^{n-2} \sum_{\theta \in \Theta} \sum_{x\in \mathbb{F}_p^n}|\sum_{\ell \in \mathcal L_\theta} n_\ell (\mathbf{1}_\ell(x) - \frac{1}{p^{n-1}})|^2 \\
&\leq np^{n-2} \sum_{\theta \in \Theta} \sum_{x\in \mathbb{F}_p^n}\left(|\sum_{\ell \in \mathcal L_\theta} n_\ell \mathbf{1}_\ell(x)|^2  +|\sum_{\ell \in \mathcal L_\theta} n_\ell p^{-(n-1)}|^2\right) \\
&\leq np^{n-2} \sum_{\ell \in \mathcal L} n_\ell^2 \cdot \sum_{x\in \mathbb{F}_p^n} |\mathbf{1}_\ell(x)|^2  + n \sum_{\theta \in \Theta} |\mathcal L_\theta|\sum_{\ell \in \mathcal L_\theta} n_\ell^2\\
&\leq 2np^{n-1} \sum_{\ell \in \mathcal L} n_\ell^2. \numberthis \label{highlowRP3}
\end{align*}
The last line follows by noting $|\mathcal L_\theta| \leq p^{n-1}$. We now bound the last sum, which follows by a double-counting argument:
\begin{align*}
    \sum_{\ell \in \mathcal L} n_\ell^2 = \sum_{\ell \in \mathcal L} |\{y,y'\in E : \ell \in \mathcal L_y \cap \mathcal L_{y'}\} 
    &= \sum_{y,y' \in E} |\{\ell \in \mathcal L : \ell \in \mathcal L_y \cap \mathcal L_{y'}\} \\
    &\leq \sum_{y\in E} |\mathcal L_y| + \sum_{y\in E}\sum_{y' \neq y} |\mathcal L_y \cap \mathcal L_{y'}| \\
    &\leq |E|s + |E|^2.
\end{align*}
The last inequality follows as $|\mathcal L_y| < s$ for all $y\in E$, and given $y\neq y'$, $|\mathcal L_y \cap \mathcal L_{y'}| \leq 1$. Plugging this into \eqref{highlowRP0}-\eqref{highlowRP3}, we have
\[
\frac{1}{2}|Y||E|^2 \leq 2np^{n-1} (|E|s + |E|^2).
\]
If the second term dominates, then $|Y| \leq 8n p^{n-1}$ which contradicts the assumption that $|Y| >8n p^{n-1}$. So, the first term must dominate and thus 
\[
|E| \leq 8np^{n-1} s |Y|^{-1}. \qedhere
\]
\end{proof}

\chapter{Problems in Projection Theory} \label{ch:problems}
As we have seen throughout the entirety of Chapter \ref{ch:Topics}, projection theory goes hand in hand with studying how points interact with lines. Beck-type and Falconer-type distance problems, which will be the focus of this chapter, are no different. Beck-type problems and Falconer-type problems will be covered in Sections \ref{SEC-BECKTYPE} and \ref{SEC:FALCONERTYPE} respectively.

For Beck-type problems, we begin in Section \ref{sec:discretebecklines} with proving discrete Beck's theorem as a consequence of the Szemer\'edi--Trotter theorem. We will then discuss how Beck originally obtained his result \cite{Beck83} without the strength of the Szemer\'edi--Trotter theorem, using a bootstrapping scheme in its place. This bootstrapping scheme is precisely what motivated Orponen, Shmerkin, and Wang's work on radial projections \cite{OSW} that implied the continuum version of Beck's theorem in $\R^2$, and which was later generalized by Ren to $\R^n$ \cite{RenRadProj}. Both of these continuum results are discussed in Section \ref{sec:OSWRen}, and are derived from the radial projection results in Section \ref{ss:radprojctm}. Furthermore, we present a continuum analogue of the discrete Erd\H{o}s--Beck theorem due to myself and Marshall \cite{BrightMarshall}, which uses radial projections and a dual Furstenberg set estimate due to myself, Fu, and Ren \cite{BrightFuRen} (Theorem \ref{BFR-DUALFURST}).

For Falconer-type distance problems, we begin in Section \ref{erdosprob} with the Erd\H{o}s distance problem. We will particularly focus on the work of Guth and Katz \cite{GuthKatz} which utilized improvements to the Szemer\'edi--Trotter theorem in $\R^3$. Then, in Section \ref{falconerprob}, we discuss the continuum analogue of the Erd\H{o}s distance problem known as the Falconer distance problem. This problem has made use of, and motivated progress on, a number of tools in harmonic analysis including projection theory. As such, we discuss the ways in which the topics covered in Chapter \ref{ch:Topics} were utilized to this end, especially Orponen's and Ren's radial projection theorems (Theorems \ref{orponenradprojtheorem} and \ref{renradprojTheorem}). Lastly, in Section \ref{sec:BMS}, we present and prove bounds for the Falconer-type theorem for dot products due to Marshall, Senger, and I \cite{BrightMarshallSenger}. The proof uses standard estimates on orthogonal and radial projections, and using more modern results yields mild improvements. As such, we present the proof both as a way to conclude our exposition of projection theory covered in this thesis, and as a way to encourage readers to go read more in this area.

\section{Beck-type Theorems} \label{SEC-BECKTYPE}

Recall that for $X\subset \R^n$, we let 
\[
\mathcal L(X) := \{\ell : |X\cap \ell| \geq 2\} \subset \mathcal A(n,1).
\]

\subsection{Discrete Beck-type Theorems} \label{sec:discretebecklines}

Let us first restate discrete Beck's theorem from the introduction (Theorem \ref{thm:introBeck}) more quantitatively for illustrative purposes. 

\begin{theorem}[Discrete Beck's theorem \cite{Beck83}] \label{thm:ch4Beck}
There exists a constant $c_0\geq 1$ such that for all $C\geq c_0$ the following holds. For all $X\subset \R^n,$ either
\begin{itemize}
    \item[1)] there exists a line $\ell$ with $|X\cap \ell| \geq |X|/C$, or 
    \item[2)] $|\mathcal L(X)| \geq \frac{1}{2C^2}|X|^2$.
\end{itemize}
\end{theorem}

We make a few remarks before we prove discrete Beck's theorem. Firstly, note that an explicit choice of $c_0$ can be found within \cite{PayneWood1/37Beck} which makes use of visibility graphs. Secondly, as in the proof of Corollary \ref{genericprojCorollary0}, note that Beck's theorem in $\R^n$ reduces to the statement in $\R^2$ by orthogonally projecting onto a generic hyperplane. Lastly, as an immediate consequence of discrete Beck's theorem, one obtains the following discrete radial projection estimate.

\begin{corollary} \label{corollaryBeckRadProj}
There exists a constant $0 < c < 1$ such that the following holds.
For all $X\subset \R^n$ finite such that $|X\setminus \ell| \geq c|X|$ for all $\ell \in \mathcal A(n,1)$, 
\[
\max_{x\in X} |\pi_x(X\setminus \{x\})| \gtrsim |X|.
\]
\end{corollary}

\begin{proof}
Let $0< c< 1$ be large enough such that $c^{-1}\geq c_0$ and such that condition 1) of Theorem \ref{thm:ch4Beck} fails for all $X$ such that $|X\setminus \ell|\geq c|X|$. Then there exists a constant $K = K(c)>0$ such that $|\mathcal L(X)|\geq K |X|^2.$ Thus, let
\[
G := \{x\in X : |\pi_x(X\setminus \{x\})| \geq K|X|/2\}.
\]
Furthermore, for each $x\in \R^n$ (and thus in particular each $x\in G$), $$|\pi_x(X\setminus \{x\})| \leq |X|.$$
Notice then that each $x \in G$ determines at most $|X|$-many lines (through $x$) in $\mathcal L(X)$. Thus, 
\[
K|X|^2 \leq |\mathcal L(X)| < |G| |X| + |X\setminus G | \cdot \frac{K|X|}{2} \leq |G||X| + \frac{K|X|^2}{2}.
\]
Rearranging the above inequality and dividing by $|X|$ gives that 
\[
\frac{K|X|}{2} < |G|,
\]
which implies $G \neq \emptyset$ and thus the corollary holds with constant $K/2.$
\end{proof}

We make a few remarks before moving on. Firstly, the proof of Corollary \ref{corollaryBeckRadProj} in fact shows that $|G| \sim |X|$, and in particular the set of ``bad'' pins $X\setminus G$,\footnote{Or perhaps more accurately, \textit{exceptional} pins, see \cite[Corollary 1.2]{OSW}.} must be significantly smaller than $G$ relative to the size of $X$. We will see this in the continuum setting as a consequence of Orponen--Shmerkin--Wang's and Ren's radial projection results; see the proof of Theorem \ref{thm:ch4OSWBeck}. Secondly, after presenting the proof of Beck's theorem (Theorem \ref{thm:ch4Beck}), we will present a generalized statement counting lines between two sets $X,Y$. This will generalize Corollary \ref{corollaryBeckRadProj} in a natural way, see Corollary \ref{CorollaryBeckRadProj2Var}. We now prove discrete Beck's theorem using Szemer\'edi--Trotter.

\begin{proof}[Proof of Theorem \ref{thm:ch4Beck} using Theorem \ref{SZEMEREDITROTTER}]
Naively, one would want to say that every 2 points in $X$ determine a line in $\mathcal L(X)$, but the issue is that numerous pairs of points in $X$ can determine the same line in $\mathcal L(X)$. However, if there exists a line in $\mathcal L(X)$ that is determined by \textit{many} (i.e. $\sim |X|^2$) pairs of points in $X$, then this gives Statement (1) in Theorem \ref{thm:ch4Beck}. Relatedly, if \textit{most} lines in $\mathcal L(X)$ are determined by $\leq C$-many points in $X$ for an absolute constant $C>0$ independent of $X$, then one can conclude Statement (2) in Theorem \ref{thm:ch4Beck}. Thus, the overall proof idea is to show this dichotomy by bounding the number of lines that are of middling richness in $X$. 

Let $C>0$ be a large constant to be determined later, and let
\[
L_j = \{\ell \in \mathcal L(X) : 2^j \leq |X\cap \ell| \leq 2^{j+1}\} \quad \text{ and } \quad 
\mathcal L_M = \bigcup_{j = \log C}^{\log(|X|/C)} L_j.
\]
Notice that $L_j \subset \mathcal L_{2^j}(X)$, the set of $2^j$-rich points in $X$. Hence, we have 
\begin{equation}\label{richlinesbeckLemma}
|L_j| \lesssim \frac{|X|^2}{2^{3j}} + \frac{|X|}{2^j}.
\end{equation}
We use this to bound the number of pairs of points in $X$ that lie on a line in $L_j$. In particular, notice that for all $\ell \in L_j$, there are $\sim 2^{2j}$-many pairs of points in $X\times X$ that determine $\ell$. Therefore,  
\[
    |\{(x,x') \in X^2 : x\neq x', \exists \ell \in \mathcal L_M \text{ with } x,x' \in \ell\}| \leq \sum_{j=\log C}^{\log(|X|/C)} 2^{2j} |L_j|
\]
Applying \eqref{richlinesbeckLemma} and evaluating the geometric series, we thus have
\[
\sum_{j=\log C}^{\log(|X|/C)} 2^{2j} |L_j| \lesssim \sum_{j = \log C}^{\log(|X|/C)} 2^{-j}|X|^2 + 2^j |X| \lesssim \frac{|X|^2}{C}.
\]
Choosing $C$ sufficiently large, we thus obtain 
\[
|\{(x,x') \in X^2 : x\neq x', \exists \ell \in \mathcal L_M \text{ with } x,x' \in \ell\}| \leq \frac{|X|^2}{100}.
\]
Hence, either there exists a line $\ell \in L_j$ for some $j >\log (|X|/C)$, i.e.
\[
|X\cap \ell| \geq |X|/C,
\]
which is (1) in Theorem \ref{thm:ch4Beck}, or there are $\sim |X|^2$-many lines $\ell \in \mathcal L(X)$ with $|X\cap \ell| < C$. Thus, it follows that 
\[
|\mathcal L(X)| \gtrsim |X|^2 /C^2,
\]
concluding the proof of the dichotomy in discrete Beck's theorem.
\end{proof}

In fact, we are able to generalize this to lines that contain pairs of distinct points from two (possibly different) sets $X,Y$. To do so, we will need to be more careful about our dyadic decomposition in the proof of the single set case as we want to allow for the case where $|X|$ may be significantly larger than $|Y|$ (or vice versa).

\begin{notation}
Given $X,Y \subset\R^n$, let $\mathcal L(X,Y)\subset \mathcal A(n,1)$ be the set of lines containing pairs of distinct points $(x,y)\in X\times Y$, i.e.
\[
\mathcal L(X,Y) := \{\ell \in \mathcal A(n,1) : \exists (x,y) \in X\times Y \text{ such that } x\neq y \text{ and } x,y\in \ell\}.
\]
\end{notation}

\begin{theorem}\label{thm:ch4Beckbivariable}
There exists a constant $c_0\geq 1$ such that for all $C\geq c_0$ the following holds. For all $X,Y\subset \R^n,$ either
\begin{itemize}
    \item[1)] there exists a line $\ell$ with 
    \[
    |X\cap \ell| \geq |X|/C \quad \text{ and } \quad |Y \cap \ell | \geq  |Y|/C, \text{ or }
    \]
    \item[2)] $|\mathcal L(X,Y)|\geq \frac{1}{2C^2} |X||Y|.$
\end{itemize}
\end{theorem}

\begin{proof}
    For ease of notation, let $m = \min\{|X|,|Y|\}$ and $M = \max\{|X|,|Y|\}$. Furthermore, Given $C>0$ to be determined, let
    \begin{align*}
        L_j &:= \{\ell \in \mathcal L(X,Y) : 2^{j} \leq \min\{|X\cap \ell|, |Y\cap \ell\} < 2^{j+1}\}, \\
        L_{j,k} &:= \{\ell \in L_j : 2^{k} \leq \max\{|X\cap \ell|, |Y\cap \ell\} < 2^{k+1}\}, \text{ and } \\
        \mathcal L_M &:= \bigcup_{j = \log C}^{\log(M/C)} L_j = \bigcup_{j = \log C}^{\log(M/C)} \bigcup_{k = j}^{\log(m/C)} L_{j,k}.
    \end{align*}
    Now, our goal is to find $c_0$, and thus $C$, sufficiently large such that the pairs of distinct points $(x,y) \in X\times Y$ that determine a line in $\mathcal L_M$ is $\leq \frac{|X||Y|}{100}$. Notice that for each line in $\ell \in L_{j,k}$, by definition, there are $\sim 2^{j+k}$-many pairs of points in $X\times Y$ that determine $\ell$. Thus, we have 
    \begin{equation} \label{countingpairsowo}
    |\{(x,y) \in X\times Y : x\neq y, \exists \ell \in \mathcal L_M \text{ with } x,y \in \ell\}| \lesssim \sum_{j,k} 2^{j+k}|L_{j,k}|.
    \end{equation}
    We now bound the size of $|L_{j,k}|$ for all $j,k$. Notice that for all $j,k$, 
    \[
    L_{j,k} \subset (\mathcal L_{2^j}(X) \cap \mathcal L_{2^j}(Y))
    \]
    by definition of $L_j$. While we can bound the right hand side in terms of rich points of $X,Y$, this set inclusion is insufficient for proving the desired claim. To obtain better bounds, one needs notice that for every line $\ell \in \mathcal A(n,1)$,
    \[
    \{\min\{|X\cap \ell|, |Y\cap \ell|\}, \max\{|X\cap \ell|, |Y\cap \ell|\}\} = \{|X\cap \ell|, |Y\cap \ell|\}.
    \]
    Hence, it follows that 
    \[
    L_{j,k} \subset (\mathcal L_{2^j}(X) \cap \mathcal L_{2^k}(Y)) \cup (\mathcal L_{2^k}(X) \cap \mathcal L_{2^j}(Y)).
    \]
    By symmetry, it suffices to bound the cardinality of one of the terms in the above union. To do so, notice that 
    \[
    |L_{j,k}| \leq \min\{|\mathcal L_{2^j}(X)|, |\mathcal L_{2^k}(Y)|\} \leq |\mathcal L_{2^j}(X)|^{1/2}|\mathcal L_{2^k}(Y)|^{1/2}.
    \]
    Hence, it follows that for all $j,k$,
    \[
    |L_{j,k}| \lesssim \frac{|X||Y|}{2^{\frac{3}{2}(j+k)}} + \frac{|X|^{1/2}|Y|^{1/2}}{2^{\frac{1}{2}(j+k)}}.
    \]
    Plugging this into \eqref{countingpairsowo} and evaluating the geometric series, we see 
    \[
    |\{(x,y) \in X\times Y : x\neq y, \exists \ell \in \mathcal L_M \text{ with } x,y \in \ell\}| \lesssim \frac{Mm}{C}.
    \]
    Hence, choosing $C$ sufficiently large we obtain 
    \[
    |\{(x,y) \in X\times Y : x\neq y, \exists \ell \in \mathcal L_M \text{ with } x,y \in \ell\}| \leq \frac{Mm}{100} = \frac{|X||Y|}{100}.
    \]
    The proof of the desired claim follows precisely as before. Either there exists a line $\ell \in L_{j,k}$ for some $k > \log(M/C)$, and thus 
    \[
    \min\{|X\cap \ell|, |Y\cap \ell|\} \geq C^{-1}M = C^{-1}\max\{|X|,|Y|\},
    \]
    or there exists $\sim \left(\frac{C-1}{C}\right) |X||Y|$ pairs of points in $$X\times Y \setminus \{(x,x) : x\in X\cap Y\}$$ that lie on lines with multiplicity at most $C$ for \textit{both} $X$ and $Y.$ Thus, choosing $c_0$ sufficiently large we have $|\mathcal L(X,Y)|\geq \frac{1}{2C^2} |X||Y|$.
\end{proof}

We now prove the discrete radial projection result that follows as a corollary, providing a full discrete analogue to Orponen, Shmerkin, and Wang's bound \cite{OSW} (Theorem \ref{oswThm1.1}).

\begin{corollary}\label{CorollaryBeckRadProj2Var}
There exists a constant $0 < c< 1$ such that the following holds. Given $X \subset \R^n$ finite, suppose that $|X\setminus \ell| \geq c|X|$ for all $\ell \in \mathcal A(n,1).$ Then, for all $Y\subset \R^n$ finite, 
\[
\max_{x\in X} |\pi_x(Y\setminus \{x\})| \gtrsim \min\{|X|, |Y|\}.
\]
\end{corollary}

\begin{remark}
Based on Proposition \ref{discretepinnedradproj}.(c), it is reasonable to expect that given there does not exist a line with a large proportion of both $X$ and $Y$, 
\[
\max_{x\in X} |\pi_x(Y\setminus \{x\})| \gtrsim \min\{|X|^{1/2}|Y|^{1/2}, |Y|\}.
\]
We do not pursue this further here.
\end{remark}

\begin{proof}[Proof of Corollary \ref{CorollaryBeckRadProj2Var} using Theorem \ref{thm:ch4Beckbivariable}]
Let $c_0$ be the implicit constant in Theorem \ref{thm:ch4Beckbivariable} Statement (1). Choose $c$ sufficiently large\footnote{Notably, dependent on $c_0$ and independent of $|X|$.} such that for all lines $\ell \in \mathcal A(n,1)$, 
\[
|X\cap \ell| < c_0 |X|.
\]
Thus, Statement (1) of Theorem \ref{CorollaryBeckRadProj2Var} always fails, and thus for all $Y\subset \R^2$ finite it follows that
\[
|\mathcal L(X,Y)|\geq K|X||Y|
\]
for some $K = K(c)>0$.

We now proceed via case work. If $|X| \geq |Y|$, let 
\[
G := \{x \in X : |\pi_x(Y\setminus \{x\})| \geq K|Y|/2\}.
\]
Notice that for all $x\in \R^n$ (and thus all $x\in G$), $|\pi_x(Y\setminus \{x\})| \leq |Y|$.
\[
K|X||Y| \leq |\mathcal L(X,Y)| < |G||Y| + |X \setminus G| \cdot \frac{K|Y|}{2} \leq |G||Y| + \frac{K|X||Y|}{2}.
\]
Rearranging the above inequality and dividing by $|Y|$ gives that $G \neq \emptyset$, and in particular there exists an $x\in X$ with 
\[
|\pi_x(Y\setminus \{x\})| \gtrsim |X| \geq \min\{|X|,|Y|\}.
\]

Now suppose $|X| \leq |Y|-1$. If this is the case, then, let 
\[
G := \{x \in X : |\pi_x(Y\setminus \{x\})| \geq K|X|/2\}.
\]
Notice that for all $x\in \R^n$ (and thus all $x\in G$), $|\pi_x(Y\setminus \{x\})| \leq |Y|$, so
\[
K|X||Y| < |G||Y| + |X \setminus G| \cdot \frac{K|X|}{2} \leq |G||Y| + \frac{K|X|^2}{2} \leq |G||Y| + \frac{K|X||Y|}{2}.
\]
Rearranging the above inequality gives that $G \neq \emptyset$, and in particular there exists an $x\in X$ with the desired property:
\[
|\pi_x(Y \setminus \{x\})| \gtrsim |X| = \min\{|X|,|Y|\}. \qedhere
\]
\end{proof}

Interestingly, up to worse implicit constants, one can also obtain Beck's theorem from their discrete radial projection estimates (Corollary \ref{CorollaryBeckRadProj2Var}) and discrete dual Furstenberg set estimates (Proposition \ref{discDFurstTotalProp}).\footnote{Of course, such a proof of Beck's theorem would be circular reasoning. For this reason, it would be interesting to have an incidence theoretic approach to obtaining Corollary \ref{CorollaryBeckRadProj2Var} directly, especially one using the Szemer\'edi--Trotter theorem.} This is highlighted well by the following proof of the Erd\H{o}s--Beck theorem, which was conjectured by Erd\H{o}s and proven by Beck in 1982 \cite{Beck83}.

\begin{theorem}[Discrete Erd\H{o}s--Beck theorem \cite{Beck83}] \label{thm:ch4erdosBeck}
    Let $X\subset \R^n$ be a finite set and $0 \leq t \leq |X|-2$
    be such that 
    \[ 
    \max_{\ell \in \mathcal L(X)} |X\cap  \ell| =|X|-t.
    \]
    Then, $|\mathcal L(X)|\gtrsim |X|t.$
\end{theorem}

\begin{proof}
Let $\ell \in \mathcal L(X)$ be such that $|X\cap \ell| = |X|-t.$
We then claim that $\mathcal L(X)$ is a discrete $(|X|-t,t)$-dual Furstenberg set. In particular, notice $\mathcal L(X) \supset \bigcup_{x\in X\setminus \ell} \mathcal L_x$, where 
\[
\mathcal L_x = \pi_{x}^{-1}(\pi_x(X\setminus \{x\})) \subset \mathcal A(n,1).
\]
Furthermore, we have $|X\setminus \ell| = t$ and for all $x\in X\setminus \ell$,
\[
|\mathcal L_x| \geq |\pi_x(X\cap \ell)| = |X|-t.
\]
Thus, applying Proposition \ref{discDFurstTotalProp}, we see 
\[
|\mathcal L(X)| \gtrsim (|X|-t)t = |X|t - t^2 \gtrsim |X|t. \qedhere
\]
\end{proof}

We conclude this section by discussing how Beck proved his result in \cite{Beck83}, as at the time he did not have access to Szemer\'edi--Trotter. In fact, as we noted earlier, both Beck's theorem and the Szemer\'edi--Trotter theorem were published in the same volume of the same journal in 1983! Beck's original argument made use of a bootstrapping technique. In particular, though he did not have access to the Szemer\'edi--Trotter theorem for rich lines, Beck was able to independently obtain that 
\[
|\mathcal L_r(X)| \lesssim \frac{|X|^2}{r^{2 + \epsilon}} + \frac{|X|}{r}
\]
for $\epsilon = \frac{1}{20}$ \cite[Theorem 1.5]{Beck83} using the ``coordinate method'' of himself and Spencer \cite{BeckSpencer}. Using this $\epsilon$-improvement to the Cauchy--Schwarz estimate for $r$-rich lines in place of \eqref{richlinesbeckLemma}, one can obtain Beck's theorem. For a proof of this approach for an arbitrary $\epsilon>0$, see \cite[Appendix A]{BrightBushlingMarshallOrtiz}. It is Beck's original ``$\epsilon$-improvement'' argument that motivated the proof of the continuum version of Beck's theorem which we turn to now.

\subsection{A Continuum Beck's Theorem}\label{sec:OSWRen}

Let us first restate the continuum version of Beck's theorem.

\begin{theorem}[Continuum Beck's theorem \cite{OSW}] \label{thm:ch4OSWBeck}
    For $X\subset \R^2$ Borel, either 
    \begin{itemize}
        \item[1)] there exists a line $\ell$ such that $\dim (X\setminus \ell) < \dim X$, or
        \item[2)] $\dim \mathcal L(X) \geq \min \{2\dim X, 2\}$.
    \end{itemize}
\end{theorem}

The above statement follows as a corollary of Theorem \ref{oswThm1.1} from the same paper, but it's first and foremost worth mentioning that Beck's 1983 argument in fact inspired Orponen, Shmerkin, and Wang's approach to studying radial projections. They describe this nicely in the following quote:

\begin{quote}
    ``The starting point of this work was the observation that while Beck's Theorem can be deduced from the Szemer\'edi--Trotter incidence bound, in fact it doesn't require the full strength of Szemer\'edi--Trotter-- any ``$\epsilon$-improvement'' over the elementary double counting bound is enough.''\,\cite{OSW}
\end{quote}

In a similar vein, Orponen and Shmerkin's $\epsilon$-improvement to the $(s,t)$-Furstenberg set problem (Theorem \ref{OSEpsilonFurst}) was the key ingredient towards Orponen, Shmerkin, and Wang's radial projection result (Theorem \ref{oswThm1.1}). We restate Theorem \ref{oswThm1.1} here for the sake of convenience.

\begin{theorem}[\cite{OSW}, Theorem \ref{oswThm1.1}] \label{ch4OSWBoundRestated}
    Let $X,Y\subset \R^2$ be Borel with $X$ not contained in any line. Then,
    \begin{equation}\label{OSWExceptionalineq}
    \sup_{x\in X} \dim \pi_x(Y\setminus \{x\}) \geq \min\{\dim X, \dim Y, 1\}.
    \end{equation}
\end{theorem}

Using the above radial projection estimate and dual Furstenberg set estimates, one can obtain their continuum version of Beck's theorem.

\begin{proof}[Proof of Corollary \ref{thm:ch4OSWBeck}]
    Let $0 \leq \sigma < \min\{\dim X, 1\}$ and 
    \[
    G := \{x \in X : \dim \pi_x(X\setminus \{x\}) >\sigma\}.
    \]
    We claim $\dim G = \dim X$. If this is the case, then it follows that for each $x\in G$, there exists a set of lines $\mathcal L_x \subset \mathcal L(X)$ containing $x$ and such that $\dim\mathcal L_x \geq \sigma$. Note that here we use that we can identify $\pi_x(X\setminus \{x\})$ with a set of lines through the point $x$ containing another point in $X\setminus \{x\}$. The fact that $\dim \mathcal L_x = \dim \pi_x(X\setminus \{x\})$ follows rather directly from this identification, see \cite[Lemma 20]{BrightMarshall}. Thus, if $\dim G = \dim X$, it follows that $\mathcal L(X)$ is a $(\sigma, \sigma)$-dual Furstenberg set as 
    \[
    \mathcal L(X) \supset \bigcup_{x\in G} \mathcal L_x.
    \]
    The claim follows by applying Theorem \ref{BFR-DUALFURST} and sending $\sigma \nearrow \min\{\dim X, 1\}$.

    It remains to prove $\dim G = \dim X$. Suppose for the sake of contradiction $\dim X > \dim G$. Notice that this implies that $\dim (X\setminus G) = \dim X$. Then, we claim that $X\setminus G$ must be contained in a line. If this weren't the case, we could apply Theorem \ref{ch4OSWBoundRestated} and see
    \[
    \sup_{x\in X\setminus G} \dim\pi_x(X\setminus \{x\}) \geq \min\{\dim X, \dim (X\setminus G), 1\} >\sigma.
    \]
    This contradicts the definition of $X\setminus G$. However, if $G$ is contained in the line $\ell$, then it follows that 
    \[
    \dim X = \dim (X\setminus G) \geq \dim (X\setminus \ell) = \dim X
    \]
    where the last equality follows by our hypotheses on $X.$ This gives a contradiction to the assumption that $\dim X > \dim G$, concluding the proof.
\end{proof}

\begin{remark}
    We use the notation $G$ here to echo the discussion of discrete Beck-type theorems, see Corollaries \ref{corollaryBeckRadProj} and \ref{CorollaryBeckRadProj2Var}.
\end{remark}

We make a few remarks before moving on. Firstly, note that the dual Furstenberg set estimate due to myself, Fu, and Ren \cite{BrightFuRen} came out \textit{after} the paper of Orponen--Shmerkin--Wang. However, as we noted in Section \ref{sec:DualFurstenberg}, in the plane the dual $(s,t)$-Furstenberg set problem is precisely dual to the $(s,t)$-Furstenberg set problem. Thus, Orponen, Shmerkin, and Wang were able to appeal to the double counting estimates due to H\'era--Shmerkin--Yavicoli and Lutz--Stull \cite{HeraShmerkinYavicoli,LutzStull} (see Theorem \ref{ElementaryFurstenberg}) to conclude their proof of the continuum version of Beck's theorem in the plane.

Secondly, as a consequence of Ren's radial projection result \cite{RenRadProj} (Theorem \ref{renradprojTheorem}), one obtains the a continuum version of Beck's theorem in $\R^n$.

\begin{theorem}[Corollary 1.5 in \cite{RenRadProj}] \label{thm:ch4RenBeck}
    For $X\subset \R^n$ Borel, either 
    \begin{itemize}
        \item[1)] there exists a $k$-plane $H \in \mathcal A(n,k)$ such that $\dim (X\setminus H) < \dim X$, or
        \item[2)] $\dim \mathcal L(X) \geq \min \{2\dim X, 2k\}$.
    \end{itemize}
\end{theorem}

We omit the proof of the above theorem as the proof follows nearly verbatim to that of Theorem \ref{thm:ch4Beck}. It's also worth noting that Ren's paper also came out before the dual Furstenberg set estimate due to himself, Fu, and I. To this end, note that one can obtain continuum Beck-type theorems from the more quantitative ``thin tubes'' estimates Orponen--Shmerkin--Wang and Ren obtain that imply the dimensional bounds in Theorems \ref{oswThm1.1} and \ref{renradprojTheorem}; see \cite[Corollary 2.22]{OSW} and \cite[Theorem 1.13]{RenRadProj}. 

Lastly, in the same vein as Corollary \ref{CorollaryBeckRadProj2Var}, one can in fact apply Theorem \ref{ch4OSWBoundRestated} to studying $\dim \mathcal L(X,Y)$. In particular, one can show the following as a consequence of Theorem \ref{renradprojTheorem}.

\begin{theorem}
    For $X,Y\subset \R^n$ Borel, either 
    \begin{itemize}
        \item[1)] there exists a $k$-plane $H \in \mathcal A(n,k)$ such that $\dim (X\setminus H) < \dim X$, or
        \item[2)] $\dim \mathcal L(X, Y) \geq \min \{2\dim X, 2\dim Y, 2k\}$.
    \end{itemize}
\end{theorem}

\noindent We leave the proof as an exercise as it follows in nearly the same manner as the proof of Theorem \ref{thm:ch4OSWBeck} with $\mathcal L_x(Y)$ in place of $\mathcal L_x(X)$.

\subsection{A Continuum Erd\H{o}s--Beck Theorem} \label{sec:BMarshall}

Shortly after studying the paper of Orponen, Shmerkin, and Wang at the University of Pennsylvania Study Guide Workshop 2023, Marshall and I learned of the Erd\H{o}s--Beck theorem (Theorem \ref{thm:ch4erdosBeck}). We conjectured that a continuum version of the Erd\H{o}s--Beck theorem should hold, and in 2024 we obtained the following:

\begin{theorem}[A Continuum Erd\H{o}s--Beck theorem \cite{BrightMarshall}] \label{BM-ERDOSBECK}
    Let $X\subset \R^n$ be Borel. If there exists a $k$-plane $H \in \mathcal A(n,k)$ such that $\dim (X\setminus H) < \dim X$, we let $0 < t \leq \dim X$ be such that $\dim (X\setminus P) \geq t$ for all $P \in\mathcal A(n,k)$. Then,
    \[
    \dim \mathcal L(X) \geq \dim X + t.
    \]
\end{theorem}

\begin{remark} \label{remarkonDimX<k}
    Note that here, we need not include the $2k$ term in the lowerbound in $\mathcal L(X)$. Indeed, if $\dim X>k$, there does not exist a $k$-plane $H$ such that $\dim (X\setminus H) < \dim X$. Thus, under the hypotheses of Theorem \ref{BM-ERDOSBECK}, for all $0 < t \leq \dim X$
    \[
    \dim X + t \leq 2k.
    \]
\end{remark}

\begin{proof}[Proof of Theorem \ref{BM-ERDOSBECK}]
    If there exists a $k$-plane $H$ such that $\dim (X\setminus H) < \dim X$, it therefore follows that 
    \[
    \dim (X\cap H) = \dim X.
    \]
    We proceed as we did in the proof of the Erd\H{o}s--Beck theorem in the discrete setting. In particular, given we have $\dim (X\setminus P) \geq t$ for all $P \in \mathcal A(n,k)$, it follows that $\dim (X\setminus H) \geq t$. Furthermore, for all $x\in X\setminus H$, $\pi_x(X\cap H)$ is relatively large, as can be seen in the following lemma.

\begin{lemma}[Lemma 25 in \cite{BrightMarshall}] \label{biLipschitzfromkplanes} \label{LemRadProjLocBiLips}
    Let $X\subset \R^n$ and assume there exists a $k$-plane $H \in \mathcal A(n,k)$ so that $\dim (X\cap H) = \dim X$. Then, for all $x\in \R^n \setminus H$, 
    \[
    \dim \pi_x(X) = \dim X.
    \]
\end{lemma}

The proof follows by noting that $\pi_x: H \to \mathbb{S}^{n-1}$ is locally biLipschitz for all $x\in \R^n \setminus H$, and thus $\pi_x$ preserves the Hausdorff dimension of $X\cap H$ when $x\in \R^n \setminus H.$

Thus, by Lemma \ref{biLipschitzfromkplanes}, it follows that for all $x\in X\setminus H$, 
\[
\dim \pi_x(X\cap H) = \dim X.
\]
Thus, $\mathcal L(X)$ is a dual $(\dim X, t)$-Furstenberg set as 
\[
\mathcal L(X) \supset \bigcup_{x\in X\setminus H} \mathcal L_x.
\]
Applying Theorem \ref{BFR-DUALFURST} and noting that $t\leq \dim X$ concludes the proof.
\end{proof}

\begin{remark}
    Note that in the proof of the continuum Erd\H{o}s--Beck theorem it was important that we had the double counting bound for dual $(s,t)$-Furstenberg sets (Theorem \ref{BFR-DUALFURST}) for all ranges of $s,t$---in particular, for $t<s$. For comparison, in the proof of continuum Beck's theorem, we only needed the double counting estimate for dual $(\sigma, \sigma)$-Furstenberg sets. In this way, I find the continuum Erd\H{o}s--Beck theorem to be quite intrinsically motivating for the study of dual $(s,t)$-Furstenberg sets.
\end{remark}

We conclude this section by discussing examples and a conjecture from \cite[Section 1.1]{BrightMarshall}. To do so, we first define the \textit{trapping number} of a set.

\begin{definition}[Trapping Number]
    For $X\subset \R^n$ Borel with $\dim X>0$, we define the \textit{trapping number} of $X$ to be 
    \[
    T(X) := \begin{cases}
        \min\{k >0 : \exists H \in \mathcal A(n,k) \text{ with } \dim (X\setminus H) < \dim X\} \\
        1 \text{ if no such $k>0$ exists}
    \end{cases}.
    \]
\end{definition}

One can think of the trapping number of a set $X$ as a way to recognize if (a large proportion) of $X$ is contained in some lower dimensional $k$-plane. Using this notation, one can restate the continuum version of Beck's theorem in a way that drops the hypothesis on $X.$

\begin{theorem}[cf. Theorem \ref{thm:ch4RenBeck}]
    For $X\subset \R^n$ Borel, 
    \[
    \dim \mathcal L(X) \geq \min\{2\dim X, 2(T(X) - 1)\}.
    \]
\end{theorem}

Note that this is technically proven when $T(X) \geq 2$, though the result is trivial when $T(X) = 1$ so we include it for the sake of a clean statement. Comparing this to Theorem \ref{BM-ERDOSBECK}, we obtain the following:

\begin{theorem}[cf. Theorem \ref{BM-ERDOSBECK}] \label{BM-EB-ForConj}
    Let $X\subset \R^n$ be Borel. If there exists a $k$-plane $H \in \mathcal A(n,k)$ such that $\dim (X\setminus H) < \dim X$, we let $0 < t \leq \dim X$ be such that $\dim (X\setminus P) \geq t$ for all $P \in\mathcal A(n,k)$. Then,
    \[
    \dim \mathcal L(X) \geq \max\{\dim X + t, \min\{2\dim X, 2(T(X)-1)\}\}.
    \]
\end{theorem}

We include the above statement for two reasons. Firstly, we are able to construct examples of sets $X$ such that each term in Theorem \ref{BM-EB-ForConj} dominates. Note that each of our examples satisfy $\dim X < k$, see Remark \ref{remarkonDimX<k}.

\begin{example}[Example 9 in \cite{BrightMarshall}]
    Let $H_1,H_2 \in \mathcal A(n,k)$ be distinct, and suppose $X = X_1 \cup X_2\subset \R^n$ is Borel where $X_i \subset H_i$ for $i = 1,2$. Furthermore, suppose we have
    \[
    k-1 < \dim X_1 < \dim X_2 = \dim X \leq k.
    \]
    Given that 
    \[
    \dim (X\setminus H_2) = \dim X_1 < \dim X,
    \]
    it follows that $T(X) \leq k$. As $k-1 < \dim X_2$, it follows that  
    \[
    \dim (X_2 \setminus P) = \dim (X_2) = \dim X \quad \forall P \in \mathcal A(n,k-1).
    \]
    Thus, $T(X) \geq k$. In particular, it follows that 
    \[
    \min\{2\dim X, 2(T(X) - 1)\} \leq 2k-2.
    \]
    Furthermore, since 
    \[
    \inf_{P\in \mathcal A(n,k)} \dim (X\setminus P) = \dim (X\setminus H_2) = \dim X_1,
    \]
    it follows that we may take $t = \dim X_1$ in the statement of Theorem \ref{BM-EB-ForConj}. Finally, notice that 
    \[
    \dim X + t = \dim X+ \dim X_1 > 2k-2,
    \]
    and so the first term in Theorem \ref{BM-EB-ForConj} dominates.
\end{example}

\begin{example}
    Let $0 < \beta \leq 1$ and let $X\subset \mathbb{S}^k\subset \R^n$ for some $1\leq k \leq n-1$ satisfy $\dim X = k-1 + \beta$. Notice that for all $k$, $\mathbb{S}^k$ is contained in a $(k+1)$-plane (if $k = n-1$, we trivially mean $\mathbb{S}^{n-1} \subset \R^n$). Thus, $T(X) \leq k+1$. Furthermore, for each $k$-plane $P \in \mathcal A(n,k)$, notice that $\dim (\mathbb{S}^k \cap P) \leq k-1$ as each $k$-plane intersecting $\mathbb{S}^k$ in more than a point is homeomorphic to $\mathbb{S}^{k-1}$. Therefore, it follows that $T(X) \geq k+1$, and so 
    \[
    \min\{2\dim X, 2(T(X)-1)\} = 2\dim X.
    \]
    Hence, by Theorem \ref{BM-EB-ForConj},
    \[
    \dim \mathcal L(X) \geq \max\{\dim X+ t, 2\dim X\} = 2\dim X = 2(k-1 + \beta).
    \]
\end{example}

\noindent Notice that when $\beta = 1$, this last example gives an equality as $\dim \mathcal A(k,1) = 2(k-1).$ In general, however, Marshall and I do not expect our bound to be sharp. Instead, we conjecture the following:

\begin{conjecture}[\cite{BrightMarshall}]\label{BM-EBConjecture}
    Given $X\subset \R^n$ Borel, let $T(X) = T$. Suppose that $H \in \mathcal A(n,T)$ is the unique hyperplane such that $\dim (X\setminus H) < \dim X$, and let $X_1 = X\cap H$ and $X_2 = X\setminus H.$ Then,
    \[
    \dim \mathcal L(X) = \max\{\dim(X_1 \times X_2), \dim(X_2 \times X_2), \min\{\dim (X_1 \times X_1), 2(T-1)\}\}.
    \]
\end{conjecture}

The motivation for the above conjecture can be seen in noting 
\[
\mathcal L(X) := \mathcal L(X_1) \cup \mathcal L(X_1,X_2) \cup \mathcal L(X_2).
\]
In particular, firstly notice that one always has 
\[
\dim \mathcal L(X_1) \leq 2(T-1)
\]
as $X_1 \subset H \in \mathcal A(n,T).$ So, if one had (for $i = 1,2$)
\begin{equation}\label{productsetlines}
\dim \mathcal L(X_1,X_2) = \dim(X_1 \times X_2) \quad \text{ and } \quad \dim \mathcal L(X_i) = \dim (X_i \times X_i),
\end{equation}
the conjecture would follow. Furthermore, note that Conjecture \ref{BM-EBConjecture} implies Theorem \ref{BM-EB-ForConj} as 
\[
\dim (X_1 \times X_2) \geq \dim X_1 + \dim X_2 = \dim X + \inf_{H \in \mathcal A(n,T)} \dim (X\setminus H)
\]
and 
\[
\dim(X_1 \times X_1) \geq 2\dim X_1 = 2\dim X.
\]
Here, we use that for all $A,B\subset \R^n$, 
\[
\dim (A\times B) \geq \dim A + \dim B,
\]
see \cite[Theorem 8.10]{Mattila95}. However, it is worth noting that the above lowerbound for product sets cannot be replaced by an equality (see e.g. \cite[Example 14]{BrightMarshall}). For this reason, we expect the equalities in \eqref{productsetlines} to be quite challenging (if true), though of course (as a consequence) quite interesting.

\section{Falconer-type Theorems} \label{SEC:FALCONERTYPE}

Thus far, we have presented most problems from the perspective of projection theory as this has been my primary area of interest and research. However, as is the case for all areas of mathematics, projection theory is not studied in a vacuum. Projection theory both \textit{inspires} and has \textit{been inspired} by work in various parts of harmonic analysis, and in this subsection we will highlight this through our discussion of Falconer-type problems. The goal of this is two-fold: to encourage the reader interested in projection theory to explore more topics in harmonic analysis, and to encourage the reader interested in these aforementioned topics to explore more problems in projection theory.

\subsection{The Erd\H{o}s Distance Problem} \label{erdosprob}

Let us first restate Question \ref{falconertypequestion}. Recall that for $X\subset \R^n$, we let 
\[
\Delta(X) := \{|x-y| : x,y\in X\} \subset \R,
\]
be the distance set of $X$.

\begin{question}\label{ch4:falconertypequestion}
    How few distinct distances can $N$ points in $\R^n$ determine? In other words, given $X\subset \R^n$ finite,  how can we lowerbound $\Delta(X)$ in terms of the size of $|X|$?
\end{question}

Erd\H{o}s studied Question \ref{ch4:falconertypequestion} in 1946 \cite{Erdos1946}. In particular, given $f(N)$ is the minimum number of distances determined by $N$ points in $\R^2$, Erd\H{o}s showed 
\begin{equation}\label{erdosbounds}
(N-3/4)^{1/2} -1/2 \leq f(N) \lesssim \frac{N}{\sqrt{\log N}}.
\end{equation}
The example that Erd\H{o}s constructed to obtain the upperbound in \eqref{erdosbounds} is a square lattice. To see this, if $X = \mathbb{Z}^n \cap [0,p]^n$ then, 
\begin{equation}\label{ErdosSharp}
|\Delta(X)| \leq |\{x_1^2 + \cdots + x_n^2 : x_i \in [0,p]\}| \leq np^2.
\end{equation}
Thus, letting $N= p^n$, it follows that $f(N) \lesssim_n N^{\frac{2}{n}}$ which implies the sharpness of the upperbound in \eqref{erdosbounds} up to the log term in the denominator. Being more careful in \eqref{ErdosSharp} for $p = \sqrt{N}$ and $n = 2$, one can obtain Erd\H{o}s' $\frac{N}{\sqrt{\log N}}$ bound, see \cite[Theorem 1]{Erdos1946}. It is conjectured that this construction is sharp. In particular, the Erd\H{o}s distance problem conjectures that given $X\subset \R^n$ finite, then for all $\epsilon>0$ there exists a constant $C_\epsilon>0$ such that 
\begin{equation}\label{ErdosConjIneq}
|\Delta(X)| \geq C_\epsilon |X|^{\frac{2}{n}-\epsilon}.
\end{equation}

The Erd\H{o}s distance problem has only been resolved in the plane by Guth and Katz \cite{GuthKatz} by the polynomial method, proving that for all $X\subset \R^2$ finite,
\[
|\Delta(X)| \gtrsim \frac{|X|}{\log |X|}.
\]
To obtain their result, Guth and Katz prove and utilize an improvement to the Szemer\'edi--Trotter theorem in $\R^3$ which we now briefly describe.

\begin{remark}
    It may be surprising that we went from the distance problem in the plane to lines in three dimensions. Relating the distances problem to lines in $\R^3$ was an approach developed by Elekes and Sharir in 2011 \cite{ElekesSharir}. For a discussion of this ``Elekes--Sharir'' framework, see e.g. \cite[Section 2]{GuthKatz}.
\end{remark}

Recall that given $\mathcal L \subset \mathcal A(n,1)$ finite, Szemer\'edi--Trotter theorem gives 
\[
|P_r(\mathcal L)|\lesssim \frac{|\mathcal L|^2}{r^3} + \frac{|\mathcal L|}{r}.
\]
This bound is sharp in all dimensions for \textit{arbitrary} sets of lines. For instance, one can take a sharp example (of points and lines) for the Szemer\'edi--Trotter theorem in the plane and consider the same example in $\R^2 \times \{0\}\subset \R^3$. However, it's reasonable to expect that if your points and lines are reasonable nonconcentrated on planes, one can improve the Szemer\'edi--Trotter theorem. This is what Guth and Katz did in $\R^3$ \cite{GuthKatz}. In particular, they showed that if at most $|\mathcal L|^{1/2}$-many lines of $\mathcal L$ lie in any 2-plane, then for $r \geq 3$
\begin{equation} \label{GK1}
|P_r(\mathcal L)|\lesssim \frac{|\mathcal L|^{3/2}}{r^2} + \frac{|\mathcal L|}{r}.
\end{equation}
Moreover, if at most $|\mathcal L|^{1/2}$-many lines of $\mathcal L$ lie in any 2-plane or any degree 2 algebraic surface, then 
\begin{equation} \label{GK2}
|P_2(\mathcal L)| \lesssim |\mathcal L|^{3/2},
\end{equation}
see \cite[Chapters 12 and 13]{GuthPolynomialMethods}.\footnote{It would be interesting to know what these estimates gives for the study of (discrete) exceptional set and Furstenberg set estimates in $\R^3$, which I hope to look into.} It is an interesting question to ask if we can impose similar algebraic conditions on a family of lines $\mathcal L \subset \mathcal A(n,1)$ that give improvements to the Szemer\'edi--Trotter theorem for $n\geq 4$, though we do not discuss this further here. For more on the Erd\H{o}s distance problem see \cite{GaribaldiSengerIosevichBook},\footnote{Based on the topics covered within this thesis the reader may be interested to know that \cite{GaribaldiSengerIosevichBook} also covers analogues of the Erd\H{o}s distance problem over finite fields. One can also use such techniques to study distance \textit{graphs} over finite fields, see \cite{BrightIosevichEtAl,BrightIosevichEtAlGeneralized,IosevichSimplexInterpolation,IosevichRudnev} and references therein for more on this subject.} and for more references and work in higher dimensions see \cite{SolymosiVu}.



\subsection{The Falconer Distance Problem} \label{falconerprob}

We now turn our attention to the Falconer distance problem from geometric measure theory. Let us first restate the Falconer Distance Conjecture:

\begin{conjecture}[Falconer] \label{ch4FalconerConj}
    Given a compact set $A\subset \R^n$,
    \[
    \dim A >\frac{n}{2} \quad \text{ implies that }\quad \mathcal L^1(\Delta(A)) >0.
    \]
\end{conjecture}

Note that the $\frac{n}{2}$ in the lowerbound is continuum analogue to the $\frac{2}{n}$ in \eqref{ErdosConjIneq}. In fact, modifying the Erd\H{o}s' square lattice example, one can show that if $\dim A \leq \frac{n}{2}$, it is possible for $\mathcal L(\Delta(A)) = 0$. For an explicit construction and an introduction to the Fourier analysis related to Falconer's conjecture, we refer the reader to \cite{IosevichWhatIsFalconer}.

In 1985, Falconer \cite{Falconer85} proved that the conclusion of Conjecture \ref{FalconerConj} holds if $\dim A >\frac{n+1}{2}$ using a potential-theoretic argument. Closing the gap between $\frac{n+1}{2}$ and $\frac{n}{2}$ has garnered much attention in the years since Falconer first introduced the problem. In some sense, the conjecture has served as a measuring stick for how much we can say about the rich relationships between measure theory, Fourier analysis and geometry. Early investigations by Mattila \cite{Mattila1985FalconerProb}, Wolff \cite{WolffFalconerProb}, Bourgain \cite{BourgainSubring}, and Erdo\u gan \cite{ErdoganFalconerProb}, have since been built upon using techniques such as Kolmogorov complexity \cite{KolmogorovEx4}, Fourier decoupling \cite{GuthIosevichOuWang}, and, most notably for the purposes of this thesis, projection theory \cite{OrponenRadProjSmoothness,RenRadProj}. We refer the reader to \cite{DuOuRenZhang} for the most recent work and literature review on this problem.

 As we stated in the introduction, the best lowerbounds towards the Falconer distance conjecture currently give the following:

\begin{theorem}[\cite{GuthIosevichOuWang,DuOuRenZhang}]
Let $n\geq 2$ and let
\[
f(n) = \begin{cases}
    \frac{5}{4}, & n=2 \quad \quad \text{(Guth--Iosevich--Ou--Wang \cite{GuthIosevichOuWang})} \\
    \frac{n}{2} + \frac{1}{4} - \frac{1}{8n + 4}, & n\geq 3 \quad \quad \text{(Du--Ou--Ren--Zhang \cite{DuOuRenZhang})}.
\end{cases}
\]
Then, for all compact sets $A\subset \R^n$ with $\dim A > f(n)$, $\mathcal L^1(\Delta(A))>0$.
\end{theorem}
Moreover, both Guth--Iosevich--Ou--Wang and Du--Ou--Ren--Zhang obtain \textit{pinned results}, showing that if $A\subset \R^n$ is compact and $\dim A>f(n)$, there exists an $a\in A$ so that the \textit{pinned distance set} pinned at $a,$
\[
\Delta^a(A) := \{|a-y| : y\in A\},
\]
has positive Lebesgue measure (and thus so too does $\Delta(A)$). 

While there is a lot we could present regarding the Falconer distance problem, I will focus on the discussion on the connections to projection theory. Firstly, we remark that explicit connections between the Furstenberg set problem and the Falconer distance problem were made in the work of Katz--Tao and Bourgain \cite{KatzTaoSubringConjecture, BourgainSubring}---the same work that led to Bourgain's projection theorem (Theorem \ref{BOURGAIN}) that we discussed in Section \ref{sec:ESE}! 

That said, more direct connections have been made between the theory of \textit{radial} projections and the Falconer distance problem, especially utilizing Orponen's radial projection theorem \cite{OrponenRadProjSmoothness}). In 2018, Orponen's result was utilized by Keleti and Shmerkin \cite{KeletiShmerkin} to obtain an exceptional set estimate \textit{for pinned distance sets} in $\R^2$, showing that for $A\subset \R^2$ with $\dim A = s$,
\[
\dim \{y\in \R^2 : \dim \Delta^y(A) < \min\{\frac{2}{3}s, 1\}\}\leq \max\{1,2-s\}.
\]
This directly implies that for $s>1$, there exists a pin $a \in A$ such that 
\[
\dim \Delta^a(A) \geq \min\{\frac{2}{3}s,1\}.
\]
After learning of the work of Keleti--Shmerkin earlier that year, Guth, Iosevich, Ou, and Wang \cite{GuthIosevichOuWang} also made use of Orponen's projection theorem with the theory of decoupling to obtain the bound of $\frac{5}{4}$ in the plane which has yet to be improved upon.

Bounds towards Falconer's distance problem in higher dimensions has also made use of radial projections. In 2021, Du, Iosevich, Ou, Wang, and Zhang \cite{DuIosevichOuWangZhang} applied a generic orthogonal projection onto a $(\frac{n}{2} + 1)$-dimensional subspace with Orponen's radial projection theorem to show that if $n$ is even, 
\[
\dim A > \frac{n}{2} + \frac{1}{4} \implies \mathcal L^1(\Delta(A))>0.\footnote{In fact, again, \cite{DuIosevichOuWangZhang} gives a pinned distance result for the same range of $\dim A.$}
\]

\begin{remark}
As discussed in \cite[Section 5]{DuIosevichOuWangZhang}, in odd dimensions they obtain bounds by projecting onto a generic $\frac{n+1}{2}$-dimensional subspace, but they obtain worse bounds than previous work of Du and Zhang \cite{DuZhang} who showed
\[
\dim A > \frac{n}{2} + \frac{1}{4} + \frac{1}{8n-4} \implies \mathcal L^1(\Delta(A))>0.
\]
\end{remark}

Though a generic orthogonal projection onto a $(\frac{n}{2} + 1)$-dimensional subspace may on the face of it be wasteful, the existence of a $(\frac{n}{2} + 1)$-plane proved tricky to get around for the purposes of applying Orponen's result (at least via the decoupling approach of \cite{GuthIosevichOuWang}\footnote{The aforementioned decoupling approach employs a ``good/bad'' tube decomposition, a technique which Liu made morally similar use of in \cite{Liu2020}.}). See \cite[\textsection Old and  new ideas]{DuOuRenZhang} for further discussion. As such, it is perhaps unsurprising that Ren's 2023 radial projection theorem \cite{RenRadProj} in $\R^n$ (Theorem \ref{renradprojTheorem}) was immediately put to use in the study of the Falconer distance problem in higher dimensions.\footnote{As an interesting historical note, \cite{RenRadProj,DuOuRenZhang,DuOuRenZhangDecoupling} all came out as preprints the same day.} Ren's result allowed him, Du, Ou, and Zhang \cite{DuOuRenZhang} to get around this generic projection argument, obtaining that for $n\geq 3$, if
\[
\dim A > \frac{n}{2} + \frac{1}{4} - \frac{1}{8n + 4},
\]
there exists an $a\in A$ such that $\mathcal L^1(\Delta^a(A))>0$, beating the bound of $\frac{n}{2} + \frac{1}{4}$ for the first time.

\subsection{Pinned Dot Product Set Estimates} \label{sec:BMS}

In 2024, motivated by the progress in projection theory over the past few years, as well as the applications of projection theory towards the Falconer distance problem, Marshall, Senger, and I \cite{BrightMarshall} studied a Falconer-type problem for dot products. 

\begin{notation}
Given $A\subset \R^n$ and $a\in \R^n$, let 
\[
\Pi^a(A) := \{a\cdot y : y\in A\} \subset \R.
\]
\end{notation}

Making use of standard estimates for radial projections and orthogonal projections from Chapter \ref{ch:Topics}, we obtained the following:

\begin{theorem}[Theorem 1.2.(2) in \cite{BrightMarshallSenger}]
    \label{BMS-DotProd}
    Let $A\subset \R^n$ be Borel with $\dim A > \frac{n+1}{2}$ and $n\geq 2$. Then, there exists an $a\in A$ such that 
    \[
    \mathcal L^1(\Pi^a(A)) >0.
    \]
\end{theorem}

To do so, we first need to make some remarks regarding the geometry of dot products. Firstly, we noted that we use dot products in the definition of orthogonal projections onto lines. In particular, given $\theta \in \mathbb{S}^{n-1}$, recall that we let $P_\theta: \R^n \to \R$ be given by
\[
P_\theta(x) = \theta \cdot x.
\]
This begs the question: given a pin $a\in \ell_\theta$, the line through the origin in direction $\theta$, how can we compare $\Pi^a(A)$ and $P_\theta(A)$?

\begin{lemma}[Key Lemma 2.9 in \cite{BrightMarshallSenger}] \label{DotProdKeyLemma}
Let $A\subset \R^n$ Borel and $\theta \in \mathbb{S}^{n-1}$. For all $a\in\theta$ and all $0 < s \leq 1,$
\[
\mathcal H^s(P_\theta(A)) = |a|^{-s} \mathcal H^s(\Pi^a(A)).
\]
\end{lemma}

\begin{proof}
    This follows by noting 
    \[
    \Pi^a(A) := \{a\cdot y : y\in A\} = |a| \left\{\frac{a}{|a|}\cdot y : y\in A\right\}.
    \]
    Thus, given $a\in \ell_\theta$, it follows that $\frac{a}{|a|} \in \{\theta, - \theta\}$. Suppose without loss of generality that $\frac{a}{|a|} = \theta$. Hence, 
    \[
    \Pi^a(A) := |a|\left\{\theta \cdot y : y \in A, \theta = \frac{a}{|a|} \text{ for some } a\in A\right\} = |a| P_\theta(A).
    \]
    Therefore, given 
    \[
    \mathcal H^s(\lambda X) = |\lambda|^s \mathcal H^s(X)
    \]
    for all $\lambda \in \R$ and $X\subset \R^n$ Borel, the claim follows.
\end{proof}

In particular, it follows that
\[
\mathcal H^1(P_\theta(A)) >0 \quad \iff\quad \mathcal H^1(\Pi^a(A))>0 \,\, \forall a\in \ell_\theta.
\]
In fact, Lemma \ref{DotProdKeyLemma} implies that for all $\theta \in \mathbb{S}^{n-1}$ and all $0 \leq u \leq 1$,
\[
\dim P_\theta(A)\geq u \quad \iff \quad \dim \Pi^a(A) \geq u \,\, \forall a\in \ell_\theta.
\]
We will use this in the proof of Theorem \ref{BMS-DotProd2}. We now prove Theorem \ref{BMS-DotProd}.

\begin{proof}[Proof of Theorem \ref{BMS-DotProd}]
    Firstly, notice that by Lemma \ref{DotProdKeyLemma}, it suffices to show that there exists a $\theta \in \mathbb{S}^{n-1}$ such that $\mathcal L^1(P_\theta(A))>0$ and $$(A \setminus \{0\})\cap \ell_\theta \neq \emptyset.$$
    To do so, we first note the following exceptional set estimate:

    \begin{theorem}[\cite{Falconer82}] \label{FalconerNull}
        Let $A\subset \R^n$ be Borel and 
        \[
        E_+(A) := \{\theta \in \mathcal G(n,1) : \mathcal L^1(P_\theta(A)) = 0\}.
        \]
        If $\dim A >1$, then 
        \[
        \dim E_+(A) \leq n-\dim A.
        \]
    \end{theorem}

Secondly, recall Lemma \ref{radprojtrivialbound}: for all $x\in \R^n$,
\[
\dim \pi_x(A\setminus \{x\}) \geq \dim A-1.
\]
Hence, the above estimate holds for $x = 0$, and thus if $\dim A>\frac{n+1}{2}>1$, 
\begin{equation}\label{InFactManyPins}
\dim \pi_0(A\setminus \{0\}) := \{\frac{a}{|a|} : a\in A\setminus \{0\}\} > \dim E_{+}(A).
\end{equation}
Thus, there exists a $\theta \in \pi_0(A\setminus \{0\}) \setminus E_+(A)$ such that $\theta = \frac{a}{|a|}$ for some $a\in A\setminus \{0\}$. Since $\theta \notin E_+(A)$, by Lemma \ref{DotProdKeyLemma} it follows that 
\[
\mathcal L^1(P_\theta(A))>0 \implies \mathcal L^1(\Pi^{a}(A))>0. 
\]
as desired.    
\end{proof}

We conclude the section with a few remarks pertaining to the content of \cite{BrightMarshallSenger}. Firstly, notice that by \eqref{InFactManyPins}, we in fact have many choices of a pin $a\in A$ such that the conclusion of Theorem \ref{BMS-DotProd} holds.

\begin{corollary}
Let $A\subset \R^n$ be Borel with $\dim A > \frac{n+1}{2}$ and $n\geq 2$. Then,
\[
\dim \{a\in A : \mathcal L^1(\Pi^a(A)) >0\} = \dim A.
\]
\end{corollary}

\begin{proof}
    For ease of (and by abuse of) notation, we assume $0 \not\in A$.
    
    Firstly, notice that 
    \[
    \{a\in A : \mathcal L^1(\Pi^a(A)) >0\} \subset \pi_0^{-1}(\pi_0(A) \setminus E_+(A)) := A_\ast.
    \]
    Suppose for the sake of contradiction that $\dim A_\ast < \dim A$. Then, it follows that $\dim (A\setminus A_\ast) = \dim A$ and 
    \[
    \pi_0(A\setminus A_\ast) \subset E_{+}(A).
    \]
    Hence, by Lemma \ref{radprojtrivialbound},
    \[
    \dim E_+(A) \geq \dim \pi_0(A\setminus A_\ast) \geq \dim (A\setminus A_\ast) - 1.
    \]
    Therefore, given $\dim A >\frac{n+1}{2}$ we have
    \[
    \dim (A\setminus A_\ast) = \dim A \leq \dim E_+(A) + 1 \leq \frac{n+1}{2}.
    \]
    This contradicts the assumption that $\dim A >\frac{n+1}{2}$, concluding the proof.
\end{proof}

A key ingredient towards proving Theorem \ref{BMS-DotProd} was the use of Falconer's exceptional set estimate (Theorem \ref{FalconerNull}) for $\theta \in \mathbb{S}^{n-1}$ such that $P_\theta(A)$ has null $\mathcal L^1$-measure. Such estimates can be thought of as a \textit{visibility} exceptional set estimate for orthogonal projections. To this end, applying Theorem \ref{FALCONER} in place of Theorem \ref{FalconerNull} gives the following:

\begin{theorem}[Theorem 1.2.(1) in \cite{BrightMarshallSenger}] \label{BMS-DotProd1-Dim}
    Let $n\geq 2$ and $0 < u \leq 1$. Given $A\subset \R^n$ Borel with $\dim A > \frac{n+u}{2}$, there exists an $a\in A$ such that 
    \[
    \dim(\Pi^a(A)) \geq u.
    \]
    Moreover, 
    \[
    \dim \{a \in A :\dim(\Pi^a(A)) \geq u \} = \dim A.
    \]
\end{theorem}

In fact, our proof methodology highlights a general heuristic that better projection estimates imply better dot product results. One way to see this heuristic is through applying better orthogonal projection estimates. For instance, applying the sharp exceptional set estimates of Cholak et al. \cite{CholakEtAl} (Theorem \ref{cholak}) gives a slight improvement to \cite[Theorem 1.2.1)]{BrightMarshallSenger}, and using a restricted projection\footnote{Though we could not discuss it further here, the restricted projections problem is quite an interesting extension of exceptional set estimates for orthogonal projections, see for instance \cite{FasslerOrponen,GanGuoWang,GanGuoGuthHarrisMaldagueWang,PramanikYangZahl,AlgomShmerkin}. It may be worth noting that \cite{GanGuoWang} makes use of the high-low method and decoupling for the cone, and \cite{GanGuoWang} uses decoupling for the moment curve.} result of Zahl \cite{ZahlRestrictedProj} gives \cite[Proposition 1.10]{BrightMarshallSenger}. Furthermore, exceptional set estimates have been studied for various notions of ``size'' besides small Hausdorff dimension and null Lebesgue measure, see e.g. \cite{Wu25}. Such estimates yield more pinned dot product set estimates for these different notions of size, see e.g. \cite[Theorem 1.2.3]{BrightMarshallSenger}.

Secondly, notice that in the proof of Theorem \ref{BMS-DotProd}, we apply the trivial radial projection estimate (Lemma \ref{radprojtrivialbound}). If we instead applied stronger radial projection results, such as that of Ren \cite{RenRadProj}, we should expect stronger pinned dot product set estimates. For instance, notice the following corollary of Ren's radial projection theorem.

\begin{corollary}[Corollary of Theorem \ref{renradprojTheorem}] \label{corollaryRenRadExceptional}
    Let $X,Y\subset \R^n$ be Borel, and suppose that for all $P\in \mathcal A(n,k)$, $\dim (X\setminus P) = \dim X$. Furthermore, let 
    \[
    E_\epsilon := \{x \in X : \dim \pi_x(Y\setminus \{x\}) \leq \min\{\dim X, \dim Y, k\}- \epsilon\}.
    \]
    Then, for all $\epsilon>0$, $\dim X = \dim (X\setminus E_\epsilon)$.
\end{corollary}

In the plane, Corollary \ref{corollaryRenRadExceptional} was noted by \cite[Corollary 1.2]{OSW}, and the proof follows nearly identically.

\begin{proof}[Proof of Corollary \ref{corollaryRenRadExceptional} from Theorem \ref{renradprojTheorem}]
    Suppose for the sake of contradiction that there exists an $\epsilon>0$ such that $\dim (X\setminus E_\epsilon) < \dim X$. Hence, $\dim E_\epsilon = \dim X$. Furthermore, it must be the case that $E_\epsilon$ is contained in a $k$-plane $H \in \mathcal A(n,k)$. If this were not the case, then we can apply Theorem \ref{renradprojTheorem} to obtain 
    \[
    \sup_{x\in E_\epsilon} \dim \pi_x(Y\setminus \{x\}) \geq \min\{\dim E_\epsilon, \dim Y, k\} = \min\{\dim X, \dim Y, k\} 
    \]
    which contradicts the definition of $E_\epsilon$. However, if $E_\epsilon$ is contained in a $k$-plane $H\in \mathcal A(n,k)$, then a contradiction arises in noting 
    \[
    \dim X > \dim (X\setminus E_\epsilon) \geq \dim (X\setminus H) = \dim X
    \]
    where the last equality follows by our hypotheses on $X$.
\end{proof}

Using Corollary \ref{corollaryRenRadExceptional}, Marshall, Senger, and I were able to obtain Falconer-type estimate for \textit{translated} pinned dot product sets.

\begin{notation}
    Given $A\subset \R^n$ and $a,x\in \R^n$, let 
\[
\Pi^a_x(A) := \{(a-x)\cdot y : y\in A\} \subset \R.\footnote{Note that we use $x$ in the subscript to align with the $x$ in $\pi_x$.}
\]
\end{notation}

\begin{theorem}[Theorem 1.6.(2) in \cite{BrightMarshall}] \label{BMS-DotProd2}
    Let $n\geq 2$ and let $\lceil \frac{n}{2}\rceil \leq k \leq n-1$ be a parameter to be chosen below. Suppose that $A,X\subset \R^n$ are Borel with $\dim X \geq \dim A = s >\frac{n}{2}$ and $\dim (X\setminus P) = \dim X$ for all $P \in \mathcal A(n,k).$ Then, there exists an $a\in A$ and $x\in X$  such that 
    \[
    \mathcal L^1(\Pi^a_x(A)) >0.
    \]
    Furthermore, there $X'\subset X$ with $\dim X' = \dim X$ such that for all $x\in X'$, 
    \[
    \dim \{a\in A : \mathcal L^1(\Pi^a_x(A)) >0\} = \dim A.
    \]
\end{theorem}

\begin{remark}
The proof follows similarly to Theorem \ref{BMS-DotProd}, applying Corollary \ref{corollaryRenRadExceptional} in place of the trivial radial projection estimate and noting 
\[
\mathcal H^s(P_{\pi_x(a)}(A)) = |a-x|^{-s}\mathcal H^s(\Pi_{\pi_x(a)}(A))
\]
where $\pi_x(a) = \frac{a-x}{|a-x|}$ is the radial projection of $a$ onto $x$.
\end{remark}

Notice that by allowing for translations of $A$ by points $X$, we have improved the bound $\dim A > \frac{n+1}{2}$ in Theorem \ref{BMS-DotProd} to $\frac{n}{2}$ in Theorem \ref{BMS-DotProd2}. This begs the question: can we lower the bound on $\dim A$ in Theorem \ref{BMS-DotProd} to $\frac{n}{2}$ instead of $\frac{n+1}{2}$? We conjecture that the answer, at least in the plane, is \textit{yes}.

\begin{conjecture}[\cite{BrightMarshallSenger}] \label{BMS-DotProdConj}
    Let $A\subset \R^2$ be Borel with $\dim A > 1$. Then, there exists an $a\in A$ such that 
    \[
    \mathcal L^1(\Pi^a(A)) >0.
    \]
\end{conjecture}

\begin{remark}
    Note by Lemma \ref{DotProdKeyLemma} that Conjecture \ref{BMS-DotProdConj} is equivalent to the statement: for all $A\subset \R^2$ Borel with $\dim A>1$,
    \begin{equation}\label{bmsconjeq}
    \pi_0(A \setminus \{0\}) \setminus E_+(A) \neq \emptyset.
    \end{equation}
\end{remark}

Firstly, we can notice that the theorem fails if $\dim A \leq 1$, Indeed, let $\mathcal C \subset \R$ be a $1$-dimensional set with null Lebesgue measure (see Section \ref{ch:background} for a construction using countable stability of Hausdorff dimension). Then, letting $A = \mathcal C\times \{0\}$, we have $\dim A = \dim \mathcal C = 1$, and for all $(a,0) \in A$,
\[
\mathcal L^1(\Pi^a(A)) = \mathcal L^1(\{ac : c\in \mathcal C\}) = |a| \mathcal L^1(\mathcal C) = 0.
\]
One may hope to apply the same projection theoretic methodology of that in the proof of Theorem \ref{BMS-DotProd} to prove Conjecture \ref{BMS-DotProdConj}.
When $\dim A > \frac{3}{2}$, we are able to prove \eqref{bmsconjeq} by noting 
\begin{equation}\label{sharpwhy}
\dim \pi_0(A\setminus \{0\}) \geq  \dim A - 1 > 2-\dim A \geq \dim E_+(A).
\end{equation}
One might hope to improve the above estimates in trying to prove Conjecture \ref{BMS-DotProdConj}. However, both the first and last inequalities are sharp for all $n\geq 2$! 

The sharpness of the first can be seen by taking $A$ to be a radial lattice of the form (in polar coordinates) $[0,1]\times \Theta$ for $\Theta \subset \mathbb{S}^{n-1}$; the sharpness of the last can be seen by modifying the square lattice example, see \cite[Example 5.13]{Mattila15}. This however does not disprove Conjecture \ref{BMS-DotProdConj}. In fact, these examples if anything may bolster one's confidence in the plausibility and difficulty of this conjecture! Indeed, the radial example making the first inequality in \eqref{sharpwhy} has $E_+(A) = \emptyset$, and the square lattice example had full dimensional radial projection onto the origin. In this sense, the examples that make the first inequality in \eqref{sharpwhy} sharp seem to have much smaller exceptional sets, and the examples with large exceptional sets seem to have large radial projection onto the origin! In this way, Conjecture \ref{BMS-DotProdConj} seeks to understand whether or not this dichotomy is always true.

As a first step towards understanding this dichotomy (or lack thereof), it would be interesting to have more classes of sets $A$ that make the first or last inequality in \eqref{sharpwhy} sharp. As a final step, it would be interesting to know whether or not the following stronger conjecture holds.

\begin{conjecture}[\cite{BrightMarshallSenger}]
    Let $A\subset \R^2$ be Borel with $\dim A>1$, Then, 
    \[
    \dim \pi_0(A\setminus \{0\}) > \dim E_+(A).
    \]
\end{conjecture}

\newpage

\bibliographystyle{plain}
\bibliography{references}


\end{document}